\documentclass[reqno,centertags,12pt]{amsart}
\usepackage{amsmath,amsthm,amsfonts,amssymb,mathrsfs}
\usepackage{enumerate}
\usepackage{color}

\newtheorem{definition}{Definition}[section]
\newtheorem{theorem}{Theorem}[section]
\newtheorem{lemma}[theorem]{Lemma}
\newtheorem{corollary}[theorem]{Corollary}
\newtheorem{proposition}[theorem]{Proposition}

\theoremstyle{remark}
\newtheorem{remark}[theorem]{Remark}

\numberwithin{equation}{section}
\numberwithin{theorem}{section}
\numberwithin{figure}{section}

\newcommand{\norm}[1]{\left\|#1\right\|}
\newcommand{\intd }{\,{\rm d}}
\DeclareMathOperator{\dtd}{\frac{d}{d{\emph t}}}

\newcommand{\ZZ}{\mathbb{Z}}

\newcommand{\NN}{\mathbb{N}}

\newcommand{\RR}{\mathbb{R}}

\DeclareMathOperator{\Div}{div}
\allowdisplaybreaks
\begin{document}
\title[Global well-posedness for the 2D Maxwell-Navier-Stokes equations]{Global well-posedness for the two-dimensional Maxwell-Navier-Stokes equations}

\author[C. Miao]{Changxing Miao}
\address[C. Miao]{Institute of Applied Physics and Computational Mathematics,
        P.O. Box 8009, Beijing 100088, P.R. China.}

\email{miao\_{}changxing@iapcm.ac.cn}

\author[X. Zheng]{Xiaoxin Zheng}

\address[X. Zheng]{School of Mathematics and Systems Science, Beihang University, Beijing 100191, P.R. China}

\email{xiaoxinzheng@buaa.edu.cn}

\date{\today}
\subjclass[2010]{35Q30; 35B40;  76D05. }
\keywords{Maxwell-Navier--Stokes equations;  weak solutions; uniqueness; Localization.}

\begin{abstract}
In this paper, we investigate Cauchy problem of the two-dimensional full Maxwell-Navier-Stokes system, and prove the global-in-time existence and uniqueness of solution  in the borderline space which is very close to $L^2$-energy space by developing the new estimate of $\sup_{j\in\ZZ}2^{2j}\int_0^t\sum_{k\in\mathbb{Z}^2}\big\|\sqrt{\phi_{i,k}}u(\tau)\big\|^2_{L^2(\RR^2)}\intd\tau<\infty$.
This  solves  the open problem in the framework of borderline space  purposed by Masmoudi in \cite{Masmoudi-10}.

\end{abstract}

\maketitle
\section{Introduction}\label{Intr}
\setcounter{section}{1}\setcounter{equation}{0}
We consider a coupled system of equations consisting of the Navier-Stokes equations of fluid dynamics and Maxwell¡¯s equations of electromagnetism. The coupling comes from the Lorentz force in the fluid equation and the electric current in the Maxwell equations which takes the following form
 \begin{equation}\label{eq.MNS}
 \left\{\begin{array}{ll}
 u_t+(u\cdot\nabla)u-\nu\Delta u+\nabla\pi= j\times B\quad (t,x)\in\RR^+\times\RR^2,\\
 E_t-\mathrm{curl}\,B=-j,\\
 B_t+\mathrm{curl}\,E=0,\\
 \Div u=\Div B=0,\\
 j=\sigma(E+u\times B).
 \end{array}\right.
 \end{equation}
System \eqref{eq.MNS} should be supplemented with an initial condition
\[u(0,x)=u_0(x),\quad B(0,x)=B_0(x),\quad E(0,x)=E_0(x),\]
where $u_0(x)$ and $B_0(x)$ satisfy $\Div u_0=\Div B_0=0$. Here, $u=(u^1,u^2,u^3)(t,x_1,x_2)$ stands for velocity of the fluid.  $E=(E^1,E^2,E^3)(t,x_1,x_2)$ and $B=(B^1,B^2,B^3)(t,x_1,x_2)$  electric field and magnetic field, respectively.
The scalar function $\pi$ is the pressure which can be recovered  at least formally by $u$ and $j\times B$ via Calder\'on-Zygmund operators, that is,
\[\pi=-\mathbb{P}\big((u\cdot\nabla)u-(j\times B)\big),\]
where $\mathbb{P}$ is the Leray projector. $j$ is the electric current which is given by Ohm's law and $j\times B$ is the Lorentz force. In addition, $\nu$ is the viscosity and $\sigma$ is the electric conductivity. For simplicity, we will take $\nu = \sigma = 1 $ in the following parts.

This system has strong physical background, the reader can refer to \cite{BIS,DAV} for more physical introduction concerning on magnetohydrodynamics. By the divergence-free condition and the following vanishing condition that
\begin{equation}\label{eq.dis}
\int_{\RR^2}j\cdot (u\times B)\intd x+\int_{\RR^2}(j\times B)\cdot u\intd x=0,
\end{equation}
it is easy to show that for a smooth solution,
\begin{equation}\label{est-energy}
\big\|\big(u,E,B\big)(t)\big\|^2_{L^2(\RR^2)}+2\int_0^t\left(\|\nabla u(\tau)\|^2_{L^2(\RR^2)}+\|j(\tau)\|^2_{L^2(\RR^2)}\right)\intd\tau
=\|(u_0,E_0,B_0)\|^2_{L^2(\RR^2)}.
 \end{equation}
 This natural energy equality is very similar to that for the bi-dimensional Navier-Stokes equations. As we know, with the help of the energy estimate, Leray \cite{Leray} showed that the bi-dimensional Navier-Stokes system has a unique global-in-time weak solution. Inspired by this Leray theory, a natural question is that does system \eqref{eq.MNS} exist a unique global-in-time weak solution enjoying the energy estimate \eqref{est-energy}. However, due to the hyperbolic nature of the Maxwell equation, it is difficult to get compactness of $B$ and hence passing to the limit in the product $j\times B$ seems to be a challenge problem. This leads to that it is difficult to get  global-in-time existence of the $L^2$ energy weak solution. The essential reason is the lack of the control of $\int_0^\infty\|u(t)\|_{L^\infty}^2\intd t$. In fact, the $ L^2$ energy estimate just provides us the control of $\int_0^\infty\|u(t)\|_{\mathrm{BMO}}^2\intd t$ which is bounded by $\int_0^\infty\|\nabla u(t)\|_{L^2}^2\intd t.$¡¡ Hence, proving global existence of weak solutions to system \eqref{eq.MNS} in
the energy space $(L^2)^3$ or the borderline space $L^2\times L^2_{\rm log}\times L^2_{\rm log}$ seems is an open problem.

From above, we easily find that $\int_0^\infty\|u(t)\|_{L^\infty}^2\intd t$ is a very important quantity to show the global regularity of weak solutions. Unfortunately, we don't have the control of this quantity. But, it is very close to $\int_0^\infty\|u(t)\|_{\rm BMO}^2\intd t$. With this observation, Masmoudi \cite{Masmoudi-10} proved the existence and uniqueness of global strong solutions in the $H^s(\RR^2)$ framework to problem \eqref{eq.MNS} with $s>0$. His proof highly relies on a time-space logarithmic inequality that enabled him to upper estimate the $L^\infty$-norm of the velocity field by
the energy norm and higher Sobolev norms. Another line of research was pursued by Ibrahim and Keraani~\cite{Sahbi-SIAM}, they proved a local-in-time strong solution in the borderline space $\dot{B}^0_{2,1}\times(L^2_{\rm log})^2$ by using parabolic regularization arguments giving control of the $L^\infty$ norm of the velocity field of the solution. Based on this, a global-in-time result for small initial data and a local-in-time result for the large initial data in the borderline space $L^2\times(L^2_{\rm log})^2$ were obtained in \cite{PSN} by establishing an $L^2_tL^\infty$ estimate on the velocity field.  Very recently, Ibrahim, Masmoudi and  Lemari\'e-Rieusset in \cite{IML} proved the existence
of time-periodic small solutions and  their asymptotic
stability for the 3D Navier-Stokes-Maxwell problem in the presence of external time-periodic forces.

In our paper, our target is to show the global-in-time existence and uniqueness of solution for the large initial data in the borderline space $L^2\times(L^2_{\rm log})^2$. Therefore, the main task is to bridge the gap between $\int_0^\infty\|u(t)\|_{\rm BMO}^2\intd t$ and $\int_0^\infty\|u(t)\|_{L^\infty}^2\intd t$. But the previous methods for problem \eqref{eq.MNS} including the argument used in \cite{Masmoudi-10} do not work. This requires us to develop a new method to overcome this difficulty. Now, we take the linear heat equation as an example to illustrate our main idea. Our strategy is to use micro-analysis in physical space to bootstrap the regularity of solution. Let $f$ be the smooth solution of the linear heat equation
$\partial_tf-\Delta f=0.$ Multiplying  this linear heat equation by $\varphi_{j,k}f$, we see that
\[\frac12\partial_t\big(\varphi_{j,k}f^2\big)-\varphi_{j,k}f\Delta f=0,\]
where $\varphi_{j,k}$ is the solution of the eigenvalue problem, see Lemma \ref{Lem-En} for details. Integrating the above equality in space variable over $\RR^2$ and using Corollary \ref{coro-test-prop}, one has
\begin{equation*}
\begin{split}
&\big\|\sqrt{\varphi_{j,k}}f\big\|^2_{L^\infty_tL^2(\RR^2)}+\lambda_1 2^{2j}\big\|\sqrt{\varphi_{j,k}}f\big\|^2_{L^2_tL^2(\RR^2)}+2\big\|\sqrt{\varphi_{j,k}}\nabla f\big\|^2_{L^2_tL^2(\RR^2)}\\
=&\big\|\sqrt{\varphi_{j,k}}f_0\big\|^2_{L^2(\RR^2)}-\int_0^t\int_{\partial B_{2^{-j}}(k)}f^{2}\nabla\varphi_{j,k}\cdot n\intd S\intd\tau.
\end{split}
\end{equation*}
By the trace theorem and the H\"older inequality, we finally get that
\begin{align*}
&\|f\|_{L^\infty_tL^2(\RR^2)}^2+\int_0^t\|\nabla f(\tau)\|_{L^2(\RR^2)}^2\intd\tau+ \lambda_1\sup_{j\in\ZZ} 2^{2j}\sum_{k\in\ZZ^2}\|\sqrt{\varphi_{j,k}}f\|_{L^2_tL^2(\RR^2)}^2\\
\leq &C\|f_0\|_{L^2(\RR^2)}^2+C\int_0^t\|f(\tau)\|_{H^1(\RR^2)}^2\intd\tau.
\end{align*}
This together with the following natural $L^2$-energy estimate
\[\|f(t)\|^{2}_{L^2(\RR^2)}+\int_0^t\|\nabla f(\tau)\|_{L^2(\RR^2)}^2\intd\tau\leq\|f_0\|_{L^2(\RR^2)}^2\]
allows us to infer that
\[\sup_{j\in\ZZ} 2^{2j}\int_0^t\sum_{k\in\ZZ^2}\big\|\sqrt{\phi_{j,k}}f(\tau)\big\|^{2}_{L^2(\RR^2)} \intd\tau<\infty,\]
which plays the key role in our proof. In virtue of the Morrey-Campanato type characterization of $L^\infty(\RR^2)$, we know that this quantity is very close to $L^2_tL^\infty$. Thus,
with this global-in-time bound, we further establish the global-in-time bound of solution in in the borderline space $L^2\times(L^2_{\rm log})^2$ in terms of techniques in harmonic analysis. As a result, we eventually get the control of $\int_0^t\|u(t')\|_{L^\infty}^2\intd t'$.
This enables us to remove the small assumption for initial data in  \cite{PSN}.

Now we state our main result as follows:
\begin{theorem}\label{THM-2}
Let $u_0\in L^2(\RR^2)$ and $(E_0,B_0)\in (L^2_{\rm log}(\RR^2))^2.$ Then system \eqref{eq.MNS} admits a unique global-in-time solution $(u(t),E(t),B(t))\in C_{\mathrm b}(\RR^+;\,L^2(\RR^2))\times\big(C(\RR^+;\,L^2_{\rm log}(\RR^2))\big)^2$ such that \eqref{est-energy} and
\begin{equation*}
\begin{split}
&\big\|u(t)\big\|_{L^2(\RR^2)}^2+\big\|(E,B)\big\|_{\widetilde{L}^\infty_tL_{\rm log}^2(\RR^2)}^2+ \int_0^t\big\|j(\tau)\big\|_{L^2_{\rm log}(\RR^2)}^2\intd\tau+\int_0^t\|u(\tau)\|_{L^\infty(\RR^2)}^2\intd\tau\\
\leq& C\left(t,\|u_0\|_{L^2(\RR^2)},\|(E_0,B_0)\|_{L^2_{\rm log}(\RR^2)}\right).
\end{split}
\end{equation*}
\end{theorem}
\begin{remark}
Compared with result in \cite{Sahbi-SIAM,PSN},  we extend the local-in-time solution established in \cite{Sahbi-SIAM} to the global-in-time solution in theorem \ref{THM-2}, while we removes the small assumption for initial data in \cite{PSN}.
\end{remark}
\begin{remark}
Let us point out that in our paper, we develop the following new estimate
\[\sup_{j\in\ZZ}2^{2j}\int_0^t\sum_{k\in\mathbb{Z}^2}\big\|\sqrt{\phi_{i,k}}u(\tau)\big\|^2_{L^2(\RR^2)}\intd\tau<\infty.\]
 In terms of the Morrey-Campanato type characterization of $L^\infty(\RR^2)$, we easily find that it is very close to  $L^2_tL^\infty$-estimate for $u$. In the other words, this type space can be viewed as the Chemin-Lerner space in the framework of localization.
\end{remark}

\textbf{Acknowledgments:}
This work was supported in part by the National Natural Science Foundation of China.   C. Miao was also supported  by Beijing Center for Mathematics and Information
Interdisciplinary Sciences.

\section{Preliminary}\label{Pre}
\setcounter{section}{2}\setcounter{equation}{0}

\subsection{Littlewood-Paley Theory and the functional spaces}
In this subsection, we first review the so-called Littlewood-Paley decomposition described, e.g., in \cite{BCD11,Can1,Can2,MWZ2012}. Next, we introduce some useful functional spaces such as Morrey-Campanato space and its properties.
Let $(\chi,\psi)$ be a couple of smooth functions with  values in $[0,1]$
such that $\chi$ is supported in the ball $\big\{\xi\in\mathbb{R}^{d}\big||\xi|\leq\frac{4}{3}\big\}$,
$\varphi$ is supported in the ring $\big\{\xi\in\mathbb{R}^{d}\,\big|\,\frac{3}{4}\leq|\xi|\leq\frac{8}{3}\big\}$ and
\begin{equation*}
    \chi(\xi)+\sum_{j\in \mathbb{N}}\psi(2^{-j}\xi)=1\quad {\rm for \ each\ }\xi\in \mathbb{R}^{d}.
\end{equation*}
For any $u\in \mathcal{S}'(\mathbb{R}^{d})$, one can define the
dyadic blocks as
\begin{equation*}
 \Delta_{-1}u=\chi(D)u\quad\text{and}\quad   {\Delta}_{j}u:=\Psi(2^{-j}D)u\quad {\rm for\ each\ }j\in\mathbb{N}.
\end{equation*}
We also define the following low-frequency cut-off:
\begin{equation*}
    {S}_{j}u:=\chi(2^{-j}D)u.
\end{equation*}
According to the support in frequency space, it is easy to verify that
\begin{equation*}
    u=\sum_{j\geq-1}{\Delta}_{j}u,\quad \text{in}\quad\mathcal{S}'(\mathbb{R}^{d}),
\end{equation*}
and this is called the \emph{inhomogeneous Littlewood-Paley decomposition}.
It has nice properties of quasi-orthogonality:
\begin{equation*}
    {\Delta}_{j}{\Delta}_{j'}u\equiv 0\quad \text{if}\quad |j-j'|\geq 2.
\end{equation*}
\begin{equation*}
{\Delta}_{j}({S}_{j'-1}u{\Delta}_{j'}v)\equiv0\quad \text{if}\quad |j-j'|\geq5.
\end{equation*}
We shall also use the \emph{homogeneous Littlewood-Paley} operators as follows:
\begin{equation*}
    \dot{S}_{j}u:=\chi(2^{-j}D)u \quad\text{and}\quad\dot{\Delta}_{j}u:=\Phi(2^{-j}D)u\quad\text{for each}\,\,j\in\mathbb{Z},
\end{equation*}
which enjoy the properties of quasi-orthogonality as above for inhomogeneous operator.
\begin{definition}
Let $\mathcal{S}_{h}'(\RR^d)$ be the space of tempered distributions $u$ such that
\begin{equation*}
    \lim_{q\rightarrow-\infty}\dot{S}_{j}u=0,\quad \text{in} \quad \mathcal{S}'(\RR^d).
\end{equation*}
\end{definition}
\begin{definition} \label{def-2.5}
For any  $u, v\in\mathcal {S}_{h}^{\prime}(\RR^d)$,  the product $u
v$ has the homogeneous Bony decomposition:
$$u v=\dot{T}_{u} v+\dot{T}_{v}u+\dot{R}(u,v),
$$
where the paraproduct term
$$\dot{T}_{u}v=\sum_{j\leq{k-2}}\dot{\Delta}_{j}u\dot{\Delta}_{k}v=\sum_{j}{\dot{S}_{j-1}u}{\dot{\Delta}_{j}v},$$
and the remainder term
$$\dot{R}(u,v)=\sum_{j}\dot{\Delta}_{j}u\widetilde{\dot{\Delta}}_{j}v, \quad \widetilde{\dot{\Delta}}_{j}:=\sum_{k=-1}^1\dot{\Delta}_{j-k}.$$
\end{definition}
In the similar way, we can define the inhomogeneous Bony decomposition:
$$u v= T _{u} v+ T _{v}u+ R (u,v),
$$
one can refer to \cite{BCD11} for the details.
Now  we introduce the Bernstein lemma which will be useful
throughout this paper.
\begin{lemma}\label{bernstein}
Let $1\leq a\leq b\leq\infty$ and  $f\in L^a(\RR^d)$. Then there exists a positive constant $C$ such that for $q,\,k\in\NN$,
\begin{equation*}
\sup_{|\alpha|=k}\|\partial ^{\alpha}\dot{S}_{q}f\|_{L^b(\RR^d)} \leq  C^k\,2^{q\left(k+d(\frac{1}{a}-\frac{1}{b})\right)}\|\dot{S}_{q}f\|_{L^a(\RR^d)},
\end{equation*}
\begin{equation*}
C^{-k}2^{qk}\|{\dot{\Delta}}_{q}f\|_{L^a(\RR^d)} \leq \sup_{|\alpha|=k}\|\partial ^{\alpha}{\dot{\Delta}}_{q}f\|_{L^a(\RR^d)}\,\leq\,C^k2^{qk}\|{\dot{\Delta}}_{q}f\|_{L^a(\RR^d)}.
\end{equation*}
\end{lemma}
\begin{definition}\label{def2.2}
Let $s\in \mathbb{R}$, $(p,q)\in [1,\infty]^{2}$ and $u\in \mathcal{S}'(\mathbb{R}^{d})$. Then we define the \emph{inhomogeneous Besov spaces} as
\begin{equation*}
   {B}^{s}_{p,q}(\mathbb{R}^{d}):=\big\{u\in\mathcal{S}'(\mathbb{R}^{d})\big|\norm{u}_{{B}^{s}_{p,q}(\mathbb{R}^{d})}<\infty\big\},
\end{equation*}
where,
\begin{equation*}
    \norm{u}_{{B}^{s}_{p,q}(\mathbb{R}^{d})}:=\begin{cases}
    \Big(\sum_{j\geq-1}2^{jsq}\norm{{\Delta}_{j}u}_{L^{p}(\mathbb{R}^{d})}^{q}\Big)^{\frac{1}{q}}
    \quad&\text{if}\quad q<\infty,\\
   \sup_{j\geq-1}2^{js}\norm{{\Delta}_{j}u}_{L^{p}(\mathbb{R}^{d})}\quad&\text{if}\quad q=\infty.
    \end{cases}
\end{equation*}
\end{definition}
\begin{definition}\label{def.FB}
Let $s\in \mathbb{R}$, $(p,q)\in [1,\infty]^{2}$ and $u\in \mathcal{S}'(\mathbb{R}^{d})$. Then we define the \emph{inhomogeneous Fourier-Herz spaces} as
\begin{equation*}
   {FB}^{s}_{p,q}(\mathbb{R}^{d}):=\left\{u\in\mathcal{S}'(\mathbb{R}^{d})\big|\norm{u}_{{FB}^{s}_{p,q}(\mathbb{R}^{d})}<\infty\right\},
\end{equation*}
where,
\begin{equation*}
    \norm{u}_{{FB}^{s}_{p,q}(\mathbb{R}^{d})}:=\begin{cases}
    \Big(\sum_{j\geq-1}2^{jsq}\big\|\widehat{{\Delta}_{j}u}\big\|_{L^{p}(\mathbb{R}^{d})}^{q}\Big)^{\frac{1}{q}}
    \quad&\text{if}\quad q<\infty,\\
   \sup_{j\geq-1}2^{js}\big\|\widehat{{\Delta}_{j}u}\big\|_{L^{p}(\mathbb{R}^{d})}\quad&\text{if}\quad q=\infty.
    \end{cases}
\end{equation*}
\end{definition}
\begin{definition}\label{def.logS}
For $s,\sigma\in\RR$ and $\alpha>0$, we define the space $\dot{H}^{s,\sigma}_\alpha(\RR^2)$ by its norm
\begin{equation*}
\|u\|_{\dot{H}^{s,\sigma}_\alpha(\RR^2)}^2:=\sum_{q\leq0}2^{2qs}\|\dot{\Delta}_qu\|_{L^2(\RR^2)}^2+\sum_{q>0}q^{\alpha}2^{2q\sigma}\|\dot{\Delta}_qu\|_{L^2(\RR^2)}^2.
\end{equation*}
Finally, we define $\widetilde{L}_t^r\dot{H}^{s,\sigma}_\alpha$ by its norm
\begin{equation*}
\|u\|_{\widetilde{L}^r_t\dot{H}^{s,\sigma}_\alpha(\RR^2)}^2:=\sum_{q\leq0}2^{2qs}\|\dot{\Delta}_qu\|_{L^2_tL^2(\RR^2)}^2
+\sum_{q>0}q^{\alpha}2^{2q\sigma}\|\dot{\Delta}_qu\|_{L^2_tL^2(\RR^2)}^2.
\end{equation*}
\end{definition}
Through the whole paper, we denote $\dot{H}^{0,0}_1(\RR^2)$ by $L^2_{\rm log}(\RR^2)$ for the sake of simplicity.

 Next, we introduce localization in physical space. Firstly,  we define partition of unity that we shall use through our paper.
\begin{proposition}\label{prop-localization-in-physical}
Let $B_1(0):= \{\xi\in\RR^2\,|\,|\xi|\leq 1 \}$. There exists radial function $\phi$, valued in the interval $[0,1]$, belonging to $\mathcal{D}(B_1(0))$, and such that
\begin{equation}\label{eq.idenity}
\sum_{k\in\mathbb{Z}^{2}}\phi(2^{j}x+k)=1,\quad\forall\,x\in\RR^2,\,\text{and}\,\,j\in\ZZ,
\end{equation}
\begin{equation}\label{eq.physical-dis}
\mathrm{Supp}\,\phi_{j,i}\cap\mathrm{Supp}\,\phi_{j,k}=\emptyset,\quad\text{if}\,\,\,\,|i-k|\geq5,
\end{equation}
and
\begin{equation}\label{eq.physical-l2}
\frac{1}{16}\leq \sum_{k\in\ZZ^2}\phi_{j,k}^2(x)\leq 1,\quad \text{for all}\quad x\in\RR^2\text{ and }j\in\ZZ.
\end{equation}
Here and what in follows, we denote $\phi_{j,k}=\phi(2^{j}x-k).$
\end{proposition}
\begin{proof}
Let us choose a radial smooth function $\zeta$  satisfying
 \begin{equation*}
 \zeta(x)=\begin{cases}
 1&|x|\leq\frac{\sqrt{2}}{2};\\
 0&|x|\geq1.
 \end{cases}
 \end{equation*}
 Thus, we have that if a couple $(i,k)$ satisfying $|k-i|\geq5,$
\begin{equation}\label{eq.disjoint}
B_1(i)\cap B_1 (k)=\emptyset.
\end{equation}
Now, we let
\[S(x)=\sum_{k\in\mathbb{Z}^2}\zeta(x+k).\]
It is obvious that $S(x+k)=S(x)$ for all $k\in\ZZ^2.$ According to property \eqref{eq.disjoint}, we know that the above summation $S(x)$ is finite on $\RR^2.$ Thus, the function $S(x)$ is smooth on $\RR^2.$ On the other hand, we have
\begin{equation*}
\bigcup_{k\in\ZZ^2}B_1 (k)=\RR^2.
\end{equation*}
Since the function $\zeta$ is nonnegative and has value $1$ near $B_{\frac{\sqrt{2}}{2}}(0)$, it follows from the covering property that the function $S$ is positive.

Now, we claim that the function $\phi=\frac{\zeta}{S}$ is suitable. In fact, it is obvious that $\phi$ belongs to $\mathcal{D}(B_1(0))$ and
\begin{equation*}
\begin{split}
\sum_{k\in\mathbb{Z}^{2}}\phi(x-k)=&\sum_{k\in\mathbb{Z}^2}\frac{\zeta(x-k)}{S(x-k)}\\
=&\sum_{k\in\mathbb{Z}^2}\frac{\zeta(x-k)}{S(x)}=1,
\quad\forall\,x\in\RR^2\,\,\text{and}\,\,j\in\ZZ.
\end{split}
\end{equation*}
Now, it remains for us to prove \eqref{eq.physical-l2}. Let us denote that for $m=0,1,2,3,$
\begin{equation*}
I_m^j:=\sum_{\substack{k_1=4i+m}}\phi^2(2^{j}x+k)=1,\quad\forall\,x\in\RR^2\,\,\text{and}\,\,j\in\ZZ,
\end{equation*}
where $k=(k_1,k_2)^T$ and $i\in\ZZ.$

Thanks to property \eqref{eq.physical-dis}, it is obvious that
\begin{equation*}
1=\bigg(\sum_{k\in\mathbb{Z}^{2}}\phi(2^{j}x+k)\bigg)^2\leq 16\sum_{m=0}^4I_m^j.
\end{equation*}
This estimate yields \eqref{eq.physical-l2} and we end the proof of Proposition \ref{prop-localization-in-physical}.
\end{proof}
\begin{lemma} \label{lemma-Y}
\begin{enumerate}[\rm (i)]
  \item Let $\Phi\in \mathcal{S}(\mathbb{R}^2)$, then there holds
\begin{equation}\label{eq.Y-1}
\big\|\Phi\ast f\|_{L^\infty(\RR^2)}\leq C\sup_{k\in\mathbb{Z}^2}\|\phi_{0,k}f\|_{L^1(\RR^2)},
\end{equation}
where $C$ is a positive constant independent of $f$.
  \item  Let $i,\,j\in\ZZ$ and $i\leq j$. Then, we have that for each $q\in[1,\infty],$
\begin{equation}\label{eq.Y-2}
\sup_{k\in\mathbb{Z}^2}\|\phi_{i,k}f\|_{L^q(\RR^2)}\leq   2^{\frac{1+2(j-i)}{q}}\sup_{k\in\mathbb{Z}^2}\|\phi_{j,k}f\|_{L^q(\RR^2)}.
\end{equation}
\end{enumerate}
\end{lemma}
\begin{proof}
Estimate \eqref{eq.Y-2} follows from the covering theorem directly.
So we just show estimate for \eqref{eq.Y-1}. In view of \eqref{eq.idenity}, one can write
\begin{align*}
\Phi\ast f(x)=&\int_{\mathbb{R}^2}\Phi(x-y)f(y)\intd y=\sum_{k\in\mathbb{Z}^2}\int_{\mathbb{R}^2}\Phi(x-y)\phi_{0,k}(y)f(y)\intd y\nonumber\\
=&\sum_{|k-x|<5}\int_{\mathbb{R}^2}\Phi(x-y)\phi_{0,k}(y)f(y)\intd y+\sum_{|k-x|\geq5}\int_{\mathbb{R}^2}\Phi(x-y)\phi_{0,k}(y)f(y)\intd y.
\end{align*}
On one hand, it is obvious from the H\"older inequality that
\begin{equation}\label{eq.L-1}
\begin{split}
\sum_{|k-x|<5}\int_{\mathbb{R}^2}\Phi(x-y)\phi_{0,k}(y)f(y)\intd y\leq&\sum_{|k-x|<5}\|\Phi\|_{L^\infty(\RR^2)}\|\phi_{0,k}f\|_{L^1(\RR^2)}\\
\leq &C  \sup_{k\in\ZZ^2}\|\phi_{0,k}f\|_{L^1(\RR^2)}.
\end{split}
\end{equation}
On the other hand,  by the H\"older inequality and the property of support of $\phi$, we readily have
\begin{equation}\label{eq.L-2}
\begin{split}
&\sum_{|k-x|\geq5}\int_{\mathbb{R}^2}\Phi(x-y)\phi_{0,k}(y)f(y)\intd y\\
=&\sum_{i\geq 0}\int_{2^{i}<|x-y|\leq 2^{i+1}}\Phi(x-y)\sum_{k\in\mathbb{Z}^2}\phi_{0,k}(y)f(y)\intd y\\
\leq&\sum_{i\geq 0}2^{-4i}\int_{2^{i}<|x-y|\leq 2^{i+1}}|x-y|^4\Phi(x-y)\sum_{|k-x|\leq2^{i+2}}\phi_{0,k}(y)f(y)\intd y\\
\leq&\sum_{i\geq 0}2^{-4i}\sup_{x\in\mathbb{R}^2}\bigl|x^{4}\Phi(x)\bigr|\sum_{|k-x|\leq2^{i+2}}\|\phi_{0,k}f\|_{L^1(\RR^2)}\\
\leq&C\sup_{k\in\ZZ^2}\|\phi_{0,k}f\|_{L^1(\RR^2)}\sum_{i\geq 0}2^{-2i}\leq C\sup_{k\in\ZZ^2}\|\phi_{0,k}f\|_{L^1(\RR^2)}.
\end{split}
\end{equation}
Collecting estimate \eqref{eq.L-1} and estimate \eqref{eq.L-2} yields the desired result \eqref{eq.Y-1}.
\end{proof}
\begin{lemma}\label{local-Bernstein}
Let $j\in\mathbb{N}$ and $2^{j}\sup_{k\in\mathbb{Z}^2} \|\phi_{j,k}f \|_{L^2(\RR^2)}<\infty$. Then there holds
\begin{equation*}
\|S_jf\|_{L^\infty(\RR^2)}\leq C2^{j}\sup_{k\in\mathbb{Z}^2}\big\|\phi_{j,k}f\big\|_{L^2(\RR^2)},
\end{equation*}
where $C$ is a positive constant  independent of $j$ and $f$.
\end{lemma}
\begin{proof}
By changing a variable, one can conclude that
\begin{equation*}
|S_jf(x)|=\Big|2^{2j}\int_{\mathbb{R}^2}\Phi\big(2^{j}y\big)f(x-y)\intd y\Big|=\Big|\int_{\mathbb{R}^2}\varphi(y)f\Big(\frac{2^{j}x-y}{2^{j}}\Big)\intd y\Big|.
\end{equation*}
Let $f_j(x):=f(x/2^{j})$, then we have
\begin{equation*}
|S_jf(x)|=\Big|\int_{\mathbb{R}^2}\Phi(y)f_j(2^{j}x-y)\intd y\Big|=|\dot\Delta_0f_j(2^{j}x)|.
\end{equation*}
By using the first estimate in Lemma  \ref{lemma-Y} and the H\"older inequality, we know that
\begin{equation*}
\begin{split}
\|S_0f_j\|_{L^\infty(\RR^2)}\leq & C\sup_{k\in\mathbb{Z}^2}\big\|\phi_{0,k}f_j\big\|_{L^1(\RR^2)}\\
=&C\sup_{k\in\mathbb{Z}^2}\|\phi_{0,k}f(\cdot/2^j)\|_{L^1(\RR^2)}\leq C\sup_{k\in\mathbb{Z}^2}\big\|\phi_{0,k}f(\cdot/2^j)\big\|_{L^2(\RR^2)}.
\end{split}
\end{equation*}
This implies
\begin{equation}\label{eq.-local-1}
\|S_0f_j\|_{L^\infty(\RR^2)}\leq C\sup_{k\in\mathbb{Z}^2}\big\|\phi_{0,k}f(\cdot/2^j)\big\|_{L^2(\RR^2)}.
\end{equation}
Clearly, we have by changing a variable that
\begin{align*}
\Big(\int_{\RR^2}\Big|\phi_{0,k}(y)f\Big(\frac{y}{2^{j}}\Big)\Big|^{2}\intd y\Big)^{\frac{1}{2}}=&\Big(\int_{\RR^2}\Big|\phi{(y-k)}f\Big(\frac{y}{2^{j}}\Big)\Big|^{2}\intd y\Big)^{\frac{1}{2}}\\
=&\Big(2^{2j}\int_{\RR^2}\Big|\phi{\big(2^{j}y-k\big)}f(y)\Big|^{2}\intd y\Big)^{\frac{1}{2}}\\
\leq&2^{j}\sup_{k\in\mathbb{Z}^2}\big\|\phi_{j,k}f\big\|_{L^2(\RR^2)}.
\end{align*}
Inserting this estimate into \eqref{eq.-local-1} yields the desired result.
\end{proof}
\subsection{The principal eigenvalue of Lapace operator}

Next, we review some statements concerning on the principal eigenvalue of Lapace operator.
\begin{proposition}[Chap 6.5 Theorem 2, \cite{Evans-Book}]\label{Prop-Evans}
Assume that $U$ is open and bounded, and $\partial U$ is smooth. There hold that
\begin{enumerate}[\rm (i)]
  \item We have
  \begin{equation*}
\lambda_1:=\min\Big\{B[u,u]:=\int_{U}|\nabla u(x)|^2\intd x\,\Big|\,u\in H_0^1(U),\,\,\,\|u\|_{L^2(U)}=1\Big\}.
\end{equation*}
Furthermore, the above minimum is attained  for a function $w_1$,  positive in $U,$ which solves
  \begin{equation*}
  \left\{\begin{array}{ll}
  -\Delta w_1=\lambda_1w_1\quad\,\,\,\,\,\text{in}\quad U,\\
  w_1=0\quad\quad \quad\quad\quad\text{on}\quad\partial U.
  \end{array}\right.
  \end{equation*}
  \item Finally, if $u\in H_0^1(U)$ is any weak solution of
  \begin{equation*}
  \left\{\begin{array}{ll}
  -\Delta u=\lambda_1u\quad\,\,\,\quad\text{in}\quad U,\\
  u=0\quad\quad \quad\quad\quad\text{on}\quad\partial U,
  \end{array}\right.
  \end{equation*}
  then $u$ is a multiple of $w_1$.
\end{enumerate}
\end{proposition}
Next, we will introduce an important property of the solution of  the eigenvalue problem, which is the main ingredient of our proof.
\begin{lemma}\label{Lem-En}
Let $f\in\mathcal{S}(\RR^2)$ and $\varphi$ be the solution of the above eigenvalue problem
\begin{equation}\label{eq.eigenvalue}
\left\{\begin{array}{ll}
-\Delta\varphi=\lambda_1 \varphi\qquad x\in B_{2}(0),\\
\varphi|_{\partial B_{2}(0)}=0
\end{array}\right.
\end{equation}
satisfying $\varphi(x)=0$ for all $x\in\RR^2\setminus B_2(0)$. Then, we have that for $\varphi_r=\varphi(x/r)$,
\begin{equation*}
-\int_{\RR^2}\Delta f f\varphi_r\intd x=\frac{\lambda_1}{2r^{2}}\int_{\RR^2}f^{2}\varphi_r\intd x+\int_{\RR^2}\varphi_r|\nabla f|^2\intd x+\frac12\int_{\partial B_{2r}(0)}f^{2}\nabla\varphi_{r}\cdot n\intd S.
\end{equation*}
\end{lemma}
\begin{proof}
Integration by parts yields
\begin{equation}\label{eq.newtest}
\begin{split}
-\int_{\RR^2}\Delta ff\varphi_r\intd x=&-\frac12\int_{\RR^2}\left(\Delta ff+\Delta ff\right)\varphi_r\intd x\\
=&-\frac12\int_{\RR^2}\big(\Delta f^2-2|\nabla f|^2\big)\varphi_r\intd x\\
=&-\frac12\int_{\RR^2} \Delta f^2\varphi_r\intd x+\int_{\RR^2}|\nabla f|^2\varphi_r\intd x.
\end{split}
\end{equation}
Integrating by parts again and using that $\varphi_r$ solves the following eigenvalue problem
\begin{equation*}
\left\{\begin{array}{ll}
-\Delta\varphi_r=\frac{\lambda_1}{r^{2}} \varphi_r\quad x\in\RR^2,\\
\varphi_r|_{\partial B_{2r}(0)}=0,
\end{array}\right.
\end{equation*}
we easily find that
\begin{equation*}
\begin{split}
-\frac12\int_{\RR^2} \Delta f^2\varphi_r\intd x=&-\frac12\int_{\RR^2} f^2\Delta \varphi_r\intd x+\frac12\int_{\partial B_{2r}(0)}f^{2}\nabla\varphi_{r}\cdot n\intd S\\
=&\frac{\lambda_1}{2r^{2}}\int_{\RR^2}f^{2}\varphi_r\intd x+\frac12\int_{\partial B_{2r}(0)}f^{2}\nabla\varphi_{r}\cdot n\intd S.
\end{split}
\end{equation*}
Plugging this estimate into \eqref{eq.newtest} yields the desired result and then we finish the proof of the lemma.
\end{proof}
\begin{corollary}\label{coro-test-prop}
Let $\varphi_{j,k}:=\varphi(2^{j}x-k)$ with $\varphi$ defined in Lemma \ref{Lem-En}. Then, we have
\begin{enumerate}[\rm(i)]
   \item  for $f\in\mathcal{S}(\RR^2)$,
\begin{equation}\label{eq.local-smooth}
\begin{split}
-\int_{\RR^2}\Delta f f\varphi_{j,k}\intd x=&\frac{\lambda_1 2^{2j}}{2}\int_{\RR^2}f^{2}\varphi_{j,k}\intd x+\int_{\RR^2}\varphi_{j,k}|\nabla f|^2\intd x\\
&+\frac12\int_{\partial B_{2^{-j+1}}(k)}f^{2}\nabla\varphi_{j,k}\cdot n\intd S.
\end{split}
\end{equation}
\item  there exits two positive constants $M_l$ and $M_u$ such that
\begin{equation}\label{eq.MlMu}
M_l\phi_{j,k}\leq\varphi_{j,k}\phi_{j,k}\leq M_u\phi_{j,k}.
\end{equation}
\item orthogonal property:
\begin{equation*}
\varphi_{j,i}\phi_{j,k}=0\quad\text{for}\quad|i-k|\geq5.
\end{equation*}
 \end{enumerate}
\end{corollary}
With this test function in hand, we will give a refined $L^2$-estimate for smooth solution of the linear heat equation.
\begin{proposition}\label{Prop-heat-im}
Let the scalar function $f$ be a smooth solution of the following linear heat equation in the plane:
\begin{equation}\label{linear-heat}
\partial_tf-\Delta f=0,\quad f|_{t=0}=f_0.
\end{equation}
Then, there hold that
\begin{equation*}
\|f(t)\|_{L^2(\RR^2)}^2+2\int_0^t\|\nabla f(\tau)\|_{L^2(\RR^2)}^2\intd\tau=\|f_0\|^2_{L^2(\RR^2)}.
\end{equation*}
and
\begin{equation}\label{estimate-new}
\lambda_1 \sup_{j\in\ZZ}2^{2j}\int_0^t\sum_{k\in\ZZ^2}\big\|\sqrt{\phi_{j,k}}f(\tau)\big\|_{L^2(\RR^2)}^2\,\intd\tau
\leq  C\big(t,\,\|f_0\|_{L^2(\RR^2)}\big).
\end{equation}
\end{proposition}
\begin{remark}
Since the convergence of the series is not uniform pertaining to parameter~$j$,  estimate   \eqref{estimate-new}  does not imply
\begin{equation*}
\lambda_1 \sup_{j\in\ZZ}2^{2j}\int_0^t\big\|f(\tau)\big\|_{L^2(\RR^2)}^2\,\intd\tau
\leq  C\big(t,\,\|f_0\|_{L^2(\RR^2)}\big).
\end{equation*}
\end{remark}
\begin{proof}[Proof of Proposition \ref{Prop-heat-im}]
Firstly, the standard $L^2$-inner argument enables us to conclude that for all $t\geq0,$
\begin{equation}\label{eq.L2heat}
\|f(t)\|_{L^2(\RR^2)}^2+2\int_0^t\|\nabla f(\tau)\|_{L^2(\RR^2)}^2\intd\tau=\|f_0\|^2_{L^2(\RR^2)}.
\end{equation}
Multiplying \eqref{linear-heat} by $\varphi_{j,k}f$, we see that
\[\frac12\partial_t\big(\varphi_{j,k}f^2\big)-\varphi_{j,k}f\Delta f=0.\]
Integrating the above equality in space variable over $\RR^2$ and using equality \eqref{eq.local-smooth}, one has
\begin{equation}\label{eq.diff-heat}
\begin{split}
&\frac12\dtd\big\|\sqrt{\varphi_{j,k}}f(t)\big\|^2_{L^2(\RR^2)} +\frac{\lambda_1 2^{2j}}{2}\int_{\RR^2}f^{2}\varphi_{j,k}\intd x+\int_{\RR^2}\varphi_{j,k}|\nabla f|^2\intd x\\
=&-\frac12\int_{\partial B_{2^{-j+1}}(k)}f^{2}\nabla\varphi_{j,k}\cdot n\intd S.
\end{split}
\end{equation}
By the Cauchy-Schwartz inequality, we can infer that
\begin{equation}\label{eq.boundary}
\begin{split}
- \int_{\partial B_{2^{-j+1}}(k)}f^{2}\nabla\varphi_{j,k}\cdot n\intd S \leq&\big\|f^{2}\big\|_{L^2\left(\partial B_{2^{-j+1}}(k)\right)}\|\nabla\varphi_{j,k}\|_{L^2\left(\partial B_{2^{-j+1}}(k)\right)}\\
\leq&\|f\|^{2}_{L^4\left(\partial B_{2^{-j+1}}(k)\right)}\|\nabla\varphi_{j,k}\|_{L^2\left(\partial B_{2^{-j+1}}(k)\right)}.
\end{split}
\end{equation}
By the trace theorem (see for example Theorem 5.36 in \cite{Adams}) and the H\"older inequality, we easily find that
\begin{equation}\label{eq.trace1}
\begin{split}
 \|f \|_{L^4\left(\partial B_{2^{-j+1}}(k)\right)}\leq &C\|f\|_{W^{1,\frac85}\left( B_{2^{-j+1}}(k)\right)}\\
 \leq& C2^{-\frac{j}{4}}\|f \|_{H^{1}\left( B_{2^{-j+1}}(k)\right)}\\
 \leq& C2^{-\frac{j}{4}}\sum_{|k'-k|\leq2}\left(\big\|\sqrt{\phi_{j,k'}}f\big \|_{L^{2}\left( \RR^2\right)}+\big\|\sqrt{\phi_{j,k'}}\nabla f\big \|_{L^{2}\left( \RR^2\right)}\right),
 \end{split}
\end{equation}
where $W^{1,\frac85}(\Omega)$ is the general Sobolev space.

On the other hand, we find that
\begin{equation}\label{eq.trace2}
\|\nabla\varphi_{j,k}\|_{L^2\left(\partial B_{2^{-j+1}}(k)\right)}=2^{j}\Big(\int_{\partial B_{2^{-j+1}}(0)}|\nabla\varphi|^{2}(2)\intd S\Big)^{\frac12}=C2^{\frac{j}{2}}.
\end{equation}
Inserting estimates \eqref{eq.trace1} and \eqref{eq.trace2} into \eqref{eq.boundary} leads to
\begin{equation}\label{eq.boundary-final}
\begin{split}
- \int_{\partial B_{2^{-j}}(k)}f^{2}\nabla\varphi_{j,k}\cdot n\intd S \leq&
C\|f\|^{2}_{H^1\left(  B_{2^{-j+1}}(k)\right)}\\
 \leq&
C\sum_{|k'-k|\leq2}\left(\big\|\sqrt{\phi_{j,k'}}f\big \|_{L^{2}\left( \RR^2\right)}+\big\|\sqrt{\phi_{j,k'}}\nabla f\big \|_{L^{2}\left( \RR^2\right)}\right).
\end{split}
\end{equation}
With this estimate, summing equality \eqref{eq.diff-heat} over $k\in\ZZ^2$ and integrating the resulting equality with respect to time $t$, we immediately have
\begin{equation*}
\begin{split}
&\lambda_1 2^{2j}\int_0^t\sum_{k\in\ZZ^2}\big\|\sqrt{\varphi_{j,k}}f(\tau)\big\|^2_{L^2(\RR^2)}\intd\tau\\
\leq&\sum_{k\in\ZZ^2}\big\|\sqrt{\varphi_{j,k}}f_0\big\|^2_{L^2(\RR^2)}+C\sum_{k\in\ZZ^2}\sum_{|k'-k|\leq2}\int_0^t\left(\big\|\sqrt{\phi_{j,k'}}f(\tau)\big \|_{L^{2}\left( \RR^2\right)}+\big\|\sqrt{\phi_{j,k'}}\nabla f(\tau)\big \|_{L^{2}\left( \RR^2\right)}\right)\intd\tau.
\end{split}
\end{equation*}
By the discrete Young inequality and \eqref{eq.idenity}, one has
\begin{equation*}
\begin{split}
\sum_{k\in\ZZ^2}\big\|\sqrt{\varphi_{j,k}}f_0\big\|^2_{L^2(\RR^2)}=&\sum_{k\in\ZZ^2}\sum_{|k-k'|<5}\big\|\sqrt{\varphi_{j,k}}\sqrt{\phi_{j,k'}}f_0\big\|^2_{L^2(\RR^2)}\\
\leq&C\sum_{k\in\ZZ^2}\sum_{|k-k'|<5}\big\|\sqrt{\phi_{j,k'}}f_0\big\|^2_{L^2(\RR^2)}\\
\leq&C\sum_{k'\in\ZZ^2}\big\|\sqrt{\phi_{j,k'}}f_0\big\|^2_{L^2(\RR^2)}=C\|f_0\|_{L^2(\RR^2)}^2.
\end{split}
\end{equation*}
On the other hand, according to the property of \eqref{eq.MlMu}, we easily find that
\begin{equation*}
\big\|\sqrt{\phi_{j,k}}f\big\|^2_{L^\infty_tL^2(\RR^2)}\leq C\big\|\sqrt{\phi_{j,k}}\sqrt{\varphi_{j,k}}f\big\|^2_{L^\infty_tL^2(\RR^2)},
\end{equation*}
which implies that
\begin{equation*}
\sum_{k\in\ZZ^2}\big\|\sqrt{\phi_{j,k}}f\big\|^2_{L^\infty_tL^2(\RR^2)}\leq C\sum_{k\in\ZZ^2}\big\|\sqrt{\varphi_{j,k}}f\big\|^2_{L^\infty_tL^2(\RR^2)}.
\end{equation*}
In a similar fashion as above, it is easy to conclude that
\begin{align*}
  \lambda_1 2^{2j}\int_0^t\sum_{k\in\ZZ^2}\big\|\sqrt{\phi_{j,k}}f(\tau)\big\|_{L^2(\RR^2)}^2\intd\tau
\leq C\lambda_1 2^{2j}\int_0^t\sum_{k\in\ZZ^2}\big\|\sqrt{\varphi_{j,k}}f(\tau)\big\|^2_{L^2(\RR^2)}\intd\tau .
\end{align*}
Collecting these estimates above, we readily have
\begin{equation*}
\begin{split}
 \lambda_1 2^{2j}\int_0^t\sum_{k\in\ZZ^2}\big\|\sqrt{\phi_{j,k}}f(\tau)\big\|_{L^2(\RR^2)}^2\intd\tau
\leq C\|f_0\|_{L^2(\RR^2)}^2+C\int_0^t\|f(\tau)\|_{H^1(\RR^2)}^2\intd\tau.
\end{split}
\end{equation*}
Taking supremum the above inequality over $j\in\ZZ$ together with estimate \eqref{eq.L2heat} gives the required result.
\end{proof}
In the last part of this section, we are devoted to show a estimate for the tri-linear term which will be used in the proof.
\begin{lemma}\label{Lem-tri}
There holds that
\begin{equation*}
\begin{split}
\int_0^t\int_{\mathbb{R}^2}fgh\intd x\mathrm{d}\tau\leq & \|f\|_{L^2_tL^2(\RR^2)}\|\nabla g\|_{L^2_tL^2(\RR^2)}\|h\|_{L^\infty_tL^2(\RR^2)}\\
&+2N\sup_{k\in\ZZ}\|\dot{S}_{k}g\|_{L^2_tL^\infty(\RR^2)}\|f\|_{L^2_tL^2(\RR^2)}\|h\|_{L^\infty_tL^2(\RR^2)}\\
&+\sup_{k\in\ZZ}\|\dot{S}_{k}g\|_{L^2_tL^\infty(\RR^2)}\|h\|_{\widetilde{L}^\infty_t\dot{B}^0_{2,2}(\RR^2)}\Big(\sum_{|k|\geq N}\|\dot{\Delta}_kf\|^2_{L^2_tL^2(\RR^2)}\Big)^{\frac12}.
\end{split}
\end{equation*}
\end{lemma}
\begin{remark}
Note that for any $i\leq -1$, we have from the support property that
\[\dot{S}_if=S_0\dot{S}_if.\]
This implies that
$\|\dot{S}_if\|_{L^\infty}\leq C\|S_0f\|_{L^\infty} $ for $i\leq-1$, hence we have
\[\sup_{i\in\ZZ}\|\dot{S}_if\|_{L^\infty}\leq C\sup_{i\geq0}\|S_if\|_{L^\infty}.\]
So, we often use $\sup_{i\geq0}\|S_if\|_{L^\infty}$ instead of $\sup_{i\in\ZZ}\|\dot{S}_if\|_{L^\infty}$ in the following parts.
\end{remark}
\begin{proof}[Proof of Lemma \ref{Lem-tri}]
According to the Bony para-product decomposition, one writes
\begin{equation*}
\begin{split}
\int_0^t\int_{\mathbb{R}^2}fgh\intd x\mathrm{d}\tau=&\sum_{q\in\mathbb{Z}}\int_0^t\int_{\mathbb{R}^2}\dot{\Delta}_qf\big(\dot{T}_gh\big)\intd x\mathrm{d}\tau+\sum_{q\in\mathbb{Z}}\int_0^t\int_{\mathbb{R}^2}\dot{\Delta}_qf\big(\dot{T}_hg\big)\intd x\mathrm{d}\tau\\
&+\sum_{q\in\mathbb{Z}}\int_0^t\int_{\mathbb{R}^2}\dot{\Delta}_qf\big(\dot{R}(g,h)\big)\intd x\mathrm{d}\tau.
\end{split}
\end{equation*}
For the first term in the right side of the above equality, by the H\"older inequality and the support property of paraproduct, we have
\begin{equation*}
\begin{split}
\sum_{q\in\mathbb{Z}}\int_0^t\int_{\mathbb{R}^2}\dot{\Delta}_qf(\dot{T}_gh)\intd x\mathrm{d}\tau=&\sum_{q\in\mathbb{Z}}\sum_{|k-q|\leq5}\int_0^t\int_{\mathbb{R}^2}\dot{\Delta}_qf(\dot{S}_{k-1}g\dot{\Delta}_kh)\intd x\mathrm{d}\tau\\
=&\sum_{q\in\mathbb{Z}}\sum_{|k-q|\leq5}\int_0^t\int_{\mathbb{R}^2}\dot{\Delta}_qf(\dot{S}_{k-1}g^{H}_{N}\dot{\Delta}_kh)\intd x\mathrm{d}\tau\\
&+\sum_{q\in\mathbb{Z}}\sum_{|k-q|\leq5}\int_0^t\int_{\mathbb{R}^2}\dot{\Delta}_qf(\dot{S}_{k-1}g_{N}^{L}\dot{\Delta}_kh)\intd x\mathrm{d}\tau.
\end{split}
\end{equation*}
where $g^{L}_{N}=\dot{S}_{N}g$ and $g^{H}_{N}=(I_{\mathrm{d}}-\dot{S}_N)g.$

For the para-product term, by the property of support and the H\"older inequality, we see that
\begin{align*}
&\sum_{q}\sum_{|k-q|\leq5}\int_0^t\int_{\mathbb{R}^2}\dot{\Delta}_qf\big(\dot{S}_{k-1}g^{H}_{N}\dot{\Delta}_kh\big)\intd x\mathrm{d}\tau\\
=&\sum_{q\geq N-5}\sum_{|k-q|\leq5}\int_0^t\int_{\mathbb{R}^2}\dot{\Delta}_qf\big(\dot{S}_{k-1}g^{H}_{N}\dot{\Delta}_kh\big)\intd x\mathrm{d}\tau\\
\leq&\sum_{q\geq N-5}\sum_{|k-q|\leq5}\int_0^t\|\dot{\Delta}_qf(\tau)\|_{L^2(\RR^2)}\|\dot{S}_{k-1}g(\tau)\|_{L^\infty(\RR^2)}\|\dot{\Delta}_kh(\tau)\|_{L^2(\RR^2)}\intd\tau\\
\leq&\sup_{k\in\ZZ}\|\dot{S}_{k-1}g\|_{L^2_tL^\infty(\RR^2)} \sum_{q\geq N-5}\|\dot{\Delta}_qf\|_{L^2_tL^2(\RR^2)}\sum_{|k-q|\leq5}\|\dot{\Delta}_kh\|_{L^\infty_tL^2(\RR^2)} \\
\leq&\sup_{k\in\ZZ}\|\dot{S}_{k}g\|_{L^2_tL^\infty(\RR^2)}\|h\|_{\widetilde{L}^\infty_t\dot{B}^0_{2,2}(\RR^2)}\bigg(\sum_{q\geq N-5}\|\dot{\Delta}_qf\|^2_{L^2_tL^2(\RR^2)}\bigg)^{\frac12}.
\end{align*}
On the other hand, we find that
\begin{align*}
 \int_0^t\int_{\mathbb{R}^2}\dot{\Delta}_qf(\dot{S}_{k-1}g_{N}^{L}\dot{\Delta}_kh)\intd x\mathrm{d}\tau
=&\int_0^t\int_{\mathbb{R}^2}\dot{\Delta}_qf_N^H(\dot{S}_{k-1}g_{N}^{L}\dot{\Delta}_kh)\intd x\mathrm{d}\tau\\
&+\int_0^t\int_{\mathbb{R}^2}\dot{\Delta}_qf_N^M(\dot{S}_{k-1}g_{N}^{L}\dot{\Delta}_kh)\intd x\mathrm{d}\tau\\
&+\int_0^t\int_{\mathbb{R}^2}\dot{\Delta}_qf_{-N}^L(\dot{S}_{k-1}g_{N}^{L}\dot{\Delta}_kh)\intd x\mathrm{d}\tau\\
:=&I+II+III,
\end{align*}
where ${f_N^M=\sum_{q=-N}^N\dot{\Delta}_qf.}$

In a similar way as above,  we can obtain
\begin{equation*}
\sum_{q\in\mathbb{Z}}\sum_{|k-q|\leq5}I\leq \sup_{k\in\mathbb{Z}}\|\dot{S}_{k}g\|_{L^2\big([0,t];\,L^\infty(\RR^2)\big)}\big\|f^{L}_{-N}\big\|_{L^2\big([0,t];\,L^2(\RR^2)\big)}
\|h\|_{\widetilde{L}^\infty\big([0,t];\,\dot{B}^0_{2,2}(\RR^2)\big)}
\end{equation*}
and
\begin{equation*}
\sum_{q\in\mathbb{Z}}\sum_{|k-q|\leq5}III\leq \sup_{k\in\mathbb{Z}}\|\dot{S}_{k-1}g\|_{L^2_tL^\infty(\RR^2)}\|f^{H}_{N}\|_{L^2_tL^2(\RR^2)}\|h\|_{L^\infty_t\dot{B}^0_{2,2}(\RR^2)}.
\end{equation*}
Now we need to tackle with the term involving the middle frequency of $f$. By the discrete Young inequality, we readily have
\begin{align*}
&\sum_{-N<q\leq N}\sum_{|k-q|\leq5}\int_0^t\int_{\mathbb{R}^2}\dot{\Delta}_qf\big(\dot{S}_{k-1}g_{N}^{L}\dot{\Delta}_kh\big)\intd x\mathrm{d}\tau\\
\leq&\sum_{-N<q\leq N}\sum_{|k-q|\leq5}\int_0^t\|\dot{\Delta}_qf(\tau)\|_{L^2(\RR^2)}\|\dot{S}_{k-1}g(\tau)\|_{L^\infty(\RR^2)}\|\dot{\Delta}_kh(\tau)\|_{L^2(\RR^2)}\intd\tau\\
\leq&2N\sup_{k\in\ZZ}\|\dot{S}_{k-1}g\|_{L^2_tL^\infty(\RR^2)}\sup_{q\in\ZZ}\Big(\|\dot{\Delta}_qf\|_{L^2_tL^2(\RR^2)} \sum_{|k-q|\leq5}\|\dot{\Delta}_kh\|_{L^\infty_tL^2(\RR^2)}\Big)\\
\leq&2N\sup_{k\in\ZZ}\|\dot{S}_{k-1}g\|_{L^2_tL^\infty(\RR^2)}\sup_{q\in\ZZ}\|\dot{\Delta}_qf\|_{L^2_tL^2(\RR^2)}\sup_{q}\|\dot{\Delta}_qh\|_{L^\infty_tL^2(\RR^2)}\\
\leq&2N\sup_{k\in\ZZ}\|\dot{S}_{k}g\|_{L^2_tL^\infty(\RR^2)}\|f\|_{L^2_tL^2(\RR^2)}\|h\|_{L^\infty_tL^2(\RR^2)}.
\end{align*}
By the H\"older inequality, the second term can be bounded as follows:
\begin{align*}
&\sum_{q\in\mathbb{Z}}\int_0^t\int_{\mathbb{R}^2}\dot{\Delta}_qf\big(\dot{T}_hg\big)\intd x \mathrm{d}\tau\\
=&\sum_{q\in\mathbb{Z}}\sum_{|k-q|\leq5}\int_0^t\int_{\mathbb{R}^2}\dot{\Delta}_qf\big(\dot{S}_{k-1}h\dot{\Delta}_kg\big)\intd x\mathrm{d}\tau\\
\leq&\sum_{q\in\mathbb{Z}}\sum_{|k-q|\leq5}\int_0^t\|\dot{\Delta}_qf(\tau)\|_{L^2(\RR^2)}\|\dot{S}_{k-1}h(\tau)\|_{L^2(\RR^2)}\|\dot{\Delta}_kg(\tau)\|_{L^\infty(\RR^2)}\intd\tau\\
\leq&\|h\|_{L^\infty_tL^2(\RR^2)}\sum_{q\in\ZZ}\|\dot{\Delta}_qf\|_{L^2_tL^2(\RR^2)} \sum_{|k-q|\leq5}2^{k}\|\dot{\Delta}_kg\|_{L^2_tL^2(\RR^2)} \\
\leq&\|f\|_{L^2_tL^2(\RR^2)}\|\nabla g\|_{L^2_tL^2(\RR^2)}\|h\|_{L^\infty_tL^2(\RR^2)}.
\end{align*}
As for the remainder term, by the support property of the remainder term and the H\"older inequality, we get
\begin{equation}\label{eq.R3}
\begin{split}
\sum_{q\in\mathbb{Z}}\int_0^t\int_{\mathbb{R}^2}\dot{\Delta}_qf\big(\dot{R}(g,h)\big)\intd x\mathrm{d}\tau=&\sum_{q\in\mathbb{Z}}\sum_{k\geq q-5}\int_0^t\int_{\mathbb{R}^2}\dot{\Delta}_qf\big(\dot{\Delta}_kg\widetilde{\dot\Delta}_kh\big)\intd x\mathrm{d}\tau\\
\leq&\sum_{q\in\mathbb{Z}}\sum_{k\geq q-5}\int_0^t\|\dot{\Delta}_qf\|_{L^2(\RR^2)}\big\|\widetilde{\dot{\Delta}}_q(\dot{\Delta}_kg\widetilde{\dot{\Delta}}_kh)\big\|_{L^2(\RR^2)}\,\mathrm{d}\tau\\
\leq&\sum_{q\in\mathbb{Z}}\sum_{k\geq q-5}2^{q}\int_0^t\|\dot{\Delta}_qf\|_{L^2(\RR^2)}\big\|\dot{\Delta}_kg\widetilde{\dot{\Delta}}_kh\big\|_{L^1(\RR^2)}\,\mathrm{d}\tau.
\end{split}
\end{equation}
Furthermore, by the H\"older inequality, we obtain
\begin{align*}
&\int_0^t\|\dot{\Delta}_qf\|_{L^2(\RR^2)}\big\|\dot{\Delta}_kg\widetilde{\dot{\Delta}}_kh\big\|_{L^1(\RR^2)}\mathrm{d}\tau\\
\leq&\int_0^t\|\dot{\Delta}_qf(\tau)\|_{L^2(\RR^2)}\|\dot{\Delta}_kg(\tau)\|_{L^2(\RR^2)}\big\|\widetilde{\dot{\Delta}}_kh(\tau)\big\|_{L^2(\RR^2)}\,\mathrm{d}\tau.
\end{align*}
Inserting this estimate into \eqref{eq.R3}, we get from the discrete Young inequality and the H\"older inequality that
\begin{equation*}
\begin{split}
&\sum_{q\in\mathbb{Z}}\int_0^t\int_{\mathbb{R}^2}\dot{\Delta}_qf\big(\dot{R}(g,h)\big)\intd x\mathrm{d}\tau\\
\leq&\int_0^t\sum_{q\in\mathbb{Z}}\sum_{k\geq q-5}2^{q}\|\dot{\Delta}_qf(\tau)\|_{L^2(\RR^2)}\|\dot{\Delta}_kg(\tau)\|_{L^2(\RR^2)}\big\|\widetilde{\dot{\Delta}}_kh(\tau)\big\|_{L^2(\RR^2)}\,\mathrm{d}\tau\\
\leq&C\int_0^t\|f(\tau)\|_{L^2(\RR^2)}\|\nabla g(\tau)\|_{L^2(\RR^2)}\|h(\tau)\|_{L^2(\RR^2)}\intd\tau\\
\leq&C\|f\|_{L^2_tL^2(\RR^2)}\|\nabla g\|_{L^2_tL^2(\RR^2)}\|h\|_{L^\infty_tL^2(\RR^2)}.
\end{split}
\end{equation*}
Collecting all these estimates yields the desired result.
\end{proof}

\section{A priori estimates}\label{Priori}
\setcounter{section}{3}\setcounter{equation}{0}
This section is devoted to show some useful \emph{a priori} estimates for the smooth solution of problem~\eqref{eq.MNS} which can be viewed as an preparation for proving our theorems. Let us begin by proving the $L^2$-energy estimate of solution $(u,B,E)$.
 \begin{proposition}
 Let $(u_0,E_0,B_0)\in (L^2(\RR^2))^3$, and $(u,B,E)$ be a smooth solution of problem \eqref{eq.MNS}. Then we have that for all $t\geq0,$
 \begin{equation}\label{eq.L2estimate}
 \begin{split}
 &\big\|\big(u,E,B\big)(t)\big\|^2_{L^2(\RR^2)}+2\int_0^t\|\nabla u(\tau)\|^2_{L^2(\RR^2)}\intd\tau+2\int_0^t\|j(\tau)\|^2_{L^2(\RR^2)}\intd\tau\\
=&\|u_0\|^2_{L^2(\RR^2)}+ \|E_0\|^2_{L^2(\RR^2)}+ \|B_0\|^2_{L^2(\RR^2)}.
 \end{split}
 \end{equation}
 where $\big\|\big(u,E,B\big)(t)\big\|^2_{L^2(\RR^2)}=\|u(t)\|^2_{L^2(\RR^2)}+ \|E(t)\|^2_{L^2(\RR^2)}+ \|B(t)\|^2_{L^2(\RR^2)}$.
 \end{proposition}
 \begin{proof}
 The proof of the theorem is standard, we also give the proof for completeness.
 Taking the $L^2$-inner product  of $(u,E,B)$, we immediately have
 \begin{equation*}
 \frac12\dtd\|u(t)\|_{L^2(\RR^2)}^2+\int_0^t\|\nabla u(\tau)\|_{L^2(\RR^2)}^2\intd\tau=\int_{\RR^2}(j\times B)\cdot u\intd x,
 \end{equation*}
 \begin{equation*}
  \frac12\dtd\|E(t)\|_{L^2(\RR^2)}^2=\int_{\RR^2}\mathrm{curl}\,B\cdot E\intd x-\int_{\RR^2}j\cdot E\intd x,
 \end{equation*}
 and
 \begin{equation*}
  \frac12\dtd\|B(t)\|_{L^2(\RR^2)}^2=-\int_{\RR^2}\mathrm{curl}\,E\cdot B\intd x.
 \end{equation*}
 Note that
 \begin{equation*}
 \int_{\RR^2}\mathrm{curl}\,B\cdot E\intd x-\int_{\RR^2}\mathrm{curl}\,E\cdot B\intd x=0
 \end{equation*}
 and from the relation $j=E+u\times B$ that
  \begin{equation*}
  \begin{split}
 &\int_{\RR^2}(j\times B)\cdot u\intd x-\int_{\RR^2}j\cdot E\intd x\\
 =&-\int_{\RR^2}|j|^2\intd x+\int_{\RR^2}j\cdot (u\times B)\intd x+\int_{\RR^2}(j\times B)\cdot u\intd x\\
  =&-\int_{\RR^2}|j|^2\intd x.
 \end{split}
 \end{equation*}
 Collecting all these estimates, we readily have
 \begin{equation*}
 \begin{split}
 &\frac12\dtd\left(\|u(t)\|_{L^2(\RR^2)}^2+\|E(t)\|_{L^2(\RR^2)}^2+\|B(t)\|_{L^2(\RR^2)}^2\right)\\
 &+\int_0^t\|\nabla u(\tau)\|_{L^2(\RR^2)}^2\intd\tau +\int_0^t\|j(\tau)\|_{L^2(\RR^2)}^2\intd\tau=0.
 \end{split}
 \end{equation*}
 Integrating the above equality with respect to time $t$ yields the desired result \eqref{eq.L2estimate}.
 \end{proof}
 \begin{proposition}\label{prop-improve}
 Let $(u,B,E)$ be the smooth solution of problem \eqref{eq.MNS}. Then there exist a constant $C=C(t,\|(u_0,B_0,E_0)\|_{L^2(\RR^2)})>0$ such that
 \begin{equation*}
 \begin{split}
 &\|u\|_{{L}^\infty_tL^2(\RR^2)}^2+\big\|(E,B)\big\|^2_{\widetilde{L}^\infty_t\dot{B}^0_{2,2}(\RR^2)}+\int_0^t\|\nabla u(\tau)\|_{L^2(\RR^2)}^2\intd\tau\\
 &+\lambda_1\sup_{i\in\ZZ}2^{2i}\int_0^t\sum_{k\in\mathbb{Z}^2}\big\|\sqrt{\phi_{i,k}}u(\tau)\big\|^2_{L^2(\RR^2)}\intd\tau+\int_0^t\|j(\tau)\|_{L^2(\RR^2)}^2\intd\tau\leq C.
 \end{split}
 \end{equation*}
 \end{proposition}
 \begin{proof}
Multiplying the first equations of system \eqref{eq.MNS} by the cut-off function $\varphi_{i,k}u$, we have that
\begin{equation*}
\frac12\partial_t\left(\varphi_{i,k} u^2\right)+\frac12(u\cdot\nabla)\left(\varphi_{i,k} u^2\right)-\varphi_{i,k}u\Delta u
= \varphi_{i,k}u(j\times B)+(u\cdot\nabla\varphi_{i,k})u^{2}-\varphi_{i,k}u \cdot\nabla\pi.
\end{equation*}
Integrating the above equality with respect to space variable  $x$ over $\RR^2$ yields
\begin{equation*}
\begin{split}
&\frac12\dtd\big\|\sqrt{\varphi_{i,k}} u(t)\big\|_{L^2(\RR^2)}^2+\big\|\sqrt{\varphi_{i,k}}\nabla u(t)\big\|^2_{L^2(\RR^2)}+ 2\lambda_12^{2i} \big\|\sqrt{\varphi_{i,k}}u(t)\big\|^2_{L^2(\RR^2)}\\
=&\int_{\mathbb{R}^2}\varphi_{i,k}u\cdot(j\times B)\intd x+\frac12\int_{\mathbb{R}^2}(u\cdot\nabla\varphi_{i,k})|u|^2\intd x
 -\int_{\mathbb{R}^2}\varphi_{i,k}u\cdot\nabla\pi \intd x\\
 &-\frac12\int_{\partial B_{2^{-i}}(k)}|u|^{2}\nabla\varphi_{i,k}\cdot n\intd S.
\end{split}
\end{equation*}
Next, summing the above equality over $k\in\ZZ^2$ and integrating the resulting inequality in time $t$ provides
\begin{equation*}
\begin{split}
&\frac12\sum_{k\in\ZZ^{2}}\big\|\sqrt{\varphi_{i,k}} u(t)\big\|_{L^2(\RR^2)}^2+\int_0^t\sum_{k\in\ZZ^2}\big\|\sqrt{\varphi_{i,k}}\nabla u(\tau)\big\|^2_{L^2(\RR^2)}\intd\tau\\
&+2\lambda_12^{2i}\int_0^t\sum_{k\in\ZZ^2}\big\|\sqrt{\varphi_{i,k}}u(\tau)\big\|^2_{L^2(\RR^2)}\intd\tau\\
=&\frac12\sum_{k\in\ZZ^{2}}\big\|\sqrt{\varphi_{i,k}} u_0\big\|_{L^2(\RR^2)}^2+\int_0^t\sum_{k\in\ZZ^2}\int_{\mathbb{R}^2}\varphi_{i,k}u(j\times B)\intd x\mathrm{d}\tau+\int_0^t\sum_{k\in\ZZ^{2}}\int_{\mathbb{R}^2}(u\cdot\nabla\varphi_{i,k})u^2\intd x\mathrm{d}\tau\\
&-\int_0^t\sum_{k\in\ZZ^{2}}\int_{\mathbb{R}^2}\varphi_{i,k}u\cdot\nabla\pi \intd x\mathrm{d}\tau-\frac12\int_0^t\sum_{k\in\ZZ^2}\int_{\partial B_{2^{-i}}(k)}|u|^{2}\nabla\varphi_{i,k}\cdot n\intd S\intd\tau.
\end{split}
\end{equation*}
By the H\"older inequality and \eqref{eq.MlMu}, we easily find that
\begin{equation}\label{eq.j-b}
\begin{split}
\sum_{k\in\ZZ^2}\int_{\mathbb{R}^2}\varphi_{i,k}u\cdot(j\times B)\intd x
=&-\sum_{k\in\ZZ^2}\int_{\mathbb{R}^2}\varphi_{i,k}j\cdot( u\times B)\intd x\\
\leq&C \int_{\mathbb{R}^2}|u(x)||j(x)||B(x)| \intd x.
\end{split}
\end{equation}

By Lemma \ref{Lem-tri}, we see that
\begin{equation}\label{eq.est-jXb}
\begin{split}
&\sum_{k\in\mathbb{Z}^{2}}\int_0^t\int_{\mathbb{R}^2}\varphi_{i,k}u\cdot(j\times B)\intd x\mathrm{d}\tau\\
\leq&C \int_0^t\int_{\mathbb{R}^2}|u(t,x)||j(t,x)||B(t,x)| \intd x\mathrm{d}\tau\\
\leq&C\|j\|_{L^2_tL^2(\RR^2)}\big\|\nabla|u|\big\|_{L^2_tL^2(\RR^2)}\|B\|_{L^\infty_tL^2(\RR^2)}\\
&+CN\sup_{i'\in\ZZ}\big\|S_{i'-1}|u|\big\|_{L^2_tL^\infty(\RR^2)}\|j\|_{L^2_tL^2(\RR^2)}\|B\|_{L^\infty_tL^2(\RR^2)}\\
&+C\sup_{i'\in\ZZ}\big\|S_{i'-1}|u|\big\|_{L^2_tL^\infty(\RR^2)}\|B\|_{L^\infty_t\dot{B}^0_{2,2}(\RR^2)}\Big(\sum_{|i|\geq N}\big\|\dot{\Delta}_ij\big\|^{2}_{L^2_tL^2(\RR^2)}\Big)^{\frac12}.
\end{split}
\end{equation}
For the second term, the H\"older inequality and the interpolation inequality allow us to conclude that
\begin{equation*}
\begin{split}
\sum_{k\in\ZZ^{2}}\int_{\mathbb{R}^2}(u\cdot\nabla\varphi_{i,k})|u|^2\intd x=&-\sum_{k\in\ZZ^{2}}\int_{\mathbb{R}^2}\varphi_{i,k}(u\cdot\nabla)u\cdot u\intd x\\
\leq&C\int_{\mathbb{R}^2}|u(x)|^2|\nabla u(x)|  \intd x\\
\leq&C\|u\|^2_{L^4(\RR^2)}\|\nabla u\|_{L^2(\RR^2)}\\
\leq&C\|u\|_{L^2(\RR^2)}\|\nabla u\|^2_{L^2(\RR^2)}.
\end{split}
\end{equation*}
From this, it follows that
\begin{equation}\label{eq.convection}
\sum_{k\in\ZZ^{2}}\int_0^t\int_{\mathbb{R}^2}(u\cdot\nabla\varphi_{i,k})|u|^2\intd x\mathrm{d}\tau
\leq C\int_0^t\|u(\tau)\|_{L^2(\RR^2)}\|\nabla u(\tau)\|^2_{L^2(\RR^2)}\intd\tau.
\end{equation}
Now, we turn to show the term involving the pressure. Since $\Div u=0$, the pressure can be expressed by
\begin{equation*}
\pi= -\frac{\Div}{\Delta}\big((u\cdot\nabla)u\big)+\Big(\frac{\Div}{-\Delta}\Big)\big(j\times B\big)
:= \pi_1+\pi_2.
\end{equation*}
Therefore, we have
\begin{equation*}
\int_{\mathbb{R}^2}\varphi_{i,k}(u\cdot\nabla)\pi \intd x=\int_{\mathbb{R}^2}\varphi_{i,k}(u\cdot\nabla)\pi_1 \intd x+\int_{\mathbb{R}^2}\varphi_{i,k}(u\cdot\nabla)\pi_2 \intd x.
\end{equation*}
To bound the second integral in the right side of the above equality, we need to resort to the following lemma.
\begin{lemma}\label{Lemma-p-II}
For each $\varepsilon>0,$ there exist a absolute constant $C>0$ such that
\begin{equation}\label{est-local-pII}
\begin{split}
&\sum_{k\in\ZZ^2}\int_{\mathbb{R}^2}\varphi_{i,k}(u\cdot\nabla)\pi_2 \intd x\\
\leq& C\|j\|^2_{L^2(\RR^2)}\|B\|^2_{L^2(\RR^2)}+\varepsilon\Big(\|\nabla u \|_{L^2(\RR^2)}^2+\lambda_12^{2i}\sum_{k\in\ZZ^2}\big\| \sqrt{\varphi_{i,k}}u \big\|_{L^2(\RR^2)}^2\Big).
\end{split}
\end{equation}
\end{lemma}
\begin{proof}[Proof of Lemma \ref{Lemma-p-II}]
Firstly, we split the integral into the following two parts:
\begin{equation*}
\begin{split}
\int_{\mathbb{R}^2}\varphi_{i,k}(u\cdot\nabla)\pi_2 \intd x
=&\sum_{|\widetilde{k}-k|\leq5}\int_{\RR^2}\nabla\Big(\frac{\Div}{-\Delta}\Big)\big(\phi_{i,\widetilde{k}}(j\times B)\big)\cdot(\varphi_{i,k}u)\intd x\\
&+\sum_{|\widetilde{k}-k|>5}\int_{\RR^2}\nabla\Big(\frac{\Div}{-\Delta}\Big)\big( \phi_{i,\widetilde{k}} (j\times B)\big)\cdot(\varphi_{i,k}u)\intd x.
\end{split}
\end{equation*}
For the first term in the above equality, by integrating by parts and using the H\"older inequality, we have
\begin{equation*}
\begin{split}
&\int_{\RR^2}\nabla\Big(\frac{\Div}{-\Delta}\Big)\big(\phi_{i,\widetilde{k}}(j\times B)\big)\cdot(\varphi_{i,k}u)\intd x\\
=&-2\sum_{|k'-k|\leq5}\int_{\RR^2} \Big(\frac{\Div}{-\Delta}\Big)\big(\phi_{i,\widetilde{k}}(j\times B)\big)\phi_{i,k'}(u\cdot\nabla)\varphi_{i,k} \intd x\\
\leq&2\sum_{|k'-k|\leq5}\Big\|\Big(\frac{\Div}{-\Delta}\Big)\big(\phi_{i,\widetilde{k}}(j\times B)\big)\Big\|_{\dot{B}^{0}_{2,\infty}(\RR^2)}\left\| \phi_{i,k'} u\cdot\nabla \varphi_{i,k} \right\|_{\dot{B}_{2,1}^0(\RR^2)}.
\end{split}
\end{equation*}
By the H\"older inequality, we easily find that
\begin{equation*}
\begin{split}
\Big\|\Big(\frac{\Div}{-\Delta}\Big)\big(\phi_{i,\widetilde{k}}(j\times B)\big)\Big\|_{\dot{B}^{0}_{2,\infty}(\RR^2)}
\leq&C\big\|\phi_{i,\widetilde{k}}(j\times B)\big\|_{\dot{B}^{-1}_{2,\infty}(\RR^2)}\\
\leq& C\big\|\phi_{i,\widetilde{k}}(j\times B)\big\|_{L^1(\RR^2)}\\
\leq&C\big\|\sqrt{\phi_{i,\widetilde{k}}}j\big\|_{L^2(\RR^2)}\big\|\sqrt{\phi_{i,\widetilde{k}}}B\big\|_{L^2(\RR^2)}.
\end{split}
\end{equation*}
Thanks to the Bony paraproduct decomposition, we have
\begin{equation*}
\begin{split}
\left( \phi_{i,k'} u\right)\cdot\nabla \varphi_{i,k} =&\sum_{\ell=1}^2\dot{T}_{\partial_\ell\varphi_{i,k}}
\left(\phi_{i,k'}u^{\ell}\right)+\sum_{\ell=1}^2\dot{T}_{ \phi_{i,k'} u^{\ell}}\left(\partial_\ell \varphi_{i,k} \right)\\
&+\sum_{\ell=1}^2\dot{R}\big(\partial_\ell \varphi_{i,k} , \phi_{i,k'} u^{\ell}\big).
\end{split}
\end{equation*}
A simple calculation allows us to conclude that
\begin{equation*}
\begin{split}
\big\|\dot{T}_{\partial_\ell\varphi_{i,k}}( \phi_{i,k'}u^{\ell})\big\|_{\dot{B}_{2,1}^0(\RR^2)}
\leq&C\sum_{i'\in\mathbb{Z}}\|\dot{S}_{i'-1}\nabla \varphi_{i,k} \|_{L^\infty(\RR^2)}\big\|\dot{\Delta}_{i'}( \phi_{i,k'} u)\big\|_{L^2(\RR^2)}\\
\leq&C\sum_{i'\in\mathbb{Z}}2^{-i'}\|\dot{S}_{i'-1}\nabla \varphi_{i,k} \|_{L^\infty(\RR^2)}2^{i'}\big\|\dot{\Delta}_{i'}(\phi_{i,k'}u)\big\|_{L^2(\RR^2)}\\
\leq&C\big\|\nabla\varphi_{i,k}\big\|_{\dot{B}^{-1}_{\infty,2}(\RR^2)}\big\|\nabla(\phi_{i,k'}u)\big\|_{L^2(\RR^2)}\\
\leq&C\big\|\nabla\varphi_{i,k}\big\|_{L^2(\RR^2)}\big\|\nabla(\phi_{i,k'}u)\big\|_{L^2(\RR^2)}.
\end{split}
\end{equation*}
In the similar fashion as above, $\dot{T}_{\phi_{i,k'}u^{\ell}}\partial_\ell\varphi_{i,k}$ can be bounded as follows:
\begin{equation*}
\begin{split}
&\big\|\dot{T}_{ \phi_{i,k}u^{\ell}} \partial_\ell \varphi_{i,k} \big\|_{\dot{B}^0_{2,1}(\RR^2)}\\
\leq&C\sum_{i'\in\mathbb{Z}}\big\|\dot{S}_{i'-1}(\phi_{i,k'}u)\big\|_{L^\infty(\RR^2)}\big\|\dot{\Delta}_{i'}(\nabla \varphi_{i,k} )\big\|_{L^2(\RR^2)}\\
\leq&C\sum_{i'\in\mathbb{Z}}2^{-i'}\big\|\dot{S}_{i'-1}(\phi_{i,k'}u)\big\|_{L^\infty(\RR^2)}2^{i'}\big\|\dot{\Delta}_{i'}(\nabla \varphi_{i,k})\big\|_{L^2(\RR^2)}\\
\leq&C\big\| \phi_{i,k'}u\big\|_{\dot{B}^{-1}_{\infty,2}(\RR^2)}\big\|\Delta \varphi_{i,k} \big\|_{L^2(\RR^2)}
\leq C\big\|\Delta\varphi_{i,k}\big\|_{L^2(\RR^2)}\big\|\phi_{i,k'}u\big\|_{L^2(\RR^2)}.
\end{split}
\end{equation*}
We turn to show the remainder term $\dot{R}(\partial_{\ell} \varphi_{i,k},  \phi_{i,k'} u^{\ell} )$. We observe that
\begin{equation*}
\begin{split}
\big\|\dot{R}(\partial_{\ell}\varphi_{i,k},\phi_{i,k'}u^{\ell})\big\|_{\dot{B}_{2,1}^0(\RR^2)}\leq&C\big\|\dot{R}(\nabla\varphi_{i,k},\phi_{i,k'}u)\big\|_{\dot{B}_{1,1}^1(\RR^2)}\\
\leq&C\sum_{i'\in\mathbb{Z}}2^{i'}\big\|\dot{\Delta}_{i'}\nabla\varphi_{i,k}\big\|_{L^2(\RR^2)}\big\|\widetilde{\dot{\Delta}}_{i'}(\phi_{i,k'}u)\big\|_{L^2(\RR^2)}\\
\leq&C \big\|\nabla\varphi_{i,k}\big\|_{L^2(\RR^2)}\big\|\phi_{i,k'}u\big\|_{\dot{B}^{1}_{2,2}(\RR^2)}\\
\leq&C \big\|\nabla\varphi_{i,k}\big\|_{L^2(\RR^2)}\big\|\nabla(\phi_{i,k'}u)\big\|_{L^2(\RR^2)}.
\end{split}
\end{equation*}
Hence, we have
\begin{equation}\label{eq.pressure-jb-low}
\begin{split}
&\int_{\RR^2}\nabla\Big(\frac{\Div}{-\Delta}\Big)\big(\phi_{i,\widetilde{k}}j\times B\big)\cdot(\varphi_{i,k}u)\intd x\\
\leq&C\sum_{|k'-k|\leq5}\big\|\sqrt{\phi_{i,\widetilde{k}}}j\big\|_{L^2(\RR^2)}\big\|\sqrt{\phi_{i,\widetilde{k}}}B\big\|_{L^2(\RR^2)}\big\|\nabla(\varphi_{i,k'}u)\big\|_{L^2(\RR^2)}\\
\leq&C\big\|\sqrt{\phi_{i,\widetilde{k}}}j\big\|^2_{L^2(\RR^2)}\big\|\sqrt{\phi_{i,\widetilde{k}}}B\big\|^2_{L^2(\RR^2)}+\varepsilon \sum_{|k'-k|\leq5} \big\|\nabla(\phi_{i,k'}u)\big\|^2_{L^2(\RR^2)}.
\end{split}
\end{equation}
Now, we turn to bound the integral term
$$
\sum_{|\widetilde{k}-k|>5}\int_{\RR^2}\nabla\Big(\frac{\Div}{-\Delta}\Big)\big( \phi_{i,\widetilde{k}} (j\times B)\big)\cdot(\varphi_{i,k}u)\intd x.$$
The term $\nabla\big(\frac{\Div}{-\Delta}\big)\big(\phi_{i,\widetilde{k}}(j\times B)\big)$ can be rewritten as
\begin{equation*}
\nabla\Big(\frac{\Div}{-\Delta}\Big)\big( \phi_{i,\widetilde{k}}(j\times B)\big)= \int_{\RR^2}K(x-z) \phi_{i,\widetilde{k}}(z) \big(j\times B\big)(z)\intd z,
\end{equation*}
where the kernel $K(x)$ satisfies $|K|\leq c\frac{1}{|x|^2}$.

Hence, the above equality allows us to write
\begin{equation*}
\begin{split}
&\int_{\RR^2}\sqrt{\varphi_{i,k}}\Big(\frac{\Div}{-\Delta}\Big) \big(\phi_{i,\widetilde{k}}(j\times B)\big)\cdot(\sqrt{\varphi_{i,k}}u)\intd x\\
= &
\int_{\RR^2}\sqrt{\varphi_{i,k}}K \ast \big(\phi_{i,\widetilde{k}}(j\times B)\big)\cdot( \sqrt{\varphi_{i,k}}u)\intd x.
\end{split}
\end{equation*}
Since $|\widetilde{k}-k|> 5$, by the H\"older inequality, we have
\begin{equation*}
\begin{split}
&\int_{\RR^2}\sqrt{\varphi_{i,k}}K\ast\big |\phi_{i,\widetilde{k}}(j\times B)\big |\sqrt{\varphi_{i,k}}|u|\intd x\\
=&\int_{\RR^2}\sqrt{\varphi_{i,k}}K_{k\widetilde{k}}\ast \big|\phi_{i,\widetilde{k}}(j\times B)\big |\sqrt{\varphi_{i,k}}|u|\intd x\\
\leq&\big\|\sqrt{\varphi_{i,k}}\big(K_{k\widetilde{k}}\ast \big|\phi_{i,\widetilde{k}}(j\times B)\big |\big)\big\|_{L^\infty(\RR^2)}\big\|\sqrt{\varphi_{i,k}}u\big\|_{L^1(\RR^2)}\\
\leq&C2^{-i}\big\|\sqrt{\varphi_{i,k}}\big(K_{k\widetilde{k}}\ast \big|\phi_{i,\widetilde{k}}(j\times B)\big |\big)\big\|_{L^\infty(\RR^2)}\big\|\sqrt{\varphi_{i,k}}u\big\|_{L^2(\RR^2)},
\end{split}
\end{equation*}
where $K_{k\widetilde{k}}=c\frac{1}{|x|^2}\chi_{B^{c}_{2^{-(1+i)}|\widetilde{k}-k|}(0)}$ deduced from the support property of $\varphi_{i,k}$ and $\phi_{i,\widetilde{k}}$.

We get from the Young inequality that
\begin{equation*}
\begin{split}
\big\|\sqrt{\varphi_{i,k}}\big(K_{k\widetilde{k}}\ast |\phi_{i,\widetilde{k}}(j\times B)\big |\big)\big\|_{L^\infty(\RR^2)}\leq&\|K_{k\widetilde{k}}\|_{L^\infty(\RR^2)}\big\|\phi_{i,\widetilde{k}}(j\times B)\big\|_{L^1(\RR^2)}\\
\leq&\frac{C2^{2i}}{|\widetilde{k}-k|^2}\big\|\sqrt{\phi_{i,\widetilde{k}}}j\big\|_{L^2(\RR^2)}\big\|\sqrt{\phi_{i,\widetilde{k}}}B\big\|_{L^2(\RR^2)}.
\end{split}
\end{equation*}
Therefore, in virtue of the discrete Young inequality and Cauchy-Schwarz inequality, we obtain
%
%
\begin{equation*}
\begin{split}
&\sum_{k\in\ZZ^2}\sum_{|k'-k|>5}\int_{\RR^2}\nabla\Big(\frac{\Div}{-\Delta}\Big)\big( \phi_{i,\widetilde{k}} (j\times B)\big)\cdot(\varphi_{i,k}u)\intd x\\
\leq& C\sum_{k\in\ZZ^2}2^{i}\sum_{|\widetilde{k}-k|>5}\frac{1}{|\widetilde{k}-k|^{2}}\big\|\sqrt{\phi_{i,\widetilde{k}}}j\big\|_{L^2(\RR^2)}\big\|\sqrt{\phi_{i,\widetilde{k}}}B\big\|_{L^2(\RR^2)}\big\|\sqrt{\varphi_{i,k}}u\big\|_{L^2(\RR^2)}\\
\leq&C\|j\|^2_{L^2(\RR^2)}\|B\|^2_{L^2(\RR^2)}+\varepsilon\lambda_1\sum_{k\in\ZZ^2}2^{2i}\big\| \sqrt{\varphi_{i,k}} u \big\|_{L^2(\RR^2)}^2.
\end{split}
\end{equation*}
This estimate together with \eqref{eq.pressure-jb-low} yields
\begin{equation}\label{eq.Pressure-jb}
\begin{split}
&\sum_{k\in\ZZ^2}\int_{\RR^2}\nabla\Big(\frac{\Div}{-\Delta}\Big)\big(j\times B\big)\cdot(\varphi_{i,k}u)\intd x\\
\leq& C\|j\|^2_{L^2(\RR^2)}\|B\|^2_{L^2(\RR^2)}+\varepsilon\sum_{k\in\ZZ^2}\Big(\lambda_12^{2i}\sum_{|\tilde{k}-k|\leq5}\| \phi_{i,\tilde{k}} u\|_{L^2(\RR^2)}^2+\big\| \sqrt{\varphi_{i,k}} \nabla u \big\|_{L^2(\RR^2)}^2\Big)\\
\leq& C\|j\|^2_{L^2(\RR^2)}\|B\|^2_{L^2(\RR^2)}+\varepsilon\sum_{k\in\ZZ^2}\Big(\lambda_12^{2i}\| \phi_{i, k } u\|_{L^2(\RR^2)}^2+\big\| \sqrt{\phi_{i,k}} \nabla u \big\|_{L^2(\RR^2)}^2\Big).
\end{split}
\end{equation}
Hence, we end the proof of Lemma \ref{Lemma-p-II}.
\end{proof}
Now, we need to bound the integral
$\int_{\mathbb{R}^2}\varphi_{i,k}(u\cdot\nabla)\pi_1 \intd x$ which is contained in the lemma below.
\begin{lemma}\label{Lemma-p-I}
For each $\varepsilon>0,$ there exist a absolute constant $C>0$ such that
\begin{equation}\label{est-p-I}
\sum_{k\in\ZZ^2}\int_{\mathbb{R}^2}\varphi_{i,k}(u\cdot\nabla)\pi_1 \intd x\leq C  \|u\|^2_{L^2(\RR^2)}\|\nabla u\|^2_{L^2(\RR^2)}+\varepsilon\lambda_1\sum_{k\in\ZZ^2}2^{2i}\big\|\sqrt{\phi_{i,k}}u\big\|^2_{L^2(\RR^2)}.
\end{equation}
\end{lemma}
\begin{proof}[Proof of Lemma \ref{Lemma-p-I}]
We see that
\[\int_{\mathbb{R}^2}\varphi_{i,k}(u\cdot\nabla)\pi_1 \intd x=\int_{\RR^2}\varphi_{i,k}u\cdot\nabla\Big(\frac{\Div}{-\Delta}\Big)\big((u\cdot\nabla) u\big)\intd x.\]
One can write
\begin{equation*}
\begin{split}
\int_{\RR^2}\varphi_{i,k}u\cdot\nabla\Big(\frac{\Div}{-\Delta}\Big)\big((u\cdot\nabla) u\big)\intd x
=&\sum_{|k'-k|\leq5}\int_{\RR^2}\varphi_{i,k}u\cdot\nabla\Big(\frac{\Div\Div}{-\Delta}\Big)\big(\phi_{i,k'}(u\otimes u)\big)\intd x\\
&+\sum_{|k'-k|>5}\int_{\RR^2}\varphi_{i,k}u\cdot\nabla\Big(\frac{\Div\Div}{-\Delta}\Big)\big(\phi_{i,k'}(u\otimes u)\big)\intd x.
\end{split}
\end{equation*}
Integration by parts leads to
\begin{equation*}
\begin{split}
&\int_{\RR^2}\varphi_{i,k}u\cdot\nabla\Big(\frac{\Div\Div}{-\Delta}\Big)\big(\phi_{i,k'}(u\otimes u)\big)\intd x\\
=&-2\int_{\RR^2}\Big(\frac{\Div\Div}{-\Delta}\Big)\big(\phi_{i,k'}(u\otimes u)\big)u\cdot\nabla\varphi_{i,k}\intd x\\
=&-2\sum_{|\tilde{k}-k|\leq5}\int_{\RR^2}\Big(\frac{\Div\Div}{-\Delta}\Big)\big(\phi_{i,k'}(u\otimes u)\big)\phi_{i,\tilde{k}}u\cdot\nabla\varphi_{i,k}\intd x.
\end{split}
\end{equation*}
Moreover, by the H\"older inequality, we immediately have
\begin{equation}\label{eq.p-I-L}
\begin{split}
&\int_{\RR^2}\Big(\frac{\Div\Div}{-\Delta}\Big)\big(\phi_{i,k'}(u\otimes u)\big) \phi_{i,\tilde{k}}u\cdot\nabla\varphi_{i,k}\intd x\\
\leq&\Big\|\Big(\frac{\Div\Div}{-\Delta}\Big)\big(\phi_{i,k'}(u\otimes u)\big)\Big\|_{L^2(\RR^2)}\big\|\nabla\varphi_{i,k}\big\|_{L^\infty(\RR^2)}\big\| \varphi_{i,\tilde{k}} u\big\|_{L^2(\RR^2)}.
\end{split}
\end{equation}
The H\"older inequality and the interpolation theorem give
\begin{equation*}
\begin{split}
\sum_{k\in\ZZ^2}\Big\|\Big(\frac{\Div\Div}{-\Delta}\Big)\big(\phi_{i,k'}(u\otimes u)\big)\Big\|^{2}_{L^2(\RR^2)}\leq&C\sum_{k\in\ZZ^2}\big\|\phi_{i,k'}(u\otimes u)\big\|^{2}_{L^2(\RR^2)}\\
\leq&C\|u\|_{L^4(\RR^2)}^4\leq C\|u\|^{2}_{L^2(\RR^2)}\|\nabla u\|^{2}_{L^2(\RR^2)}.
\end{split}
\end{equation*}
Plugging this estimate in \eqref{eq.p-I-L} provides
\begin{equation}\label{eq.-p-I-LL}
\begin{split}
 &\sum_{k\in\ZZ^2}\int_{\RR^2}\Big(\frac{\Div\Div}{-\Delta}\Big)\big(\phi_{i,k'}(u\otimes u)\big)u\cdot\nabla \varphi_{i,k}\intd x\\
\leq&C\sum_{k\in\ZZ^2} 2^i\Big\|\Big(\frac{\Div\Div}{-\Delta}\Big)\big(\phi_{i,k'}(u\otimes u)\big)\Big\|_{L^2(\RR^2)}\sum_{|\tilde{k}-k|\leq5}\big\| \phi_{i,\tilde{k} }u\big\|_{L^2(\RR^2)}\\
\leq & C \sum_{k\in\ZZ^2}\|u\|^2_{L^2(\RR^2)}\|\nabla u\|^2_{L^2(\RR^2)}+\varepsilon\lambda_{1}\sum_{k\in\ZZ^2}2^{2i}\big\|\sqrt{\phi_{i,k}}u\big\|^2_{L^2(\RR^2)}.
\end{split}
\end{equation}
On the other hand, we observe that
\begin{equation*}
\begin{split}
\nabla\Big(\frac{\Div\Div}{-\Delta}\Big)\big(\phi_{i,k'}(u\otimes u)\big)=G\ast\big(\phi_{i,k'}(u\otimes u)\big),
\end{split}
\end{equation*}
where $G(x)$ satisfies $|G(x)|\leq\frac{c}{|x|^{3}}$.

It follows that
\begin{equation*}
\begin{split}
\int_{\RR^2}\nabla\Big(\frac{\Div\Div}{-\Delta}\Big)\big(\phi_{i,k'}(u\otimes u)\big)\cdot(\varphi_{i,k}u)\intd x=\int_{\RR^2} \sqrt{\varphi_{i,k}}G\ast\big(\phi_{i,k'}(u\otimes u)\big)\cdot(\sqrt{\varphi_{i,k}}u)\intd x.
\end{split}
\end{equation*}
Since $|k'-k|>5$, we find that the supports of  $\varphi_{i,k}$ and $\phi_{i,k'}$ are disjoint and thus we have
\[\sqrt{\varphi_{i,k}}G\ast\big(\phi_{i,k'}(u\otimes u)\big)=\sqrt{\varphi_{i,k}}(G\chi^c_{[0,2^{-(1+i)}|k'-k|]})\ast\big(\phi_{i,k'}(u\otimes u)\big).\]
Letting $G_{kk'}:=G\chi^c_{[0,|k'-k|2^{-(1+i)}]}$ and using the H\"older inequality, we immediately get
\begin{equation*}
\begin{split}
&\int_{\RR^2}\nabla\Big(\frac{\Div\Div}{-\Delta}\Big)\big(\phi_{i,k'}(u\otimes u)\big)\cdot(\varphi_{i,k}u)\intd x\\
=&\int_{\RR^2} \sqrt{\varphi_{i,k}}G_{ik}\ast\big(\phi_{i,k'}(u\otimes u)\big)\cdot(\sqrt{\varphi_{i,k}}u)\intd x\\
\leq&\big\|\sqrt{\varphi_{i,k}}\big\|_{L^\infty}\big\|G_{kk'}\ast\big(\phi_{i,k'}(u\otimes u)\big)\big\|_{L^\infty(\RR^2)}\big\|\sqrt{\varphi_{i,k}}u\big\|_{L^1(\RR^2)}\\
\leq&C2^{-i}\big\|G_{kk'}\ast\big(\phi_{i,k'}(u\otimes u)\big)\big\|_{L^\infty(\RR^2)}\big\|\sqrt{\varphi_{i,k}}u\big\|_{L^2(\RR^2)}.
\end{split}
\end{equation*}
By the discrete Young inequality, we have
\begin{equation*}
\begin{split}
\big\|G_{kk'}\ast\big(\phi_{i,k'}(u\otimes u)\big)\big\|_{L^\infty(\RR^2)}\leq&\|G_{kk'}\|_{L^2(\RR^2)}\big\|\phi_{i,k'}(u\otimes u)\big\|_{L^2(\RR^2)}\\
\leq&C\frac{2^{2i}}{|k'-k|^3}\big\|\sqrt{\phi_{i,k'}}(u\otimes u)\big\|_{L^2(\RR^2)}.
\end{split}
\end{equation*}
This estimate enables us to infer that
\begin{equation*}
\begin{split}
&\sum_{k\in\ZZ^2}\sum_{|k'-k|>5}\int_{\RR^2}\nabla\Big(\frac{\Div\Div}{-\Delta}\Big)\big(\phi_{i,k'}(u\otimes u)\big)\cdot(\varphi_{i,k}u)\intd x\\
\leq&C\sum_{k\in\ZZ^2}2^{i}\big\|\sqrt{\varphi_{i,k}}u\big\|_{L^2(\RR^2)}\sum_{|k'- k|>5}\frac{1}{|k'-k|^3}\big\|\sqrt{\phi_{i,k'}}(u\otimes u)\big\|_{L^2(\RR^2)}\\
\leq&C2^{i}\Big(\sum_{k\in\ZZ^2}\big\|\sqrt{\varphi_{i,k}}u\big\|^2_{L^2(\RR^2)} \Big)^{\frac12}\|u\otimes u \|_{L^2(\RR^2)}\\
\leq& C2^{i}\Big(\sum_{k\in\ZZ^2}\big\|\sqrt{\varphi_{i,k}}u\big\|^2_{L^2(\RR^2)} \Big)^{\frac12}\|u\|_{L^2(\RR^2)}\|\nabla u\|_{L^2(\RR^2)}\\
\leq&C \|u\|^2_{L^2(\RR^2)}\|\nabla u\|^2_{L^2(\RR^2)}+\varepsilon\lambda_1\sum_{k\in\ZZ^2}2^{2i}\big\|\sqrt{\phi_{i,k}}u\big\|^2_{L^2(\RR^2)}.
\end{split}
\end{equation*}

This estimate together with estimate \eqref{eq.-p-I-LL} gives the desired result in Lemma \ref{Lemma-p-I}.
\end{proof}
From Lemma \ref{Lemma-p-II} and Lemma \ref{Lemma-p-I}, we know that
\begin{equation}\label{eatimate-p}
\begin{split}
&\sum_{k\in\ZZ^2}\int_0^t \int_{\mathbb{R}^2}\varphi_{i,k}(u\cdot\nabla)\pi \intd x \mathrm{d}\tau\\
\leq& C\int_0^t\|j(\tau)\|^2_{L^2(\RR^2)}\|B(\tau)\|^2_{L^2(\RR^2)}\intd\tau+C\int_0^t \|u(\tau)\|^2_{L^2(\RR^2)}\|\nabla u(\tau)\|^2_{L^2(\RR^2)}\intd\tau\\&+\varepsilon\sum_{k\in\ZZ^2}\int_0^t\Big(\lambda_12^{2i}\sum_{|k''-k|\leq5}\big\|\sqrt{\phi_{i,k''}} u(\tau) \big\|_{L^2(\RR^2)}^2+\big\| \sqrt{\varphi_{i,k}} \nabla u(\tau) \big\|_{L^2(\RR^2)}^2\Big)\intd\tau.
\end{split}
\end{equation}
Repeating the same argument as used in \eqref{eq.boundary-final}, we can show that
\begin{equation}\label{eq.ubun}
-\frac12\int_{\partial B_{2^{-i}}(k)}|u|^{2}\nabla\varphi_{i,k}\cdot n\intd S\leq C\sum_{|k'-k|\leq2}\left(\big\|\sqrt{\phi_{i,k'}}u\big\|^{2}_{L^2(\RR^2)}+\big\|\sqrt{\phi_{i,k'}}\nabla u\big\|^{2}_{L^2(\RR^2)}\right).
\end{equation}
Combining \eqref{eq.convection}, \eqref{eq.Pressure-jb} \eqref{eatimate-p} and \eqref{eq.ubun} gives
\begin{equation}\label{eq.est-u}
\begin{split}
&\sum_{k\in\ZZ^2}\big\|\sqrt{\varphi_{i,k}} u(t)\big\|_{L^2(\RR^2)}^2+2\sum_{k\in\ZZ^2}\int_0^t\big\|\sqrt{\varphi_{i,k}}(\nabla u)(\tau)\big\|^2_{L^2(\RR^2)}\intd\tau\\
&+ 2\lambda_12^{2i}\sum_{k\in\ZZ^2} \int_0^t\big\|\sqrt{\varphi_{i,k}}u(\tau)\big\|^2_{L^2(\RR^2)}\intd\tau\\
\leq&\sum_{k\in\ZZ^2}\big\|\sqrt{\varphi_{i,k}} u_0\big\|_{L^2(\RR^2)}^2+\sum_{k\in\ZZ^2}\int_0^t\int_{\mathbb{R}^2}\varphi_{i,k}u\cdot(j\times B)\intd x\intd\tau\\
&+ C\int_0^t\|u(\tau)\|_{L^2(\RR^2)}\|\nabla u(\tau)\|^2_{L^2(\RR^2)}\intd\tau+C\sum_{k\in\ZZ^2}\sum_{|k'-k|\leq2}\int_0^t\|\sqrt{\phi_{i,k}}u(\tau)\big\|^{2}_{L^2(\RR^2)}\intd\tau\\
&+ C\int_0^t\|j(\tau)\|^2_{L^2(\RR^2)}\|B(\tau)\|^2_{L^2(\RR^2)}\intd\tau+C\int_0^t\|u(\tau)\|^2_{L^2(\RR^2)}\|\nabla u(\tau)\|^2_{L^2(\RR^2)}\intd\tau\\&+\varepsilon\sum_{k\in\ZZ^2}\int_0^t\Big(\lambda_12^{2i}\sum_{|\tilde{k}-k|\leq5}\big\|\phi_{i,\tilde{k}} u(\tau) \big\|_{L^2(\RR^2)}^2+\big\| \sqrt{\varphi_{i,k}} \nabla u(\tau) \big\|_{L^2(\RR^2)}^2\Big)\intd\tau\\
&+C\sum_{k\in\ZZ^2}\sum_{|k'-k|\leq2}\int_0^t\|\sqrt{\phi_{i,k}}\nabla u(\tau)\big\|^{2}_{L^2(\RR^2)}\intd\tau.
\end{split}
\end{equation}
Inserting estimate \eqref{eq.est-jXb} into the above estimate, we readily have
\begin{equation}\label{eq.u-improve}
\begin{split}
& \|u(t)\|_{L^2(\RR^2)}^2 +\int_0^t\|\nabla u(\tau)\|_{L^2(\RR^2)}^2\intd\tau +\lambda_{1}\sup_{i\in\ZZ}2^{2i}\int_0^t\sum_{k\in\mathbb{Z}^2}\big\|\sqrt{\phi_{i,k}}u(\tau)\big\|^2_{L^2(\RR^2)}\intd\tau\\
 \leq&\|u_0\|_{L^2(\RR^2)}^2 + C\big(1+\|u\|^2_{L^\infty_tL^2(\RR^2)}\big)\|\nabla u\|^2_{L^2_tL^2(\RR^2)}+C\|B\|^2_{L^\infty_tL^2(\RR^2)}\|j\|^2_{L^2_tL^2(\RR^2)}\\
 &+C\|j\|_{L^2_tL^2(\RR^2)}\big\|\nabla u\big\|_{L^2_tL^2(\RR^2)}\|B\|_{L^\infty_tL^2(\RR^2)}\\
 &+C\sup_p\big\|S_{p-1}|u|\big\|_{L^2_tL^\infty(\RR^2)}\Big(\sum_{|q|\geq N}\big\|\dot{\Delta}_qj\big\|^{2}_{L^2_tL^2(\RR^2)}\Big)^{\frac12}\|B\|_{L^\infty_t\dot{B}^0_{2,2}(\RR^2)}\\
&+CN\sup_p\big\|S_{p-1}|u|\big\|_{L^2_tL^\infty(\RR^2)}\|j\|_{L^2_tL^2(\RR^2)}\|B\|_{L^\infty_tL^2(\RR^2)}\\
&+\varepsilon\lambda_12^{2i}\sum_{\tilde{k}\in\mathbb{Z}^2}\sum_{|\tilde{k}-k'|<3}\int_0^t\big\|\sqrt{\phi_{i,k'}}u(\tau)\big\|^2_{L^2(\RR^2)}\intd\tau.
 \end{split}
\end{equation}
Moreover, by estimate \eqref{eq.L2estimate}, we have
\begin{equation}\label{eq.est-F-u}
\begin{split}
 &\|u(t)\|_{L^2(\RR^2)}^2 +\int_0^t\|\nabla u(\tau)\|_{L^2(\RR^2)}^2\intd\tau +\lambda_{1}\sup_{i\in\ZZ}2^{2i}\int_0^t\sum_{k\in\mathbb{Z}^2}\big\|\sqrt{\phi_{i,k}}u(\tau)\big\|^2_{L^2(\RR^2)}\intd\tau\\
 \leq&Ct+C\sup_p\big\|S_{p-1}|u|\big\|_{L^2_tL^\infty(\RR^2)}\Big(\sum_{|q|\geq N}\big\|\dot{\Delta}_qj\big\|^{2}_{L^2_tL^2(\RR^2)}\Big)^{\frac12}\|B\|_{\widetilde{L}^\infty_t\dot{B}^0_{2,2}(\RR^2)}
\\& +CN\sup_p\big\|S_{p-1}|u|\big\|_{L^2_tL^\infty(\RR^2)}.
\end{split}
\end{equation}
where the positive constant $C$ only depends on the initial data, independent of $t$.

Next, applying $\dot{\Delta}_q$ to the second equation and the third equation, respectively, and taking $L^2$-norm of the resulting equations,  we have
\begin{equation}\label{eq.est-E}
\begin{split}
&\sup_{\tau\in[0,t]}\big\|\dot{\Delta}_qE(\tau)\big\|_{L^2(\RR^2)}^2+\sup_{\tau\in[0,t]}\big\|\dot{\Delta}_qB(\tau)\big\|_{L^2(\RR^2)}^2+\int_0^t\big\|\dot{\Delta}_qj(\tau)\big\|_{L^2(\RR^2)}^2\intd\tau\\
\leq &\big\|\dot{\Delta}_qE_0\big\|_{L^2(\RR^2)}^2+\big\|\dot{\Delta}_qB_0\big\|_{L^2(\RR^2)}^2-\int_0^t\int_{\RR^2}\dot{\Delta}_q(u\times B)\cdot\dot{\Delta}_qj\intd x\mathrm{d}\tau.
\end{split}
\end{equation}
\begin{lemma}\label{Lem-F-JXB}
There holds that
\begin{equation}\label{eq.Est-q-JXB}
\begin{split}
&\int_0^t\int_{\RR^2}\dot{\Delta}_q(u\times B)\cdot\dot{\Delta}_qj\intd x\mathrm{d}\tau\\
\leq&\sup_{p\in\ZZ}\|\dot{S}_{p}u\|_{L^2_tL^\infty(\RR^2)}\|\dot{\Delta}_qj\|_{L^2_tL^2(\RR^2)}\sum_{|p-q|\leq5}\|\dot{\Delta}_pB\|_{L^\infty_tL^2(\RR^2)}\\
&+C\int_0^tc_q\|\dot{\Delta}_qj(\tau)\|_{L^2(\RR^2)}\|B(\tau)\|_{L^2(\RR^2)}\|\nabla u(\tau)\|_{L^2(\RR^2)}\intd\tau,
\end{split}
\end{equation}
where $c_q\in\ell^2.$
\end{lemma}
\begin{proof}[Proof of Lemma \ref{Lem-F-JXB}]
Thanks to the Bony decomposition, we decompose $u^{\ell}B^{m}$ into three parts:
\[u^{\ell} B^{m}=\dot{T}_{u^\ell}B^m+\dot{T}_{B^m}u^{\ell}+\dot{R}(u^{\ell},B^{m}).\]
By the H\"older inequality, we have
\begin{equation*}
\begin{split}
 \int_{\RR^2}\dot{\Delta}_q(\dot{T}_{B^m} u^{\ell})\dot{\Delta}_qj^{i}\intd x =& \int_{\RR^2}\dot{\Delta}_q\Big(\sum_{|p-q|\leq5}\dot{S}_{p-1}B^{m} \dot{\Delta}_pu^{\ell}\Big)\dot{\Delta}_qj^{i}\intd x\\
 \leq&\sum_{|p-q|\leq5}\|\dot{S}_{p-1}B\|_{L^\infty(\RR^2)}\|\dot{\Delta}_pu\|_{L^2(\RR^2)}\|\dot{\Delta}_qj\|_{L^2(\RR^2)}\\
 =&\sum_{|p-q|\leq5}2^{-p}\|\dot{S}_{p-1}B\|_{L^\infty(\RR^2)}2^{p}\|\dot{\Delta}_pu\|_{L^2(\RR^2)}\|\dot{\Delta}_qj\|_{L^2(\RR^2)}\\
 \leq&Cc_q\|\dot{\Delta}_qj\|_{L^2(\RR^2)}\|B\|_{\dot{B}^{-1}_{\infty,2}(\RR^2)}\|u\|_{\dot{B}^1_{2,2}(\RR^2)}\\
 \leq& Cc_q\|\dot{\Delta}_qj\|_{L^2(\RR^2)}\|B\|_{L^2(\RR^2)}\|\nabla u\|_{L^2(\RR^2)},
\end{split}
\end{equation*}
where $c_q\in\ell^2.$

This estimate means
\begin{equation}\label{eq.-q-JXB-I}
\begin{split}
&\int_0^t \int_{\RR^2}\dot{\Delta}_q(\dot{T}_{B^m} u^{\ell})\dot{\Delta}_qj^{i}\intd x\mathrm{d}\tau\\
\leq &C\int_0^tc_q\|\dot{\Delta}_qj(\tau)\|_{L^2(\RR^2)}\|B(\tau)\|_{L^2(\RR^2)}\|\nabla u(\tau)\|_{L^2(\RR^2)}\intd\tau.
\end{split}
\end{equation}
By the H\"older inequality again and the Bernstein inequality, the remainder term can be bounded as follows:
\begin{align*}
 \int_{\RR^2}\dot{\Delta}_q\big(\dot{R}(B^{m}, u^{\ell})\big)\dot{\Delta}_qj^{i}\intd x =& \int_{\RR^2}\dot{\Delta}_q\Big(\sum_{p\geq q-5}\dot{\Delta}_{p}B^{m} {\widetilde{\dot\Delta}}_pu^{\ell}\Big)\dot{\Delta}_qj^{i}\intd x\\
 \leq&\sum_{p\geq q-5}\big\|\dot{\Delta}_q\big(\dot{\Delta}_{p}B{\widetilde{\dot\Delta}}_pu\big)\big\|_{L^2(\RR^2)}\|\dot{\Delta}_qj\|_{L^2(\RR^2)}\\
 \leq&\sum_{p\geq q-5}2^{q}\big\| \dot{\Delta}_q\big(\dot{\Delta}_{p}B{\widetilde{\dot\Delta}}_pu\big)\big\|_{L^1(\RR^2)}\|\dot{\Delta}_qj\|_{L^2(\RR^2)} \\
 \leq& \sum_{p\geq q-5}2^{q}\big\| \dot{\Delta}_{p}B\big \|_{L^2(\RR^2)}\big\|{\widetilde{\dot\Delta}}_pu \big\|_{L^2(\RR^2)}\|\dot{\Delta}_qj\|_{L^2(\RR^2)}.
\end{align*}
Moreover, by the Young inequality, we obtain
\begin{equation*}
\begin{split}
  \int_{\RR^2}\dot{\Delta}_q\big(\dot{R}(B^{m}, u^{\ell})\big)\dot{\Delta}_qj^{i}\intd x
 \leq& \sum_{p\geq q-5}2^{-(p-q)}\big\|{\dot\Delta}_{p}B\big \|_{L^2(\RR^2)}2^{p}\big\|{\widetilde{\dot\Delta}}_pu \big\|_{L^2(\RR^2)}\|\dot{\Delta}_qj\|_{L^2(\RR^2)}\\
 \leq& Cc_q\|\dot{\Delta}_qj\|_{L^2(\RR^2)}\|B\|_{L^2(\RR^2)}\|\nabla u\|_{L^2(\RR^2)},
 \end{split}
 \end{equation*}
 where $c_q\in\ell^2.$

 Thus we have
 \begin{equation}\label{eq.-q-JXB-II}
 \begin{split}
&\int_0^t \int_{\RR^2}\dot{\Delta}_q\big(\dot{R}(B^{m}, u^{\ell})\big)\dot{\Delta}_qj^{i}\intd x\mathrm{d}\tau\\
\leq& C\int_0^tc_q\|\dot{\Delta}_qj(\tau)\|_{L^2(\RR^2)}\|B(\tau)\|_{L^2(\RR^2)}\|\nabla u(\tau)\|_{L^2(\RR^2)}\intd\tau.
\end{split}
\end{equation}
At last, we deal with the para-product term involving the low frequency of $u$. Note that
\begin{equation*}
\begin{split}
 \int_0^t\int_{\RR^2}\dot{\Delta}_q(\dot{T}_{u^\ell} B^{m})\dot{\Delta}_qj^{i}\intd x\mathrm{d}\tau =& \int_0^t\int_{\RR^2}\dot{\Delta}_q\Big(\sum_{|p-q|\leq5}\dot{S}_{p-1}u^{\ell} \dot{\Delta}_pB^{m}\Big)\dot{\Delta}_qj^{i}\intd x\mathrm{d}\tau\\
 \leq&\sum_{|p-q|\leq5}\|\dot{S}_{p-1}u\|_{L^2_tL^\infty(\RR^2)}\|\dot{\Delta}_pB\|_{L^\infty_tL^2(\RR^2)}\|\dot{\Delta}_qj\|_{L^2_tL^2(\RR^2)}\\
\leq&\sup_{p\in\ZZ}\|\dot{S}_{p}u\|_{L^2_tL^\infty(\RR^2)}\|\dot{\Delta}_qj\|_{L^2_tL^2(\RR^2)}\sum_{|p-q|\leq5}\|\dot{\Delta}_pB\|_{L^\infty_tL^2(\RR^2)}.
\end{split}
\end{equation*}
Combining this estimate with estimate  \eqref{eq.-q-JXB-I} and estimate \eqref{eq.-q-JXB-II} implies the required result.
\end{proof}
Now, we come back to the proof of Proposition \ref{prop-improve}. Inserting  estimate \eqref{eq.Est-q-JXB} into \eqref{eq.est-E} leads to
\begin{equation*}
\begin{split}
&\sup_{\tau\in[0,t]}\big\|\dot{\Delta}_qE(\tau)\big\|_{L^2(\RR^2)}^2+\sup_{\tau\in[0,t]}\big\|\dot{\Delta}_qB(\tau)\big\|_{L^2(\RR^2)}^2+\int_0^t\big\|\dot{\Delta}_qj(\tau)\big\|_{L^2(\RR^2)}^2\intd\tau\\
\leq & \big\|\dot{\Delta}_qE_0\big\|_{L^2(\RR^2)}^2+\big\|\dot{\Delta}_qB_0\big\|_{L^2(\RR^2)}^2+\sup_{p\in\ZZ}\|\dot{S}_{p}u\|_{L^2_tL^\infty(\RR^2)}\|\dot{\Delta}_qj\|_{L^2_tL^2(\RR^2)}\sum_{|p-q|\leq5}\|\dot{\Delta}_pB\|_{L^\infty_tL^2(\RR^2)}\\
&+C\int_0^tc_q\|\dot{\Delta}_qj(\tau)\|_{L^2(\RR^2)}\|B(\tau)\|_{L^2(\RR^2)}\|\nabla u(\tau)\|_{L^2(\RR^2)}\intd\tau.
\end{split}
\end{equation*}
Summing the above inequality over $q\in\mathbb{Z}$ provides us
\begin{equation}\label{eq.EB-improve}
\begin{split}
&\|E\|^2_{\widetilde{L}^\infty_t\dot{B}^0_{2,2}(\RR^2)}+\|B\|^2_{\widetilde{L}^\infty_t\dot{B}^0_{2,2}(\RR^2)}+\int_0^t\|j(\tau)\|_{L^2(\RR^2)}^2\intd\tau\\
\leq&\|E_0\|^2_{L^2(\RR^2)}+\|B_0\|_{L^2(\RR^2)}^2+C\int_0^t\sum_{q\in\ZZ}c_q\|\dot{\Delta}_qj\|_{L^2(\RR^2)}\|B\|_{L^2(\RR^2)}\|\nabla u\|_{L^2(\RR^2)}\intd\tau\\
&+\sup_{p\in\ZZ}\|\dot{S}_{p}u\|_{L^2_tL^\infty(\RR^2)}\sum_{q\in\ZZ}\Big(\|\dot{\Delta}_qj\|_{L^2_tL^2(\RR^2)}\sum_{|p-q|\leq5}\|\dot{\Delta}_pB\|_{L^\infty_tL^2(\RR^2)}\Big).
\end{split}
\end{equation}
On one hand, by the H\"older inequality, one has
\begin{equation}\label{eq.EB-1}
\begin{split}
&\int_0^t\sum_{q\in\ZZ}c_q\|\dot{\Delta}_qj(\tau)\|_{L^2(\RR^2)}\|B(\tau)\|_{L^2(\RR^2)}\|\nabla u(\tau)\|_{L^2(\RR^2)}\intd\tau\\
\leq&C\int_0^t \|j(\tau)\|_{L^2(\RR^2)}\|B(\tau)\|_{L^2(\RR^2)}\|\nabla u(\tau)\|_{L^2(\RR^2)}\intd\tau\\
\leq&C\|j\|_{L^2_tL^2(\RR^2)}\|B\|_{L^\infty_tL^2(\RR^2)}\|\nabla u\|_{L^2_tL^2(\RR^2)}.
\end{split}
\end{equation}
On the other hand, the high-low frequency technique enables us to infer that
\begin{equation}\label{eq.EB-2}
\begin{split}
&\sum_{q\in\ZZ}\|\dot{\Delta}_qj\|_{L^2_tL^2(\RR^2)}\sum_{|p-q|\leq5}\|\dot{\Delta}_pB\|_{L^\infty_tL^2(\RR^2)}\\
=&\sum_{|q|>N}\|\dot{\Delta}_qj\|_{L^2_tL^2(\RR^2)}\sum_{|p-q|\leq5}\|\dot{\Delta}_pB\|_{L^\infty_tL^2(\RR^2)}\\
&+\sum_{-N\leq q\leq N}\|\dot{\Delta}_qj\|_{L^2_tL^2(\RR^2)}\sum_{|p-q|\leq5}\|\dot{\Delta}_pB\|_{L^\infty_tL^2(\RR^2)}\\
\leq&\Big(\sum_{|q|\geq N}\big\|\dot{\Delta}_qj\big\|^{2}_{L^2_tL^2(\RR^2)}\Big)^{\frac12}\|B\|_{\widetilde{L}^\infty_t\dot{B}^0_{2,2}(\RR^2)}+CN\|B\|_{L^\infty_tL^2(\RR^2)}\|j\|_{L^2_tL^2(\RR^2)},
\end{split}
\end{equation}
where the positive integer $N$ to be fixed later.

Plugging both estimates \eqref{eq.EB-1}, \eqref{eq.EB-2} in \eqref{eq.EB-improve} yields
\begin{equation}\label{eq.EB-I}
\begin{split}
&\|E\|^2_{\widetilde{L}^\infty_t\dot{B}^0_{2,2}(\RR^2)}+\|B\|^2_{\widetilde{L}^\infty_t\dot{B}^0_{2,2}(\RR^2)}+\int_0^t\|j(\tau)\|_{L^2(\RR^2)}^2\intd\tau\\
\leq&\|E_0\|^2_{L^2(\RR^2)}+\|B_0\|_{L^2(\RR^2)}^2+\sup_{p}\|\dot{S}_{p}u\|_{L^2_tL^\infty(\RR^2)}\Big(\sum_{|q|\geq N}\big\|\dot{\Delta}_qj\big\|^{2}_{L^2_tL^2(\RR^2)}\Big)^{\frac12}\|B\|_{\widetilde{L}^\infty_t\dot{B}^0_{2,2}(\RR^2)}\\
&+CN\sup_{p}\|\dot{S}_{p}u\|_{L^2_tL^\infty(\RR^2)}\|B\|_{L^\infty_tL^2(\RR^2)}\|j\|_{L^2_tL^2(\RR^2)}\\
&+C\|j\|_{L^2_tL^2(\RR^2)}\|B\|_{L^\infty_tL^2(\RR^2)}\|\nabla u\|_{L^2_tL^2(\RR^2)}.
\end{split}
\end{equation}
This together with estimate \eqref{eq.est-F-u} entails
\begin{equation}\label{eq.est-F-uEB}
\begin{split}
 &\|(E,B)\|^2_{\widetilde{L}^\infty_t\dot{B}^0_{2,2}(\RR^2)}+\|u(t)\|_{L^2(\RR^2)}^2 +\int_0^t\|j(\tau)\|_{L^2(\RR^2)}^2\intd\tau\\
 &+\int_0^t\|\nabla u(\tau)\|_{L^2(\RR^2)}^2\intd\tau +\lambda_1\sup_{i\in\ZZ}2^{2i}\int_0^t\sum_{k\in\mathbb{Z}^2}\big\|\sqrt{\phi_{i,k}}u(\tau)\big\|^2_{L^2(\RR^2)}\intd\tau\\
 \leq&Ct+C\sup_{p\geq-1}\big\|S_{p}|u|\big\|_{L^2_tL^\infty(\RR^2)}\Big(\sum_{|q|\geq N}\big\|\dot{\Delta}_qj\big\|^{2}_{L^2_tL^2(\RR^2)}\Big)^{\frac12}\|B\|_{\widetilde{L}^\infty_t\dot{B}^0_{2,2}(\RR^2)}\\
 &+CN\sup_{p\geq-1}\big\|S_{p}|u|\big\|_{L^2_tL^\infty(\RR^2)}.
\end{split}
\end{equation}
By resorting to Lemma \ref{local-Bernstein}, we readily have
\begin{equation*}
\begin{split}
\sup_{p\geq-1}\big\|S_{p}|u|\big\|_{L^2_tL^\infty(\RR^2)}\leq&C\sup_{p\geq-1}\Big(2^{2p}\int_0^t\Big(\sup_{k\in\mathbb{Z}^{2}}\big\|\sqrt{\phi_{p,k}}u(\tau)\big\|_{L^2(\RR^2)}\Big)^2\intd\tau\Big)^{\frac12}\\
\leq&C\sup_{p\geq-1}\Big(2^{2p}\int_0^t\sum_{k\in\mathbb{Z}^{2}}\big\|\sqrt{\phi_{p,k}}u(\tau)\big\|^2_{L^2(\RR^2)}\intd\tau\Big)^{\frac12}\\
\leq&C\Big(\sup_{p\geq-1}2^{2p}\int_0^t\sum_{k\in\mathbb{Z}^2}\big\|\sqrt{\phi_{p,k}}u(\tau)\big\|^2_{L^2(\RR^2)}\intd\tau\Big)^{\frac12}.
\end{split}
\end{equation*}
Inserting this estimate into \eqref{eq.est-F-uEB} and using the Cauchy-Schwarz inequality, we immediately have
\begin{equation}\label{eq.est-FV-uEB}
\begin{split}
 &\|E\|^2_{\widetilde{L}^\infty_t\dot{B}^0_{2,2}(\RR^2)}+\|B\|^2_{\widetilde{L}^\infty_t\dot{B}^0_{2,2}(\RR^2)}+\|u(t)\|_{L^2(\RR^2)}^2 +\int_0^t\|j(\tau)\|_{L^2(\RR^2)}^2\intd\tau\\
 &+\int_0^t\|\nabla u(\tau)\|_{L^2(\RR^2)}^2\intd\tau +\lambda_1\sup_{i\in\ZZ}2^{2i}\int_0^t\sum_{k\in\mathbb{Z}^2}\big\|\sqrt{\phi_{i,k}}u(\tau)\big\|^2_{L^2(\RR^2)}\intd\tau\\
 \leq&C+\widetilde{C}\Big(\sum_{|q|\geq N}\big\|\dot{\Delta}_qj\big\|^{2}_{L^2_tL^2(\RR^2)}\Big)^{\frac12}\Big(\sup_{i\in\ZZ}2^{2i}\int_0^t\sum_{k\in\mathbb{Z}^2}\big\|\sqrt{\phi_{i,k}}u(\tau)\big\|^2_{L^2(\RR^2)}\intd\tau\Big)^{\frac12}\|B\|_{L^\infty_t\dot{B}^0_{2,2}(\RR^2)}\\
 &+CN^2+\frac{\lambda_1}{8} \sup_{i\in\ZZ}2^{2i}\int_0^t\sum_{k\in\mathbb{Z}^2}\big\|\sqrt{\phi_{i,k}}u(\tau)\big\|^2_{L^2(\RR^2)}\intd\tau.
\end{split}
\end{equation}
Since
\begin{align*}
\sum_{q\in\ZZ}\int_0^\infty\big\|\dot{\Delta}_qj(t)\big\|^2_{L^2(\RR^2)}\intd t\leq& C\int_0^\infty\|j(t)\|_{L^2(\RR^2)}^2\intd t\\
\leq& C\left(\|u_0\|_{L^2(\RR^2)}^2+\|B_0\|^2_{L^2(\RR^2)}+\|E_0\|^2_{L^2(\RR^2)}\right),
\end{align*}
 we can choose the integer $N $ sufficiently large such that  $$\widetilde{C}\Big(\sum_{|q|\geq N}\big\|\dot{\Delta}_qj\big\|^{2}_{L^2_tL^2(\RR^2)}\Big)^{\frac12}\leq\frac{\lambda_1}{8}.$$ From this, estimate \eqref{eq.est-FV-uEB} reduces to
\begin{equation}\label{est-EB}
  \|E\|^2_{\widetilde{L}^\infty_t\dot{B}^0_{2,2}(\RR^2)}+\|B\|^2_{\widetilde{L}^\infty_t\dot{B}^0_{2,2}(\RR^2)}
+\lambda_1\sup_{i\in\ZZ}2^{2i}\int_0^t\sum_{k\in\mathbb{Z}^2}\big\|\sqrt{\phi_{i,k}}u(\tau)\big\|^2_{L^2(\RR^2)}\intd\tau\leq C.
\end{equation}
So, we complete the proof of this proposition.
 \end{proof}
 Based on the estimates for $u$ established in Proposition \ref{prop-improve}, we further show the global-in-time a priori estimates for smooth solutions in the borderline space.
\begin{proposition}\label{prop-strong}
Let $u_0\in L^2(\RR^2),\,B_0\in L^2_{\rm log}(\RR^2)$ and $E_0\in L^2_{\rm log}(\RR^2)$. Then, for any smooth solution $(u,E,B)$, there exists a positive constant $C$ such that
\begin{equation*}
\big\|E\big\|_{\widetilde{L}^\infty_tL_{\rm log}^2(\RR^2)}^2+\big\|B\big\|_{\widetilde{L}^\infty_tL_{\rm log}^2(\RR^2)}^2+ \int_0^t\big\|j(\tau)\big\|_{L^2_{\rm log}(\RR^2)}^2\intd\tau\leq C\big(t, \|u_0\|_{L^2(\RR^2)}, \|(E_0, B_0)\|_{L^2_{\rm log}(\RR^2)}\big).
\end{equation*}
\end{proposition}
\begin{proof}
First of all, the same argument as in proving \eqref{eq.est-E} provides
\begin{equation*}
\begin{split}
&\frac12\dtd\big\|\dot{\Delta}_qE(t)\big\|_{L^2(\RR^2)}^2+\frac12\dtd\big\|\dot{\Delta}_qB(t)\big\|_{L^2(\RR^2)}^2+ \big\|\dot{\Delta}_qE(t)\big\|_{L^2(\RR^2)}^2\\
=&-\int_{\RR^2}\dot{\Delta}_q(u\times B)\cdot\dot{\Delta}_qE\intd x.
\end{split}
\end{equation*}
Thanks to the Bony-paraproduct decomposition, the integral in the right side of the above equality can be written as
\begin{equation*}
\begin{split}
\int_{\RR^2}\dot{\Delta}_q(u^{m} B^{\ell})\dot{\Delta}_qE^{i}\intd x=&\int_{\RR^2}\dot{\Delta}_q(\dot{T}_{u^{\ell}}B^{m})\dot{\Delta}_qE^{i}\intd x+\int_{\RR^2}\dot{\Delta}_q(\dot{T}_{B^m}u^{\ell})\dot{\Delta}_qE^{i}\intd x\\
&+\int_{\RR^2}\dot{\Delta}_q\dot{R}(u^{\ell}, B^{m})\dot{\Delta}_qE^{i}\intd x.
\end{split}
\end{equation*}
By the H\"older inequality, we find that for $q>5,$
\begin{equation}\label{eq.TB}
\begin{split}
\int_{\RR^2}\dot{\Delta}_q(\dot{T}_{B^m} u^{\ell})\dot{\Delta}_qE^{i}\intd x
\leq&\sum_{|q'-q|\leq5}\big\|\dot{S}_{q'-1}B\big\|_{L^4(\RR^2)}\big\|\dot{\Delta}_{q'} u\big\|_{L^4(\RR^2)}\big\|\dot{\Delta}_q E\big\|_{L^2(\RR^2)}\\
\leq&\frac{Cc_q}{\sqrt{q}}\sup_{q'\geq1}\sqrt{q'}2^{-\frac{q'}{2}}\big\|\dot{S}_{q'-1}B\big\|_{L^4(\RR^2)}\|u\|_{\dot{B}^{\frac12}_{4,2}(\RR^2)}\big\|\dot{\Delta}_q E\big\|_{L^2(\RR^2)},
\end{split}
\end{equation}
where $c_q\in\ell^2.$

Note that
\begin{align*}
&\sup_{q'\geq1}\sqrt{q'}2^{-\frac{q'}{2}}\big\|\dot{S}_{q'-1}B\big\|_{L^4(\RR^2)}\\
\leq&\sup_{q'\geq1}\sqrt{q'}2^{-\frac{q'}{2}}\sum_{1\leq i\leq  q'-2}\big\|\dot{\Delta}_{i}B\big\|_{L^4(\RR^2)}+\sup_{q'\geq1}\sqrt{q'}2^{-\frac{q'}{2}}\big\|\dot{S}_{0}B\big\|_{L^4(\RR^2)}\\
\leq&\sum_{1\leq i\leq k-2}\sqrt{\frac{k}{i}}2^{-\frac{k-i}{2}}\sqrt{i}2^{-\frac{i}{2}}\big\|\dot{\Delta}_{i}B\big\|_{L^4(\RR^2)}+C\|B\|_{L^2(\RR^2)}\\
\leq&\sum_{1\leq i\leq k-2}\sqrt{k-i}2^{-\frac{k-i}{2}}\sqrt{i}2^{-\frac{i}{2}}\big\|\dot{\Delta}_{i}B\big\|_{L^4(\RR^2)}\\
&+\sum_{1\leq i\leq k-2}2^{-\frac{k-i}{2}}\sqrt{i}2^{-\frac{i}{2}}\big\|\dot{\Delta}_{i}B\big\|_{L^4(\RR^2)}+C\|B\|_{L^2(\RR^2)}\\
\leq&\sup_{i\geq1}\sqrt{i}2^{-\frac{i}{2}}\big\|\dot{\Delta}_{i}B\big\|_{L^4(\RR^2)}+C\|B\|_{L^2(\RR^2)}\leq C\|B\|_{L^2_{\rm log}(\RR^2)}.
\end{align*}
Plugging this estimate in \eqref{eq.TB}  and applying the Cauchy-Schwarz inequality to the resulting estimate, we get
\begin{equation}\label{T1}
\begin{split}
\int_{\RR^2}\dot{\Delta}_q(\dot{T}_{B^m} u^{\ell})\dot{\Delta}_qE^{i}\intd x
\leq&\frac{Cc_q}{\sqrt{q}}\|B\|_{L^2_{\rm log}(\RR^2)}\|u\|_{\dot{B}^{\frac12}_{4,2}(\RR^2)}\big\|\dot{\Delta}_q E\big\|_{L^2(\RR^2)}\\
\leq&\frac{Cc_q}{\sqrt{q}}\|B\|_{L^2_{\rm log}(\RR^2)}\|\nabla u\|_{L^2(\RR^2)}\big\|\dot{\Delta}_q E\big\|_{L^2(\RR^2)}\\
\leq&\frac{Cc^2_q}{q}\|\nabla u\|^2_{L^2(\RR^2)}\|B\|^2_{L^2_{\rm log}(\RR^2)}+\frac14\big\|\dot{\Delta}_q E\big\|^{2}_{L^2(\RR^2)}.
\end{split}
\end{equation}
As for the remainder term, it can be bounded as follows:
\begin{equation}\label{eq.R(u,B)}
\begin{split}
\int_{\RR^2}\dot{\Delta}_q\dot{R}(u^{\ell}, B^{m})\dot{\Delta}_qE^{i}\intd x=&\sum_{q'\geq q-5}\int_{\RR^2}\dot{\Delta}_q\Big(\widetilde{\dot{\Delta}}_{q'}u^{\ell}\dot{\Delta}_{q'}B^{m}\Big)\dot{\Delta}_qE^{i}\intd x\\
\leq&C\sum_{q'\geq q-5}2^{q}\big\|\widetilde{\dot{\Delta}}_{q'}u\big\|_{L^2(\RR^2)}\big\|\dot{\Delta}_{q'}B\big\|_{L^2(\RR^2)}\big\|\dot{\Delta}_qE\big\|_{L^2(\RR^2)}.
\end{split}
\end{equation}
A simple calculation yields that for $q>5,$
\begin{equation*}
\begin{split}
&\sum_{q'\geq q-5}2^{q}\big\|\widetilde{\dot{\Delta}}_{q'}u\big\|_{L^2(\RR^2)}\big\|\dot{\Delta}_{q'}B\big\|_{L^2(\RR^2)}\\
=&\frac{1}{\sqrt{q}}\sum_{q'\geq q-5}\frac{\sqrt{q}}{\sqrt{q'}}2^{q-q'}2^{q'}\big\|\widetilde{\dot{\Delta}}_{q'}u\big\|_{L^2(\RR^2)}\sqrt{q'}\big\|\dot{\Delta}_{q'}B\big\|_{L^2(\RR^2)}\\
\leq&\frac{1}{\sqrt{q}}\sum_{q'\geq q-5}{\sqrt{|q-q'|}}2^{q-q'}2^{q'}\big\|\widetilde{\dot{\Delta}}_{q'}u\big\|_{L^2(\RR^2)}\sqrt{q'}\big\|\dot{\Delta}_{q'}B\big\|_{L^2(\RR^2)}\\
&+\frac{1}{\sqrt{q}}\sum_{q'\geq q-5}2^{q-q'}2^{q'}\big\|\widetilde{\dot{\Delta}}_{q'}u\big\|_{L^2(\RR^2)}\sqrt{q'}\big\|\dot{\Delta}_{q'}B\big\|_{L^2(\RR^2)}.
\end{split}
\end{equation*}
Moreover, by the discrete Young inequality, one has that for $q>5$,
\begin{equation*}
\begin{split}
\sum_{q'\geq q-5}2^{q}\big\|\widetilde{\dot{\Delta}}_{q'}u\big\|_{L^2(\RR^2)}\big\|\dot{\Delta}_{q'}B\big\|_{L^2(\RR^2)}=&\sum_{q'\geq q-5}2^{q-q'}2^{q'}\big\|\widetilde{\dot{\Delta}}_{q'}u\big\|_{L^2(\RR^2)}\big\|\dot{\Delta}_{q'}B\big\|_{L^2(\RR^2)}\\
\leq& \frac{Cc_q}{\sqrt{q}}\|\nabla u\|_{L^2(\RR^2)}\|B\|_{L^2_{\rm log}(\RR^2)},
\end{split}
\end{equation*}
where $c_q\in\ell^2.$

Inserting this estimate into \eqref{eq.R(u,B)} leads to
\begin{equation}\label{T2}
\begin{split}
\int_{\RR^2}\dot{\Delta}_q\dot{R}(u^{\ell}, B^{m})\dot{\Delta}_qE^i\intd x\leq& \frac{Cc_q}{\sqrt{q}}\|\nabla u\|_{L^2(\RR^2)}\|B\|_{L^2_{\rm log}(\RR^2)}\big\|\dot{\Delta}_qE\big\|_{L^2(\RR^2)}\\
\leq&\frac{Cc^2_q}{q}\|\nabla u\|^2_{L^2(\RR^2)}\|B\|^2_{L^2_{\rm log}(\RR^2)}+\frac14\big\|\dot{\Delta}_q E\big\|^{2}_{L^2(\RR^2)}.
\end{split}
\end{equation}
Lastly, we tackle with the para-product term $\int_{\RR^2}\dot{\Delta}_q(\dot{T}_{u^{\ell}} B^{m}) \dot{\Delta}_qE^{i}\intd x$. We see that for $q>5,$
\begin{equation*}
\begin{split}
\int_{\RR^2}\dot{\Delta}_q(\dot{T}_{u^{\ell}} B^{m}) \dot{\Delta}_qE^{i}\intd x=&\sum_{|k-q|\leq 5}\int_{\RR^2}\dot{\Delta}_q\big(\dot{S}_{k-1}u^{\ell}\dot{\Delta}_kB^{m}\big)\dot{\Delta}_qE^{i}\intd x\\
\leq&\sum_{|k-q|\leq 5}\big\|\dot{S}_{k-1}u\big\|_{L^\infty(\RR^2)}\big\|\dot{\Delta}_kB\big\|_{L^2(\RR^2)}\big\|\dot{\Delta}_qE\big\|_{L^2(\RR^2)}\\
\leq&C\big\|\dot{S}_{q+5}u\big\|_{L^\infty(\RR^2)}\sum_{|k-q|\leq 5}\big\|\dot{\Delta}_kB\big\|_{L^2(\RR^2)}\big\|\dot{\Delta}_qE\big\|_{L^2(\RR^2)}\\
\leq&\frac{Cc_q}{\sqrt{q}}\big\|\dot{S}_{q+5}u\big\|_{L^\infty(\RR^2)}\|B\|_{L^2_{\rm log}(\RR^2)}\big\|\dot{\Delta}_qE\big\|_{L^2(\RR^2)}.
\end{split}
\end{equation*}
By the Cauchy-Schwarz inequality, we readily have that for $q>5,$
\begin{equation}\label{T3}
\int_{\RR^2}\dot{\Delta}_q(u\times B)\cdot\dot{\Delta}_qE\intd x\leq\big\|\dot{S}_{q+5}u\big\|^2_{L^\infty(\RR^2)}\big\|\dot{\Delta}_qE\big\|^2_{L^2(\RR^2)}+\frac{Cc^2_q}{q}\|B\|^2_{L^2_{\rm log}(\RR^2)}.
\end{equation}
Collecting all these estimates \eqref{T1}, \eqref{T2}, \eqref{T3}  yields that for $q>5,$
\begin{equation*}
\begin{split}
&\dtd\big\|\dot{\Delta}_qE(t)\big\|_{L^2(\RR^2)}^2+\dtd\big\|\dot{\Delta}_qB(t)\big\|_{L^2(\RR^2)}^2+ \big\|\dot{\Delta}_qE(t)\big\|_{L^2(\RR^2)}^2\\
\leq&\big\|\dot{S}_{q+5}u\big\|^2_{L^\infty(\RR^2)}\big\|\dot{\Delta}_qE\big\|^2_{L^2(\RR^2)}+\frac{Cc^2_q}{q}\|B\|^2_{L^2_{\rm log}(\RR^2)}+\frac{Cc^2_q}{q}\|\nabla u\|^2_{L^2(\RR^2)}\|B\|^2_{L^2_{\rm log}(\RR^2)}.
\end{split}
\end{equation*}
In view of the Gronwall inequality, we immediately have that for $q>5,$
\begin{equation}\label{eq.Gronwall-1}
\begin{split}
&\big\|\dot{\Delta}_qE(t)\big\|_{L^2(\RR^2)}^2+\big\|\dot{\Delta}_qB(t)\big\|_{L^2(\RR^2)}^2+ \int_0^t\big\|\dot{\Delta}_qE(\tau)\big\|_{L^2(\RR^2)}^2\intd\tau\\
\leq&e^{\int_0^t\|\dot{S}_{q+5}u(\tau)\|^2_{L^\infty(\RR^2)}\intd\tau}\Big(\big\|\dot{\Delta}_q(E_0,B_0)\big\|^2_{L^2(\RR^2)}+C\int_0^t\frac{c^2_q}{q}\|B(\tau)\|^2_{L^2_{\rm log}(\RR^2)}\intd\tau\\
&+C\int_0^t\frac{c^2_q}{q}\|\nabla u(\tau)\|^2_{L^2(\RR^2)}\|B(\tau)\|^2_{L^2_{\rm log}(\RR^2)}\intd\tau\Big).
\end{split}
\end{equation}
Note that
\begin{equation*}
\begin{split}
\int_0^t\|\dot{S}_{q+5}u(\tau)\big\|^2_{L^\infty(\RR^2)}\intd\tau\leq& C2^{2q}\int_0^t\sup_{k\in\ZZ^2}\big\|\sqrt{\phi_{q,k}}u(\tau)\big\|^2_{L^2(\RR^2)}\intd\tau\\
\leq& C\sup_{i\in\ZZ}2^{2i}\int_0^t\sup_{k\in\ZZ^2}\big\|\sqrt{\phi_{i,k}}u(\tau)\big\|^2_{L^2((\RR^2))}\intd\tau.
\end{split}
\end{equation*}
Multiplying \eqref{eq.Gronwall-1} by $q$ and summing the resulting inequality over $q>5$, we get by using the $L^2$-estimate \eqref{eq.L2estimate} that
\begin{align*}
&\sum_{q>5}q\big\|\dot{\Delta}_{q}E\big\|_{\widetilde{L}^\infty_tL^2(\RR^2)}^2+\sum_{q>5}q\big\|\dot{\Delta}_{q}B\big\|_{\widetilde{L}^\infty_tL^2(\RR^2)}^2+ \sum_{q>5}q \int_0^t\big\|\dot{\Delta}_{q}E(\tau)\big\|_{L^2(\RR^2)}^2\intd\tau\\&
+\|u(t)\|^{2}_{L^2(\RR^2)}+\|E(t)\|^{2}_{L^2(\RR^2)}+\|B(t)\|^{2}_{L^2(\RR^2)}+\|E(t)\|^{2}_{L^2(\RR^2)}\\
\leq&C\|u_0\|_{L^2(\RR^2)}^2+C\big\|(E_0,B_0)\big\|^2_{L^2_{\rm log}(\RR^2)}+C\int_0^t\|B(\tau)\|^2_{L^2_{\rm log}(\RR^2)}\intd\tau\\
&+C\int_0^t\|\nabla u(\tau)\|^2_{L^2(\RR^2)}\|B(\tau)\|^2_{L^2_{\rm log}(\RR^2)}\intd\tau.
\end{align*}
Since
\begin{align*}
 &\big\|E\big\|_{\widetilde{L}^\infty_tL_{\rm log}^2(\RR^2)}^2+ \big\|B\big\|_{\widetilde{L}^\infty_tL_{\rm log}^2(\RR^2)}^2+  \int_0^t\big\|E(\tau)\big\|_{L^2_{\rm log}(\RR^2)}^2\intd\tau\\
 \leq&\sum_{q>5}q\big\|\dot{\Delta}_{q}E\big\|_{\widetilde{L}^\infty_tL^2(\RR^2)}^2+\sum_{q>5}q\big\|\dot{\Delta}_{q}B\big\|_{\widetilde{L}^\infty_tL^2(\RR^2)}^2+ \sum_{q>5}q \int_0^t\big\|\dot{\Delta}_{q}E(\tau)\big\|_{L^2(\RR^2)}^2\intd\tau\\&
+C\|u(t)\|^{2}_{L^2(\RR^2)}+C\|E(t)\|^{2}_{L^2(\RR^2)}+C\|B(t)\|^{2}_{L^2(\RR^2)}+C\|E(t)\|^{2}_{L^2(\RR^2)},
\end{align*}
we have
\begin{align*}
 &\big\|E\big\|_{\widetilde{L}^\infty_tL_{\rm log}^2(\RR^2)}^2+ \big\|B\big\|_{\widetilde{L}^\infty_tL_{\rm log}^2(\RR^2)}^2+  \int_0^t\big\|E(\tau)\big\|_{L^2_{\rm log}(\RR^2)}^2\intd\tau\\
 \leq&C +C\int_0^t\|B(\tau)\|^2_{L^2_{\rm log}(\RR^2)}\intd\tau
+C\int_0^t\|\nabla u(\tau)\|^2_{L^2(\RR^2)}\|B(\tau)\|^2_{L^2_{\rm log}(\RR^2)}\intd\tau.
 \end{align*}
By the Gronwall inequality again, we eventually get that
\begin{equation}\label{est-EBE}
\begin{split}
&\big\|E\big\|_{\widetilde{L}^\infty_tL_{\rm log}^2(\RR^2)}^2+\big\|B\big\|_{\widetilde{L}^\infty_tL_{\rm log}^2(\RR^2)}^2+ \int_0^t\big\|E(\tau)\big\|_{L^2_{\rm log}(\RR^2)}^2\intd\tau\\
\leq& C e^{Ct+C\int_0^t\|\nabla u(\tau)\|^2_{L^2(\RR^2)}\,\mathrm{d}\tau}.
\end{split}
\end{equation}
Based on this regularity, we turn to show that
 $\int_0^t\|j(\tau)\|^{2}_{L^2_{\rm log}(\RR^2)}\intd\tau<\infty.$ Since $j=E+u\times B$ and $\int_0^t\|E(\tau)\|^{2}_{L^2_{\rm log}(\RR^2)}\intd\tau<\infty$, we just need to show that $\int_0^t\|(u\times B)(\tau)\|^{2}_{L^2_{\rm log}(\RR^2)}\intd\tau<\infty.$
 Thanks to the Bony paraproduct decomposition, one writes
 \begin{align*}
 u^{\ell}B^{i}=\dot{T}_{u_{\ell}}B^{i}+\dot{T}_{B^i}u^{\ell}+\dot{R}(u^{\ell},B^{i}).
 \end{align*}
According to the definition of $L^2_{\rm log}(\RR^2)$, we have
 \begin{align*}
&\int_0^t\|\dot{T}_{u^{\ell}}B^{i}(\tau)\|^{2}_{L^2_{\rm log}(\RR^2)}\intd\tau\\
\leq&C\int_0^t\sum_{q\leq5}\big\|\dot{\Delta}_{q}(\dot{T}_{u^{\ell}}B^{i})\big\|^{2}_{L^2(\RR^2)}\intd\tau
+C\int_0^t\sum_{q>5}q\big\|\dot{\Delta}_{q}(\dot{T}_{u^{\ell}}B^{i})\big\|^{2}_{L^2(\RR^2)}\intd\tau\\
:=&I+II.
 \end{align*}
 By the H\"older inequality, we immediately have
  \begin{align*}
 I\leq &C\sum_{q\leq5}\sum_{|p-q|\leq5}\int_0^t\big\| \dot{S}_{p-1} u(\tau)\big\|^{2}_{L^\infty(\RR^2)}\|\dot{\Delta}_pB(\tau)\|^{2}_{L^2(\RR^2)}\intd\tau\\
  \leq&C\sup_{p\in\ZZ}\int_0^t\big\| \dot{S}_{p-1} u(\tau)\big\|^{2}_{L^\infty(\RR^2)}\intd\tau\sum_{p\in\ZZ}\|\dot{\Delta}_pB\|^{2}_{L^\infty_tL^2(\RR^2)}\\
  \leq&C\sup_{i\in\ZZ}2^{2i}\int_0^t\sup_{k\in\ZZ^2}\big\|\sqrt{\phi_{i,k}}u(\tau)\big\|^2_{L^2((\RR^2))}\intd\tau\|B\|_{\widetilde{L}^\infty_tL^2(\RR^2)}^{2}.
  \end{align*}
  Similarly, we have
  \begin{align*}
 II\leq &C\sum_{q>5}\sum_{|p-q|\leq5}q\int_0^t\big\| \dot{S}_{p-1} u(\tau)\big\|^{2}_{L^\infty(\RR^2)}\|\dot{\Delta}_pB(\tau)\|^{2}_{L^2(\RR^2)}\intd\tau\\
  \leq&C\sup_{p\in\ZZ}\int_0^t\big\| \dot{S}_{p-1} u(\tau)\big\|^{2}_{L^\infty(\RR^2)}\intd\tau\sum_{p\geq1}p\|\dot{\Delta}_pB\|^{2}_{L^\infty_tL^2(\RR^2)}\\
  \leq&C\sup_{i\in\ZZ}2^{2i}\int_0^t\sup_{k\in\ZZ^2}\big\|\sqrt{\phi_{i,k}}u(\tau)\big\|^2_{L^2((\RR^2))}\intd\tau\|B\|_{\widetilde{L}^\infty_tL_{\rm log}^2(\RR^2)}^{2}.
  \end{align*}
 From both estimates for $I$ and $II$, we have
 \[\int_0^t\|\dot{T}_{u^{\ell}}B^{i}(\tau)\|^{2}_{L^2_{\rm log}(\RR^2)}\intd\tau\leq C\sup_{i\in\ZZ}2^{2i}\int_0^t\sup_{k\in\ZZ^2}\big\|\sqrt{\phi_{i,k}}u(\tau)\big\|^2_{L^2((\RR^2))}\intd\tau\|B\|_{\widetilde{L}^\infty_tL_{\rm log}^2(\RR^2)}^{2}.\]
  Next, according to the definition of $L^2_{\rm log}(\RR^2)$, we see that
 \begin{align*}
&\int_0^t\|\dot{T}_{B^{i}}u^{\ell}(\tau)\|^{2}_{L^2_{\rm log}(\RR^2)}\intd\tau\\
\leq&C\int_0^t\sum_{q\leq5}\big\|\dot{\Delta}_{q}(\dot{T}_{B^{i}}u^{\ell})\big\|^{2}_{L^2(\RR^2)}\intd\tau
+C\int_0^t\sum_{q>5}q\big\|\dot{\Delta}_{q}(\dot{T}_{B^{i}}u^{\ell})\big\|^{2}_{L^2(\RR^2)}\intd\tau\\
:=&J_1+J_2.
 \end{align*}
 By the H\"older inequality and the Bernstein inequality, one gets
  \begin{align*}
J_1\leq &C\sum_{q\leq5}\sum_{|p-q|\leq5}\int_0^t\big\| \dot{S}_{p-1} B(\tau)\big\|^{2}_{L^2(\RR^2)}\|\dot{\Delta}_pu(\tau)\|^{2}_{L^\infty(\RR^2)}\intd\tau\\
  \leq&C\|B\|_{\widetilde{L}^\infty_tL^2(\RR^2)}^{2}\sum_{p\in\ZZ}\|\dot{\Delta}_pu\|^{2}_{L^2_tL^\infty(\RR^2)}
  \leq C\|\nabla u\|^{2}_{L^2_tL^2(\RR^2)}\|B\|_{\widetilde{L}^\infty_tL^2(\RR^2)}^{2}.
  \end{align*}
  Similarly, we see that
   \begin{align*}
J_2\leq &C\sum_{q>5}q\sum_{|p-q|\leq5}\int_0^t\big\| \dot{S}_{p-1} B(\tau)\big\|^{2}_{L^\infty(\RR^2)}\|\dot{\Delta}_pu(\tau)\|^{2}_{L^2(\RR^2)}\intd\tau\\
\leq &C\sum_{q>5}\sum_{|p-q|\leq5}\frac{q}{p}\int_0^tp2^{-2p}\big\| \dot{S}_{p-1} B(\tau)\big\|^{2}_{L^\infty(\RR^2)}2^{2p}\|\dot{\Delta}_pu(\tau)\|^{2}_{L^2(\RR^2)}\intd\tau\\
\leq &C\sup_{p\geq1}p2^{-2p}\big\| \dot{S}_{p-1} B(\tau)\big\|^{2}_{L_t^\infty L^\infty(\RR^2)}\sum_{p\geq1}2^{2p}\|\dot{\Delta}_pu \|^{2}_{L^2_tL^2(\RR^2)} \\
  \leq&C\|B\|_{\widetilde{L}^\infty_tL^2_{\rm log}(\RR^2)}^{2}\sum_{p\in\ZZ}2^{2p}\|\dot{\Delta}_pu\|^{2}_{L^2_tL^2(\RR^2)}
  \leq C\|\nabla u\|^{2}_{L^2_tL^2(\RR^2)}\|B\|_{\widetilde{L}^\infty_tL^2_{\rm log}(\RR^2)}^{2}.
  \end{align*}
  As for the remainder term, we find that
   \begin{align*}
&\int_0^t\|\dot{R}(u^{\ell},B^{i})(\tau)\|^{2}_{L^2_{\rm log}(\RR^2)}\intd\tau\\
\leq&C\int_0^t\sum_{q\leq5}\big\|\dot{\Delta}_{q}(\dot{R}(u^{\ell},B^{i}))\big\|^{2}_{L^2(\RR^2)}\intd\tau
+C\int_0^t\sum_{q>6}q\big\|\dot{\Delta}_{q}(\dot{R}(u^{\ell},B^{i}))\big\|^{2}_{L^2(\RR^2)}\intd\tau\\
:=&K_1+K_2.
 \end{align*}
 By the H\"older inequality and the discrete Young inequality, one has
 \begin{align*}
 K_1\leq&C\sum_{q\leq5}\sum_{p\geq q-5}\int_0^t2^{2q}\big\|\widetilde{\dot{\Delta}}_{p}u(\tau)\big\|^{2}_{L^2(\RR^2)}\big\|\dot{\Delta}_{p}B(\tau)\big\|^{2}_{L^2(\RR^2)}\intd\tau\\
 \leq&C\sum_{q\leq5}\sum_{p\geq q-5}\int_0^t2^{2(q-p)}2^{2p}\big\|\widetilde{\dot{\Delta}}_{q}u(\tau)\big\|^{2}_{L^2(\RR^2)}\big\|\dot{\Delta}_{q}B(\tau)\big\|^{2}_{L^2(\RR^2)}\intd\tau\\
 \leq&C\|B\|_{\widetilde{L}^\infty_tL^2(\RR^2)}^{2}\sum_{p\in\ZZ}2^{2p}\|\widetilde{\dot{\Delta}}_pu\|^{2}_{L^2_tL^2(\RR^2)}
  \leq C\|\nabla u\|^{2}_{L^2_tL^2(\RR^2)}\|B\|_{\widetilde{L}^\infty_tL^2(\RR^2)}^{2}.
 \end{align*}
In the similar fashion, we can obtain
\begin{align*}
 K_2\leq&C\sum_{q>5}q\sum_{p\geq q-5}\int_0^t2^{2p}\big\|\widetilde{\dot{\Delta}}_{p}u(\tau)\big\|^{2}_{L^2(\RR^2)}\big\|\dot{\Delta}_{p}B(\tau)\big\|^{2}_{L^2(\RR^2)}\intd\tau\\
 \leq&C\sum_{q\geq1}\sum_{p\geq q-5}\int_0^t\frac{q}{p}2^{2(q-p)}2^{2p}\big\|\widetilde{\dot{\Delta}}_{p}u(\tau)\big\|^{2}_{L^2(\RR^2)}p\big\|\dot{\Delta}_{p}B(\tau)\big\|^{2}_{L^2(\RR^2)}\intd\tau\\
 \leq&C\|B\|_{\widetilde{L}^\infty_tL^2_{\rm log}(\RR^2)}^{2}\sum_{p\geq1}2^{2p}\|\widetilde{\dot{\Delta}}_pu\|^{2}_{L^2_tL^2(\RR^2)}
  \leq C\|\nabla u\|^{2}_{L^2_tL^2(\RR^2)}\|B\|_{\widetilde{L}^\infty_tL^2_{\rm log}(\RR^2)}^{2}.
\end{align*}
Therefore, we finally get
\begin{align*}
&\int_0^t\|(u\times B)(\tau)\|^{2}_{L^2_{\rm log}(\RR^2)}\intd\tau\\
\leq&C\Big(\sup_{i\in\ZZ}2^{2i}\int_0^t\sup_{k\in\ZZ^2}\big\|\sqrt{\phi_{i,k}}u(\tau)\big\|^2_{L^2((\RR^2))}+\|\nabla u\|^{2}_{L^2_tL^2(\RR^2)}\Big)\|B\|_{\widetilde{L}^\infty_tL^2_{\rm log}(\RR^2)}^{2}.
\end{align*}
By Proposition \ref{prop-improve} and estimate \eqref{est-EBE}, we know that $\int_0^t\|(u\times B)(\tau)\|^{2}_{L^2_{\rm log}(\RR^2)}\intd\tau<\infty$. So, we finish the proof of the proposition.
\end{proof}
Based on this regularity in the borderline space, we can show the global-in-time bound for $\int_0^t\|u(\tau)\|_{L^\infty(\RR^2)}^2\intd\tau$, which plays an important role in the proof of some known results such as \cite{PSN}.
\begin{proposition}\label{coro-L2LI}
Let $u_0\in L^2(\RR^2),\,B_0\in L^2_{\rm log}(\RR^2)$ and $E_0\in L^2_{\rm log}(\RR^2)$. Then, for any smooth solution $(u,E,B)$, there holds that
\begin{equation}\label{eq.u-L-infty}
\int_0^t\left\|\widehat{u}(\tau)\right\|_{L^1(\RR^2)}^2\,\mathrm{d}\tau\leq C\big(t, \|u_0\|_{L^2(\RR^2)}, \|(E_0, B_0)\|_{L^2_{\rm log}(\RR^2)}\big).
\end{equation}
\end{proposition}
\begin{proof}
By Duhamel formula, one writes the solution $u$ in the following form
\begin{equation*}
u(t,x)=u_{2d}+u_2,
\end{equation*}
where
\begin{equation*}
u_{2}(x,t)=\int_0^te^{(t-\tau)\Delta}\mathbb{P}(j\times B)(\tau)\intd\tau
\end{equation*}
and
\begin{equation*}
u_{2d}(x,t)=e^{t\Delta}u_0+\int_0^te^{(t-\tau)\Delta}\mathbb{P}(u\otimes u)(\tau)\intd\tau
\end{equation*}
which is a solution of the following equations governed by
\begin{equation}\label{eq.2dNS}
\left\{\begin{array}{ll}
\partial_tv+(u\cdot\nabla)u-\Delta v+\nabla\pi=0\qquad(t,x)\in\RR^+\times\RR^2,\\
\Div v=0,\\
v|_{t=0}=u_0.
\end{array}\right.
\end{equation}
First of all, we are going to show
\[\int_0^t\left\|\widehat{v}(\tau)\right\|_{L^1(\RR^2)}^2\,\mathrm{d}\tau<\infty,\]
which is the direct consequence of the following proposition.
\begin{proposition}\label{PropNS}
Let $v$ be a solution of the nonlinear equations \eqref{eq.2dNS}.
 Then, we have
\[ \int_0^t\|\hat{v}(\tau)\|_{L^1(\RR^2)}^2\,\mathrm{d}\tau\leq C\left(\|u_0\|_{L^2(\RR^2)}\right).\]
\end{proposition}
\begin{remark}
Let us point out that in this proposition we give a new method to show that the Leray solution of the two-dimensional Navier-Stokes equations satisfies
\[\int_0^t\big\|u_{2d}(\tau)\big\|^2_{L^\infty(\RR^2)}\,\mathrm{d}\tau<\infty,\]
which was shown in \cite{CG06}. More importantly, we also prove that
$\int_0^t\left\|\widehat{u_{2d}}(\tau)\right\|_{L^1(\RR^2)}^2\,\mathrm{d}\tau<\infty.$
\end{remark}
\begin{proof}[Proof of Proposition \ref{PropNS}.]
By Duhamel formula, we have that
\[v(x,t)=e^{t\Delta}u_0+\int_0^te^{(t-s)\Delta}\mathbb{P}\mathrm{div}\big(u\otimes u\big)\,\mathrm{d}\tau.\]
Taking Fourier transform yields
\[\hat{v}(\xi,t)=e^{-t|\xi|^2}\hat{u}_0+\int_0^te^{-(t-s)|\xi|^2}\left(\mathrm{I_{d}}-\frac{\xi_i\xi_j}{|\xi|^2}\right)i\xi\cdot\widehat{\big(u\otimes u\big)}\,\mathrm{d}\tau.\]
For the linear part, Proposition \ref{prop.negative} allows us to get
\begin{equation*}
\begin{split}
\big\|e^{-t|\xi|^2}\hat{u}_0\big\|_{L^2(\RR^+;\,L^1(\RR^2))}=\left\|t^{\frac12}\big\|e^{-t|\xi|^2}\hat{u}_0\big\|_{L^1(\RR^2)}\right\|_{L^2(\mathbb{R}^+;\,\frac{\mathrm{d}t}{t})}\sim\|u_0\|_{F\dot{B}^{-1}_{1,2}}.
\end{split}
\end{equation*}
For the nonlinear part, we see that
\begin{equation*}
\begin{split}
& \left\|\int_0^te^{-(t-\tau)|\xi|^2}\left(\mathrm{I_{d}}-\frac{\xi_i\xi_j}{|\xi|^2}\right)i\xi\cdot\widehat{\big(u\otimes u\big)}\,\mathrm{d}\tau\right\|_{L^2((\RR^+;\,L^1(\RR^2))}\\
\leq&\left\|\int_0^t\left\||\xi|^{\frac12}e^{-(t-\tau)|\xi|^2}\right\|_{L^2(\RR^2)}\left\||\xi|^{-\frac12}\widehat{\big((u\cdot\nabla) u\big)}\right\|_{L^2(\RR^6)}\,\mathrm{d}\tau\right\|_{L^2(\RR^+)}\\
\leq&\left\|\int_0^t\left\||\xi|^{\frac12}e^{-(t-\tau)|\xi|^2}\right\|_{L^2(\RR^2)}\left\||\xi|^{-\frac12}\widehat{\big((u\cdot\nabla) u\big)}\right\|_{L^2(\RR^2)}\,\mathrm{d}\tau\right\|_{L^2(\RR^+)}.
\end{split}
\end{equation*}
Thanks to the Bony paraproduct decomposition and $\Div u=0$, we have the following estimate for the bilinear term
\[\left\||\xi|^{-\frac12}\widehat{\big((u\cdot\nabla) u\big)}\right\|_{L^2(\RR^2)}\leq C\|\Lambda^{\frac12} u(t)\|_{L^2(\RR^2)}\|\nabla u(t)\|_{L^2(\RR^2)}.\]
Therefore, we have
 \begin{align*}
& \left\|\int_0^te^{-(t-\tau)|\xi|^2}\left(\mathrm{I_{d}}-\frac{\xi_i\xi_j}{|\xi|^2}\right)i\xi\cdot\widehat{\big(u\otimes u\big)}\,\mathrm{d}\tau\right\|_{L^2((\RR^+;\,L^1(\RR^2))}\\
\leq&C\Big\|\int_0^t(t-\tau)^{-\frac34}\|\Lambda^{\frac12} u(t)\|_{L^2(\RR^2)}\|\nabla u(t)\|_{L^2(\RR^2)}\,\mathrm{d}\tau\Big\|_{L^2(\RR^+)}\\
\leq&C\big\|t^{-\frac34}\big\|_{L^{\frac43,\infty}(\RR^+)}\left\|\|\Lambda^{\frac12} u(t)\|_{L^2(\RR^2)}\|\nabla u(t)\|_{L^2(\RR^2)}\right\|_{L^{\frac43,2}(\RR^+)}\\
\leq&C\big\|\|\Lambda^{\frac12} u(t)\|_{L^2(\RR^2)}\big\|_{L^{4,\infty}(\RR^+)}\left\|\|\nabla u(t)\|_{L^2(\RR^2)}\right\|_{L^{2,2}(\RR^+)}\\
\leq&C\big\|\Lambda^{\frac12} u\big\|_{L^4(\RR^+;\,L^2(\RR^2))} \|\nabla u \|_{L^2(\RR^+;\,L^2(\RR^2))},
\end{align*}
where we have used the following lemma.
\begin{lemma}[\cite{LGrafakos,ONeil}]
\begin{itemize}
               \item Let $1<p,q,r<\infty$, $0< s_1,s_2\leq\infty,$ $\frac1p+\frac1q=\frac1r+1$, and $\frac{1}{s_1}+\frac{1}{s_2}=\frac1s.$ Then there holds
               \[\|f\ast g\|_{L^{r,s}(\RR^d)}\leq C(p,q,s_1,s_2)\|f\|_{L^{p,s_1}(\RR^d)}\|g\|_{L^{q,s_2}(\RR^d)}.\]
               \item Let $0<p,q,r\leq\infty$, $0< s_1,s_2\leq\infty,$ $\frac1p+\frac1q=\frac1r$, and $\frac{1}{s_1}+\frac{1}{s_2}=\frac1s.$ Then we have the H\"older inequality for Lorentz spaces
                   \[\|fg\|_{L^{r,s}(\RR^d)}\leq C(p,q,s_1,s_2)\|f\|_{L^{p,s_1}(\RR^d)}\|g\|_{L^{q,s_2}(\RR^d)}.\]
             \end{itemize}

\end{lemma}
Collecting these estimates, we immediately get
\begin{equation*}
\left(\int_0^\infty\|\hat{u}(t)\|^2_{L^1(\RR^2)}\,\mathrm{d}\tau\right)^{\frac12}\leq C\|u_0\|^2_{L^2(\RR^2)}+C\big\|\Lambda^{\frac12} u\big\|_{L^4(\RR^+;\,L^2(\RR^2))}\big\|\nabla u\big\|_{L^2(\RR^+;\,L^2(\RR^2))}.
\end{equation*}
This together with the energy estimate
\[\|u(t)\|_{L^2(\RR^2)}^2+\int_0^t\|\nabla u(\tau)\|_{L^2(\RR^2)}^2\,\mathrm{d}\tau\leq \|u_0\|_{L^2(\RR^2)}^2\]
and the Hausdorff-Young inequality entails the desired result.
\end{proof}

Next, we just need to bound the following the quantity including $u_2.$
Taking the Fourier transform and taking $L^1$-norm, we readily have
\begin{equation*}
\begin{split}
\left(\int_0^t\|\hat{u_{2}}(s)\|_{L^1(\RR^2)}^2\intd s\right)^{\frac12}\leq&\left(\int_0^t\left\|\int_0^se^{-(s-\tau)|\xi|^2}\widehat{\mathbb{P}(j\times B)}(\tau)\intd\tau\right\|_{L^1(\RR^2)}^2\intd s\right)^{\frac12}\\
 \leq&C\left(\int_0^t\left\|\int_0^se^{-(s-\tau)|\xi|^2} \widehat{(j\times B)} (\tau)\intd\tau\right\|_{L^1(\RR^2)}^2\intd s\right)^{\frac12}.
\end{split}
\end{equation*}
The inhomogeneous Bony paraproduct decomposition allows us to write
\begin{align*}
&\left(\int_0^t\left\|\int_0^se^{-(s-\tau)|\xi|^2}\widehat{(j\times B)} (\tau)\intd\tau\right\|_{L^1(\RR^2)}^2\intd s\right)^{\frac12} \\
=&\left(\int_0^t\left\|\int_0^se^{-(s-\tau)|\xi|^2} \widehat{(T_{j^\ell} B^m)} (\tau)\intd\tau\right\|_{L^1(\RR^2)}^2\intd s\right)^{\frac12}\\
&+\left(\int_0^t\left\|\int_0^se^{-(s-\tau)|\xi|^2} \widehat{(T_{B^m} j^\ell )} (\tau)\intd\tau\right\|_{L^1(\RR^2)}^2\intd s\right)^{\frac12} \\
&+\left(\int_0^t\left\|\int_0^se^{-(s-\tau)|\xi|^2} \widehat{ R(B^m ,j^\ell) } (\tau)\intd\tau\right\|_{L^1(\RR^2)}^2\intd s\right)^{\frac12}\\
:=&I+II+III,
\end{align*}
where $1\leq j,\,m\leq3.$

By the  Minkowski inequality, the H\"older inequality and the Young inequality, we have
\begin{equation}\label{L1-I}
\begin{split}
I\leq&\left\|\Big(\int_0^t\Big|\int_0^se^{-(s-\tau)|\xi|^2}\widehat{(T_{j^\ell} B^m)}(\tau)\intd\tau\Big|^2\intd s\Big)^{\frac12}\right\|_{L^1(\RR^2)} \\
\leq &C\Big\|\frac{1}{|\xi|^2}\big\|\widehat{\dot{T}_{j^{\ell}} B^{m}}\big\|_{L^2_t}\Big\|_{L^1(\RR^2)}\\
\leq&C\sum_{q=-1}^{\infty}2^{-2q}\Big\|\big\|\widehat{{S}_{q-1}j}\ast\widehat{\dot{\Delta}_qB }\big\|_{L^2_t}\Big\|_{L^1(\RR^2)}\\
\leq&C\sum_{q=-1}^{\infty}2^{-q}\Big\|\Big\|\widehat{{S}_{q-1}j}\ast \widehat{{\Delta}_qB} \Big\|_{L^2_t}\Big\|_{L^2(\RR^2)}\\
\leq&C\Big(\sum_{q=-1}^{\infty}2^{-2q}\big\|\widehat{{S}_{q-1}j}\big\|^2_{L^2_tL^1(\RR^2)}\Big)^{\frac12}\Big(\sum_{q=-1}^{\infty}
\big\|\widehat{{\Delta}_qB}\big\|^2_{L^\infty_tL^2(\RR^2)}\Big)^{\frac12}.
\end{split}
\end{equation}
Note that
\begin{align*}
\sum_{q=-1}^{\infty}2^{-2q}\big\|\widehat{{S}_{q-1}j}\big\|^2_{L^2_tL^1(\RR^2)}\leq& C\sum_{q=-1}^{\infty}2^{-2q}\big\|\widehat{{\Delta}_{q}j}\big\|^2_{L^2_tL^1(\RR^2)}\\
\leq& C\sum_{q=-1}^{\infty}\big\|\widehat{{\Delta}_{q}j}\big\|^2_{L^2_tL^2(\RR^2)}.
\end{align*}
Plugging this estimate in \eqref{L1-I} and using the Plancherel theorem, we see that
\begin{equation*}
\begin{split}
I\leq&C\Big(\sum_{q=-1}^{\infty}\big\|\widehat{{\Delta}_{q}j}\big\|^2_{L^2_tL^2(\RR^2)}\Big)^{\frac12}\Big(\sum_{q=-1}^{\infty}
\big\|\widehat{{\Delta}_qB}\big\|^2_{L^\infty_tL^2(\RR^2)}\Big)^{\frac12}\\
\leq&C\|j\|_{L^2_tL^2(\RR^2)}\|B\|_{\widetilde{L}^\infty_t{B}^0_{2,2}(\RR^2)}\leq C\|j\|_{L^2_tL^2(\RR^2)}\|B\|_{\widetilde{L}^\infty_tL^2_{\rm log}(\RR^2)}.
\end{split}
\end{equation*}
In the similar way, one has
\begin{equation*}
II\leq C\|j\|_{L^2_tL^2(\RR^2)}\|B\|_{\widetilde{L}^\infty_tL^2_{\rm log}(\RR^2)}.
\end{equation*}
It remains for us to bound the term $III.$ We bound it as follows:
\begin{equation}\label{III-1}
\begin{split}
III\leq&C\left(\int_0^t\left\|\int_0^se^{-(s-\tau)|\xi|^2} \widehat{R(j^\ell,B^m)} (\tau)\intd\tau\right\|_{L^1(\RR^2)}^2\intd s\right)^{\frac12}\\
 \leq&C\sum_{p=-1}^{\infty}\left(\int_0^t\left\|\int_0^s\psi_{p}e^{-(s-\tau)|\xi|^2} \widehat{R(j^\ell,B^m)} (\tau)\intd\tau\right\|_{L^1(\RR^2)}^2\intd s\right)^{\frac12}\\
 \leq&C\sum_{p=-1}^{\infty}\left(\int_0^t\left\|\int_0^s\psi_{p}e^{-c2^{2p}(s-\tau)} \widehat{R(j^\ell,B^m)} (\tau)\intd\tau\right\|_{L^1(\RR^2)}^2\intd s\right)^{\frac12}\\
  \leq&C\sum_{p=-1}^{\infty}\left(\int_0^t\left |\int_0^s2^{2p}e^{-c2^{2p}(s-\tau)} \big\|\psi_{p}\widehat{R(j^\ell,B^m)}\big\|_{L^\infty(\RR^2)} (\tau)\intd\tau\right |^2\intd s\right)^{\frac12}.
\end{split}
\end{equation}
By the Young inequality and the H\"older inequality, one has
\begin{align*}
&\sum_{p=-1}^{\infty}\left(\int_0^t\left |\int_0^s2^{2p}e^{-c2^{2p}(s-\tau)} \big\|\psi_{p}\widehat{R(j^\ell,B^m)}\big\|_{L^\infty(\RR^2)} (\tau)\intd\tau\right |^2\intd s\right)^{\frac12}\\
\leq&\sum_{p=-1}^{\infty}\left(\int_0^t \big\|\psi_{p}\widehat{R(j^\ell,B^m)}\big\|^{2}_{L^\infty(\RR^2)} (\tau) \intd \tau\right)^{\frac12}\\
\leq&\sum_{p=-1}^{\infty}\sum_{k\geq p-5}\big\|\widehat{{\Delta}_kj}\big\|_{L^2_tL^2(\RR^2)}\big\|\widehat{{\Delta}_kB}\big\|_{L^\infty_tL^2(\RR^2)}
\end{align*}
Inserting this estimate into \eqref{III-1} and using Fubini theorem, we readily have
\begin{align*}
III\leq&C\sum_{p=-1}^{\infty}\sum_{k\geq p-5}\big\|\widehat{ {\Delta}_kj}\big\|_{L^2_tL^2(\RR^2)}\big\|\widehat{ {\Delta}_kB}\big\|_{L^\infty_tL^2(\RR^2)}\\
=&C\sum_{k=-1}^{\infty}\sum_{-1\leq p\leq k+5}\big\|\widehat{ {\Delta}_kj}\big\|_{L^2_tL^2(\RR^2)}\big\|\widehat{ {\Delta}_kB}\big\|_{L^\infty_tL^2(\RR^2)}\\
=&C\sum_{k=-1}^{\infty}(k+1)\big\|\widehat{ {\Delta}_kj}\big\|_{L^2_tL^2(\RR^2)}\big\|\widehat{ {\Delta}_kB}\big\|_{L^\infty_tL^2(\RR^2)}\\
\leq&C\|j\|_{L^2_tL^2_{\rm log}(\RR^2)}\|B\|_{\widetilde{L}^\infty_tL^2_{\rm log}(\RR^2)}.
\end{align*}
Collecting estimates for $I, II, III$, we end the proof of this proposition.
\end{proof}
\section{Proof of main results}\label{Proof}
\setcounter{section}{4}\setcounter{equation}{0}
In this section, we are going to show the main theorems. Let us begin with the uniqueness of solution.

 \subsection{Uniqueness}
 This subsection is devoted to prove the uniqueness of solutions established in our theorems. To do this, it suffices to show the following proposition.
 \begin{proposition}\label{prop.uni}
 Let  $ E,B,\widetilde{E},\widetilde{B}\in C_{\rm b}([0,T];\,L^2_{\rm log}(\RR^2))$, $j,\widetilde{j}\in L^2([0,T];\,L^2_{\rm log}(\RR^2))$ and  $ u,\widetilde{u}\in C_{\rm b}([0,T];\,L^2(\RR^2))\cap L^2([0,T];\,\dot{H}^1(\RR^2))$ satisfying
 \[ \int_0^T\big\|(u,\widetilde{u})(\tau)\big\|^2_{L^\infty(\RR^2)}\intd\tau<\infty.\]
 Assume that  $(u,E,B,p)$ and $(\widetilde{u},\widetilde{E},\widetilde{B},\widetilde{p})$ be two solutions of system \eqref{eq.MNS} associated with the same initial data. Then $(u,E,B,p)\equiv(\widetilde{u},\widetilde{E},\widetilde{B},\widetilde{p})$ on interval $[0,T].$
 \end{proposition}
 \begin{proof}
Letting $(\delta u,\delta E,\delta B,\delta p):=(u-\tilde{u},E-\tilde{E},B-\tilde{B},\pi-\tilde{\pi})$, then we easily find that the difference  $(\delta u,\delta E,\delta B,\delta p)$ satisfies
\begin{equation}\label{eq.diff}
\left\{\begin{array}{ll}
\partial_t\delta u+(u\cdot\nabla)\delta u-\Delta \delta u+\nabla\delta \pi= j\times \delta B+\delta j\times \tilde{B}-(\delta u\cdot\nabla)\tilde{u}\\
\partial_t\delta E-\mathrm{curl}\,\delta B=-\delta j,\\
\partial_t\delta B+\mathrm{curl}\,\delta E=0\\
\mathrm{div}\,\delta u=\mathrm{div}\,\delta B=0,
\end{array}\right.
\end{equation}
where $\delta j=\delta E+\delta u\times B+\tilde{u}\times\delta B.$ It corresponds to the following initial condition
\begin{equation*}
(\delta u,\delta E,\delta B)|_{t=0}=(0,0,0).
\end{equation*}
Taking the standard $L^2$-estimate of $\delta u$ yields
\begin{equation*}
\begin{split}
&\frac12\dtd\big\|\delta u(t)\big\|_{L^2(\RR^2)}^2+\big\|\nabla \delta u(t)\big\|^2_{L^2(\RR^2)}\\
=&\int_{\mathbb{R}^2} (\delta j\times B)\cdot\delta u\intd x+\int_{\mathbb{R}^2} (j\times \delta B)\cdot\delta u\intd x-\int_{\mathbb{R}^2} (\delta u\cdot\nabla)\tilde{u}\cdot\delta u\intd x.
\end{split}
\end{equation*}
By the H\"older inequality and the interpolation theorem, we see that
\begin{equation*}
\begin{split}
-\int_{\mathbb{R}^2} (\delta u\cdot\nabla) \tilde{u}\cdot\delta u\intd x
\leq& \int_{\mathbb{R}^2}|\delta u|^2|\nabla\tilde{u}|\intd x\\
\leq& \|\nabla\tilde{u}\|_{L^2(\RR^2)}\|\delta u\|_{L^4(\RR^2)}^2\\
\leq&C\|\nabla\tilde{u}\|_{L^2(\RR^2)}\|\delta u\|_{L^2(\RR^2)}\|\nabla\delta u\|_{L^2(\RR^2)}\\
\leq&C\|\nabla\tilde{u}\|^2_{L^2(\RR^2)}\|\delta u\|^2_{L^2(\RR^2)}+\frac14 \|\nabla\delta u\|^2_{L^2(\RR^2)}.
\end{split}
\end{equation*}
By the same argument as in the proof of Proposition \ref{prop-strong}, we have
\begin{equation*}
\begin{split}
&\int_{\mathbb{R}^2} (\delta j\times B)\cdot\delta u\intd x+\int_{\mathbb{R}^2} (j\times \delta B)\cdot\delta u\intd x\\
\leq&C\big(\|(B,j)(t)\|^2_{L^2_{\rm log}(\RR^2)}+\|u(t)\|^2_{L^\infty(\RR^2)}\big)\big(\|\delta B(t)\|^2_{L^2_{\rm log}(\RR^2)}+\|\delta E(t)\|^2_{L^2_{\rm log}(\RR^2)}\big)\\
&+\frac14 \|\nabla\delta u\|^2_{L^2(\RR^2)}.
\end{split}
\end{equation*}
Collecting the above estimates, we readily have
\begin{equation}\label{eq.deltau}
\begin{split}
&\big\| \delta u(t)\big\|_{L^2(\RR^2)}^2+\int_0^t\big\|\nabla \delta u(s)\big\|^2_{L^2(\RR^2)}\intd s\\
\leq&\int_0^tC\big(\|(B,j)(\tau)\|^2_{L^2_{\rm log}(\RR^2)}+\|u(\tau)\|^2_{L^\infty(\RR^2)}\big)\big(\|\delta B(\tau)\|^2_{L^2_{\rm log}(\RR^2)}+\|\delta E(\tau)\|^2_{L^2_{\rm log}(\RR^2)}\big)\intd\tau\\
&+C\int_0^t\|\nabla\tilde{u}(\tau)\|^2_{L^2(\RR^2)}\|\delta u(\tau)\|^2_{L^2(\RR^2)}\intd\tau.
\end{split}
\end{equation}
Taking $L^2_{\rm log}$-norm of $(\delta E,\delta B)$ and integrating the resulting equality with respect to time $t$, we obtain
\begin{equation*}
\begin{split}
&\|\delta E(t)\|^2_{ L^2_{\rm log}(\RR^2)}+\|\delta B(t)\|^2_{L^2_{\rm log}(\RR^2)}+\int_0^t\|\delta E(\tau)\|^2_{L^2_{\rm log}(\RR^2)}\intd\tau\\
\leq&\sum_{q=-\infty}^0\int_0^t\int_{\RR^2}\big|\dot{\Delta}_q(\delta u\times B)\dot{\Delta}_q\delta E\big|\,\mathrm{d}x\mathrm{d}\tau+\sum_{q=-\infty}^0\int_0^t\int_{\RR^2}\big|\dot{\Delta}_q( u\times\delta B)\dot{\Delta}_q\delta E\big|\,\mathrm{d}x\mathrm{d}\tau\\
&+\sum_{q=1}^{\infty}q\int_0^t\int_{\RR^2}\big|\dot{\Delta}_q(\delta u\times B)\dot{\Delta}_q\delta E\big|\,\mathrm{d}x\mathrm{d}\tau+\sum_{q=1}^\infty q\int_0^t\int_{\RR^2}\big|\dot{\Delta}_q( u\times\delta B)\dot{\Delta}_q\delta E\big|\,\mathrm{d}x\mathrm{d}\tau.
\end{split}
\end{equation*}
By the same argument as in the proof of Proposition \ref{prop-strong}, we have
\begin{equation*}
\begin{split}
&\sum_{q=-\infty}^0\int_0^t\int_{\RR^2}\big|\dot{\Delta}_q(\delta u\times B)\dot{\Delta}_q\delta E\big|\,\mathrm{d}x\mathrm{d}\tau+\sum_{q=-\infty}^0\int_0^t\int_{\RR^2}\big|\dot{\Delta}_q( u\times\delta B)\dot{\Delta}_q\delta E\big|\,\mathrm{d}x\mathrm{d}\tau\\
&+\sum_{q=1}^{\infty}q\int_0^t\int_{\RR^2}\big|\dot{\Delta}_q(\delta u\times B)\dot{\Delta}_q\delta E\big|\,\mathrm{d}x\mathrm{d}\tau+\sum_{q=1}^\infty q\int_0^t\int_{\RR^2}\big|\dot{\Delta}_q( u\times\delta B)\dot{\Delta}_q\delta E\big|\,\mathrm{d}x\mathrm{d}\tau\\
\leq&C\int_0^t\|B(\tau)\|^{2}_{L^2_{\rm log}(\RR^2)}\|\nabla \delta u(\tau)\|_{L^2(\RR^2)}^2\intd\tau+C\int_0^t\|\widetilde{u}(\tau)\|_{L^\infty(\RR^2)}^2\|\delta B(\tau)\|_{L^2_{\rm log}(\RR^2)}^2\intd\tau\\
&+\frac14\int_0^t\|\delta E(\tau)\|^2_{L^2_{\rm log}(\RR^2)}\intd\tau.
\end{split}
\end{equation*}
Collecting all estimates of $(\delta E,\delta B)$ gives
\begin{equation*}
\begin{split}
&\|\delta E(t)\|^2_{L^2_{\rm log}(\RR^2)}+\|\delta B(t)\|^2_{L^2_{\rm log}(\RR^2)}+\int_0^t\|\delta E(\tau)\|^2_{L^2_{\rm log}(\RR^2)}\intd\tau\\
\leq&C\int_0^t\|B(\tau)\|^{2}_{L^2_{\rm log}(\RR^2)}\|\nabla \delta u(\tau)\|_{L^2(\RR^2)}^2\intd\tau+C\int_0^t\|\widetilde{u}(\tau)\|_{L^\infty(\RR^2)}^2\|\delta B(\tau)\|_{L^2_{\rm log}(\RR^2)}^2\intd\tau.
\end{split}
\end{equation*}
This estimate together with \eqref{eq.deltau} enables us to conclude that
\begin{equation*}
\begin{split}
&\big\|\delta u(t)\big\|_{L^2(\RR^2)}^2+\int_0^t\big\|\nabla \delta u(\tau)\big\|^2_{L^2(\RR^2)}\intd\tau+\big\|(\delta E,\delta B)(t)\big\|^2_{L^2_{\rm log}(\RR^2)}+\int_0^t\|\delta E(\tau)\|^2_{L^2_{\rm log}(\RR^2)}\intd\tau\\
\leq&\int_0^tC\big(\|(B,j)(\tau)\|^2_{L^2_{\rm log}(\RR^2)}+\|u(\tau)\|^2_{L^\infty(\RR^2)}\big)\big(\|\delta B(\tau)\|^2_{L^2_{\rm log}(\RR^2)}+\|\delta j(\tau)\|^2_{L^2_{\rm log}(\RR^2)}\big)\intd\tau.
\end{split}
\end{equation*}
Since $(\delta u,\delta E,\delta B)|_{t=0}=(0,0,0)$, there exists a time $t_0\in[0,T]$ such that
\[\|(\delta u,\delta E,\delta B)(t)\|_{L^2}\equiv0\,\, \text{on}\,\,t\in[0,t_0]\quad\text{and}\quad\|(\delta u,\delta E,\delta B)(t)\|_{L^2}>0\quad\text{on}\,\,(t_0,T]. \]
If $t_0= T$ then the uniqueness follows. Therefore, we assume $t_0< T$.
By the Gronwall inequality, it follows that $(\delta u,\delta E,\delta B)\equiv0 $ on $[t_0,T].$ So, we eventually get the uniqueness of solution.
\end{proof}
  \subsection{Existence}
In this subsection, we focus on the existence statement of Theorem \ref{THM-2}. To do this, we will adopt the following approximate scheme:
\begin{equation}\label{eq.app}
\left\{\begin{array}{ll}
\partial_tu^N+(u^N\cdot\nabla)u^N-\Delta u^N+\nabla p^N=j^N\times B^N\quad(t,x)\in\RR^+\times\RR^2,\\
\partial_tE^N-\mathrm{curl}\,B^N=-j^N,\\
\partial_tB^N+\mathrm{curl}\,E^N=0,\\
j^N=\sigma\left(E^N+u^N\times B^N\right),\\
\Div u^N=\Div B^N=0,\\
(u^N,E^N,B^N)|_{t=0}=(S_{N+1}u_0,S_{N+1}E_0,S_{N+1}B_0).
\end{array}\right.
\end{equation}
Since $(u_0,B_0,E_0)\in\big(L^2(\RR^3)\big)^3$, we have $(u_0,B_0,E_0)\in\cap_{s>0}\big(H^s(\RR^2)\big)^3$. From the main theorem proved in \cite{Masmoudi-10}, we know that the approximate system \eqref{eq.app} exists a unique global solution $(u^N,B^N,E^N)\in\big(C(\RR^+;\,H^s(\RR^2))\big)^3$ for all $s>0.$
Thanks to some a priori estimates established in Section \ref{Priori}, it follows from the Fatou lemma that
\begin{equation}\label{eq.uniform}
 \begin{split}
 &\|u^N(t)\|_{L^2(\RR^2)}^2+\big\|(E^N,B^N)\big\|^2_{\widetilde{L}^\infty_tL^2_{\rm log}(\RR^2)}+\int_0^t\|\nabla u^N(\tau)\|_{L^2(\RR^2)}^2\intd\tau\\
 &+\lambda_1\sup_{i\in\ZZ}2^{2i}\int_0^t\sum_{k\in\mathbb{Z}^2}\big\|\sqrt{\phi_{i,k}}u^{N}(\tau)\big\|^2_{L^2(\RR^2)}\intd\tau+\int_0^t\|j^N(\tau)\|_{L^2_{\rm log}(\RR^2)}^2\intd\tau\leq C,
 \end{split}
 \end{equation}
 where the constant $C$ does not depend on parameter $N$.

Letting $u^{M,N}=u^M-u^N, E^{M,N}=E^M-E^N$ and $B^{M,N}=B^M-B^N$, we see that the triple $(u^{M,N},E^{M,N},B^{M,N})$ solves the following system in $\RR^+\times\RR^2$:
\begin{equation}\label{eq.app-Cau}
\left\{\begin{array}{ll}
u_t^{M,N}+(u^{M}\cdot\nabla)u^{M,N}-\Delta u^{M,N}+\nabla p^{M,N}=j^M\times B^{M,N}+j^{M,N}\times B^N+(u^{M,N}\cdot\nabla)u^{N}\\
E_t^{M,N}-\mathrm{curl}\,B^{M,N}=-j^{M,N},\\
B_t^{M,N}+\mathrm{curl}\,E^{M,N}=0,\\
j^{M,N}=\sigma\left(E^{M,N}+u^{M,N}\times B^M+u^N\times B^{M,N}\right),\\
(u^{M,N},E^{M,N},B^{M,N})|_{t=0}=(S_{M+1}-S_{N+1})(u_0,E_0,B_0).
\end{array}\right.
\end{equation}
By the same argument in proving the uniqueness, we can infer that
\begin{equation*}
\begin{split}
&\big\|u^{M,N}(t)\big\|_{L^2}^2+\int_0^t\big\|\nabla u^{M,N}(\tau)\big\|^2_{L^2}\intd\tau+\sup_i2^{2i}\int_0^t\sum_{k\in\mathbb{Z}^2}\big\|\varphi_{i,k}u^{M,N}(\tau)\big\|^2_{L^2}\intd\tau\\
&+\|E^{M,N}\|^2_{\widetilde{L}^\infty_t \dot{B}^0_{2,2}}+\|B^{M,N}\|^2_{\widetilde{L}^\infty_t \dot{B}^0_{2,2}}+\int_0^t\|E^{M,N}(\tau)\|^2_{L^2}\intd\tau\\
\leq&C\int_0^t\|\nabla u^N(\tau)\|_{L^2}^2\|B^{M,N}(\tau)\|_{L^2}^2\intd\tau+C\int_0^t\big(\|j^N(\tau)\|^2_{L^2}+\|B^N(\tau)\|^2_{L^2}\big)\|B^{M,N}(\tau)\|^2_{L^2}\,\mathrm{d}\tau\\
&+C\sup_{q\in\mathbb{Z}}\int_0^t\big\|\dot{S}_qu^N(\tau)\big\|^2_{L^\infty}\,\mathrm{d}\tau\int_0^t\|E^{M,N}\tau)\|^2_{L^2}\,\mathrm{d}\tau+C\int_0^t\|\nabla{u}^N(\tau)\|^2_{L^2}\|u^{M,N}(\tau)\|^2_{L^2}\intd\tau\\
&+\big\|(u^{M,N}_0,E^{M,N}_0,B^{M,N}_0)\big\|_{L^2}^2.
\end{split}
\end{equation*}
Performing the Gronwall inequality and using the uniform estimate \eqref{eq.uniform}, we get
\begin{equation}\label{eq.Cauchy}
\begin{split}
&\big\|u^{M,N}(t)\big\|_{L^2}^2+\int_0^t\big\|\nabla u^{M,N}(\tau)\big\|^2_{L^2}\intd\tau+\sup_i2^{2i}\int_0^t\sum_{k\in\mathbb{Z}^2}\big\|\varphi_{i,k}u^{M,N}(\tau)\big\|^2_{L^2}\intd\tau\\
&+\|E^{M,N}\|^2_{\widetilde{L}^\infty_t \dot{B}^0_{2,2}}+\|B^{M,N}\|^2_{\widetilde{L}^\infty_t \dot{B}^0_{2,2}}+\int_0^t\|E^{M,N}(\tau)\|^2_{L^2}\intd\tau\\
\leq&\big\|(u^{M,N}_0,E^{M,N}_0,B^{M,N}_0)\big\|_{L^2}^2e^{Ct.}
\end{split}
\end{equation}
This implies that $\{(u^N,B^N,E^N)\}_{N=1}^\infty$ is a Cauchy sequence in the Banach space
\begin{equation*}
X:=\big(L^\infty([0,T];\,L^2)\cap L^2([0,T];\,\dot{H}^1)\big)\times L^\infty([0,T];\,L^2)\times L^\infty([0,T];\,L^2).
\end{equation*}
 Therefore, there exists a strong limit $(u,E,B)$ such that
 \begin{equation}\label{eq.sconver-u}
 u^N\to u \in L^\infty([0,T];\,L^2)\cap L^2([0,T];\,\dot{H}^1)\quad\text{as}\quad N\to\infty;
 \end{equation}
 \begin{equation}\label{eq.sconver-E}
 E^N\to E \in L^\infty([0,T];\,L^2) \quad\text{as}\quad N\to\infty;
 \end{equation}
  \begin{equation}\label{eq.sconver-B}
 B^N\to B \in L^\infty([0,T];\,L^2) \quad\text{as}\quad N\to\infty.
 \end{equation}
 Next, we want to show that
 \begin{equation}\label{eq.sconver-j}
 j^N\to j \in L^2([0,T];\,L^2) \quad\text{as}\quad N\to\infty.
 \end{equation}
 Note that $j^N=\sigma\left(E^N+u^N\times B^N\right)$. Since \eqref{eq.sconver-E} holds, we just need to show that
 \begin{equation*}
u^N\times B^N\to u\times B \in L^2([0,T];\,L^2) \quad\text{as}\quad N\to\infty.
 \end{equation*}
One writes
 \begin{equation*}
 u^N\times B^N-u\times B= (u^N-u)\times B^N+ u\times (B^N-B).
 \end{equation*}
 With the help of the Bony para-product decomposition and the H\"older inequality, we can show that
\begin{equation*}
\begin{split}
&\big\|(u^N-u)\times B^N\big\|_{L^2([0,T];\,L^2) }\\
&\leq C\|B^N\|_{L^\infty([0,T];\,L^2)}\Big(\|u^N-u\|_{L^2([0,T];\,\dot{H}^1)}+\sup_r\frac{1}{r^2}\int_0^t\sum_{k\in\mathbb{Z}^2}\big\|\varphi_{r,k}(u^{N}-u)(\tau)\big\|^2_{L^2}\intd\tau\Big).
\end{split}
\end{equation*}
This together with the uniform estimate \eqref{eq.uniform}, \eqref{eq.sconver-u} and estimate \eqref{eq.Cauchy} entails
\[\big\|(u^N-u)\times B^N\big\|_{L^2([0,T];\,L^2) }\to0\quad\text{as}\quad N\to\infty.\]
In the same way, we have
\[\big\|u\times (B^N-B)\big\|_{L^2([0,T];\,L^2) }\to0\quad\text{as}\quad N\to\infty.\]
Hence, we have the required convergence \eqref{eq.sconver-j}. The main task is now to show that $(u,E,B)$ is a solution of system \eqref{eq.MNS} in the sense of distribution. Let the vector $\omega\in\mathcal{S}(\RR^2)$ satisfying $\Div\omega=0$, and $\vartheta(t)\in\mathcal{D}([0,T))$. Then, we have
\begin{equation*}
\begin{split}
&\langle u^N(0),\omega\rangle\vartheta(0)+\int_0^T\langle u^N(t),\omega\rangle\vartheta(t)\intd t\\
&+\int_0^T\langle u^N,(u^N\cdot\nabla)\omega\rangle\theta(t)\intd t
 +\int_0^T\langle j^N\times B^N,\omega\rangle\theta(t)\intd t=0;
 \end{split}
 \end{equation*}
\begin{equation*}
\langle E^N(0),\omega\rangle\vartheta(0)+\int_0^T\langle B^N(t),\mathrm{curl}\,\omega\rangle\vartheta(t)\intd t=\int_0^T\langle j^N,\omega\rangle\vartheta(t)\intd t;
\end{equation*}
\begin{equation*}
\langle B^N(0),\omega\rangle\vartheta(0)-\int_0^T\langle E^N(t),\mathrm{curl}\,\omega\rangle\vartheta(t)\intd t=0,
\end{equation*}
where $\langle,\rangle$ denotes the standard $L^2$-inner product.

For the linear term, it is easy to show that, as $N\to\infty,$
\[\langle u^N(0),\omega\rangle\vartheta(0)+\int_0^T\langle u^N(t),\omega\rangle\vartheta(t)\intd t\to \langle u(0),\omega\rangle\vartheta(0)+\int_0^T\langle u(t),\omega\rangle\vartheta(t)\intd t;\]
\[
\langle E^N(0),\omega\rangle\vartheta(0)+\int_0^T\langle B^N(t),\mathrm{curl}\,\omega\rangle\vartheta(t)\intd t\to \langle E(0),\omega\rangle\vartheta(0)+\int_0^T\langle B(t),\mathrm{curl}\,\omega\rangle\vartheta(t)\intd t;\]
\[\langle B^N(0),\omega\rangle\vartheta(0)-\int_0^T\langle E^N(t),\mathrm{curl}\,\omega\rangle\vartheta(t)\intd t\to \langle B(0),\omega\rangle\vartheta(0)-\int_0^T\langle E(t),\mathrm{curl}\,\omega\rangle\vartheta(t)\intd t \]
and
\begin{equation*}
\int_0^T\langle j^N,\omega\rangle\vartheta(t)\intd t\to \int_0^T\langle j,\omega\rangle\vartheta(t)\intd t.
\end{equation*}
So, we need to show that as $N\to\infty,$
\begin{equation*}
\begin{split}
&\int_0^T\langle u^N,(u^N\cdot\nabla)\omega\rangle\theta(t)\intd t
 +\int_0^T\langle j^N\times B^N,\omega\rangle\theta(t)\intd t\\
 \to&\int_0^T\langle u,(u\cdot\nabla)\omega\rangle\theta(t)\intd t
 +\int_0^T\langle j\times B,\omega\rangle\theta(t)\intd t. \end{split}
\end{equation*}
A simple calculation yields
\begin{equation*}
\begin{split}
&\int_0^T\langle u^N,(u^N\cdot\nabla)\omega\rangle\theta(t)\intd t-\int_0^T\langle u,(u\cdot\nabla)\omega\rangle\theta(t)\intd t\\
=&\int_0^T\langle u^N- u,(u\cdot\nabla)\omega\rangle\theta(t)\intd t+\int_0^T\langle u,((u-u^N)\cdot\nabla)\omega\rangle\vartheta(t)\intd t.
\end{split}
\end{equation*}
By the H\"older inequality, one has
\begin{equation*}
\begin{split}
&\int_0^T\langle u^N- u,(u\cdot\nabla)\omega\rangle\vartheta(t)\intd t\\
\leq&\|\vartheta\|_{L^1(\RR^+)}\|\nabla\omega\|_{L^\infty(\RR^2)}\big\|u^N- u\big\|_{L^\infty([0,T];\;L^2(\RR^2))}\|u\|_{L^\infty([0,T];\;L^2(\RR^2))}.
\end{split}
\end{equation*}
This combined with the uniform estimate \eqref{eq.uniform} and \eqref{eq.sconver-u} leads to
\[\int_0^T\langle u^N- u,(u\cdot\nabla)\omega\rangle\vartheta(t)\intd t\to 0\quad\text{as}\quad N\to\infty.\]
Performing the same argument, we can obtain
\[\int_0^T\langle u,((u-u^N)\cdot\nabla)\omega\rangle\vartheta(t)\intd t\to0\quad\text{as}\quad N\to\infty.\]
Thus, we have
\[\int_0^T\langle u^N,(u^N\cdot\nabla)\omega\rangle\theta(t)\intd t\to\int_0^T\langle u,(u\cdot\nabla)\omega\rangle\theta(t)\intd t\quad\text{as}\quad N\to\infty.\]
Similarly, we have
\[\int_0^T\langle j^N\times B^N,\omega\rangle\theta(t)\intd t\to\int_0^T\langle j\times B,\omega\rangle\theta(t)\intd t\quad\text{as}\quad N\to\infty.\]
From above, we show that $(u,E,B)$ is a distributional solution of system \eqref{eq.MNS}.

Now, we begin to show the time continuity of solution. Since $u\in \widetilde{L}^\infty(\RR^+;\dot{B}^0_{2,2}(\RR^2))$, there exists a positive integer $N$ such that
\begin{equation}\label{eq.high-small}
\sum_{k\geq N}\|\Delta_k u\|^2_{L^\infty(\RR^+;\,L^2(\RR^2))}<\frac\varepsilon2.
\end{equation}
For all $t_1,\,t_2\in\RR^+$, we assume $t_2>t_1$ without lose of generality. By computations, one has
\begin{equation}\label{eq.meanvale}
\big\|S_N\big(u(t_2)-u(t_1)\big)\big\|_{L^2}\leq\int_{t_1}^{t_2}\|S_N\partial_tu(\tau)\|_{L^2}\intd\tau.
\end{equation}
Recall that
\[\partial_tu=\Delta u-\mathbb{P}\big((u\cdot\nabla) u+j\times B\big).\]
It follows form the Bernstein inequality that
\begin{equation*}
\begin{split}
\|S_N\partial_tu\|_{L^2(\RR^+;\,L^2)}\leq&\|S_N\Delta u\|_{L^2(\RR^+;\,L^2)}+\big\|\mathbb{P}S_N\big((u\cdot\nabla) u+j\times B\big)\big\|_{L^2(\RR^+;\,L^2)}\\
\leq&C2^{N}\|\nabla u\|_{L^2(\RR^+;\,L^2)}+C\big\|S_N\big((u\cdot\nabla) u+j\times B\big)\big\|_{L^2(\RR^+;\,L^2)}\\
\leq&C2^{N}\|\nabla u\|_{L^2(\RR^+;\,L^2)}+C2^{N}\big\|S_N\big((u\cdot\nabla) u+j\times B\big)\big\|_{L^2(\RR^+;\,L^1)}\\
\leq&C2^{N}\|\nabla u\|_{L^2(\RR^+;\,L^2)}+C2^{N}\|\nabla u\|_{L^2(\RR^+;\,L^2)}\|u\|_{L^\infty(\RR^+;\,L^2)}\\
&+C2^{N}\|j\|_{L^2(\RR^+;\,L^2)}\|B\|_{L^\infty(\RR^+;\,L^2)}<\infty.
\end{split}
\end{equation*}
Inserting this estimate into \eqref{eq.meanvale} leads to
\begin{equation*}
\big\|S_N\big(u(t_2)-u(t_1)\big)\big\|_{L^2}\leq C\|S_N\partial_tu\|_{L^2(\RR^+;\,L^2)}(t_2-t_1)^{\frac12}.
\end{equation*}
According to the low-high decomposition technique and \eqref{eq.high-small}, we obtain
\begin{equation*}
\begin{split}
\big\|\big(u(t_2)-u(t_1)\big)\big\|_{L^2}\leq&\big\|S_N\big(u(t_2)-u(t_1)\big)\big\|_{L^2}+2\sum_{k\geq N}\|\Delta_k u\|^2_{L^\infty(\RR^+;\,L^2(\RR^2))}\\
\leq& C(t_2-t_1)^{\frac12}+\varepsilon.
\end{split}
\end{equation*}
This implies $u(t)\in C(\RR^+;\,L^2(\RR^2))$. In the same way as used for $u$, we can obtain that  $E(t)\in C(\RR^+;\,L_{\rm log}^2(\RR^2))$ and  $B(t)\in C(\RR^+;\,L_{\rm log}^2(\RR^2))$.

Now, we begin to show the existence statement in Theorem \ref{THM-2}. By the compact argument, we know that system \eqref{eq.MNS} admits a unique global-in-time solution $(u,B,E)$. By Proposition \ref{prop-strong} and Corollary \ref{coro-L2LI}, we get from Fatou's lemma that $\|(E,B)\|_{\widetilde{L}^\infty_t L^2_{\rm\log}(\RR^2)}\leq C(t)$ and $u\in L^2_{\rm loc}(\RR^+;\;L^\infty(\RR^2)).$  Thus, we finish the proof of our theorems.
\begin{center}{\bf Appendix }\end{center}

\appendix
 \setcounter{section}{5}\setcounter{theorem}{0}\setcounter{equation}{0}
 In this appendix, we will give a lemma and a proposition which have been used in Section \ref{Priori}.
 \begin{lemma}[\cite{BCD11}, Lemma 2.35]\label{lemma.2.35}
For any positive $s$, there holds that
\[\sup_{t>0}\sum_{j\in\mathbb{Z}}t^{s}2^{2js}e^{-ct2^{2j}}<\infty.\]
\end{lemma}
\begin{proposition}\label{prop.negative}
Let $s$ be a positive real number and $(p,r)\in[1,\infty]^2$. Then there exists a constant $C>0$ such that
\begin{equation*}
C^{-1}\|u\|_{F\dot{B}^{-2s}_{p,r}}\leq\left\|t^{s}\big\|e^{-t|\xi|^2}\hat{u}\big\|_{L^p}\right\|_{L^r(\RR^+;\,\frac{\mathrm{d}t}{t})}\leq C \|u\|_{F\dot{B}^{-2s}_{p,r}}\quad\text{for all}\;\;u\in \mathcal{S}'_h.
\end{equation*}
\end{proposition}
\begin{proof}
According to the support property of $\varphi_j$, we see that
\[\big\|t^s\varphi_je^{-t|\xi|^2}\hat{u}\big\|_{L^p}\leq Ct^{s}2^{2js}e^{-ct2^{2j}}2^{-2js}\big\|\widehat{\dot{\Delta}_ju}\big\|_{L^p}.\]
Using the fact that $u\in\mathcal{S}'_h$ and the definition of the homogeneous Fourier-Herz spaces, we have
\begin{equation*}
\begin{split}
t^{s}\big\|e^{-t|\xi|^2}\hat{u}\big\|_{L^p}\leq& \sum_{j\in\mathbb{Z}}\big\|t^s\varphi_je^{-t|\xi|^2}\hat{u}\big\|_{L^p}\\
\leq&C\|u\|_{F\dot{B}^{-2s}_{p,r}}\sum_{j\in\mathbb{Z}}t^{s}2^{2js}e^{-ct2^{2j}}c_{r,j},
\end{split}
\end{equation*}
where $\|c_{r,j}\|_{\ell^r}=1$.

If $r=\infty$, then the inequality readily follows from Lemma \ref{lemma.2.35}.

If $r<\infty$, then using the H\"older inequality and Lemma \ref{lemma.2.35}, we obtain
\begin{equation*}
\begin{split}
\int_0^\infty t^{rs}\big\|e^{-t|\xi|^2}\hat{u}\big\|^r_{L^p}\frac{\mathrm{d}t}{t}
\leq&C\|u\|^r_{F\dot{B}^{-2s}_{p,r}}\int_0^\infty\Big(\sum_{j\in\mathbb{Z}}t^{s}2^{2js}e^{-ct2^{2j}}c_{r,j}\Big)^r\,\frac{\mathrm{d}t}{t}\\
\leq&C\|u\|^r_{F\dot{B}^{-2s}_{p,r}}\int_0^\infty\Big(\sum_{j\in\mathbb{Z}}t^{s}2^{2js}e^{-ct2^{2j}}\Big)^{r-1}\Big(\sum_{j\in\mathbb{Z}}t^{s}2^{2js}e^{-ct2^{2j}}c^r_{r,j}\Big)\,\frac{\mathrm{d}t}{t}\\
\leq&C\|u\|^r_{F\dot{B}^{-2s}_{p,r}}\int_0^\infty \sum_{j\in\mathbb{Z}}t^{s}2^{2js}e^{-ct2^{2j}}c^r_{r,j} \,\frac{\mathrm{d}t}{t}.
\end{split}
\end{equation*}
Using Fubini's theorem, one infers that
\begin{equation*}
\begin{split}
\int_0^\infty t^{rs}\big\|e^{-t|\xi|^2}\hat{u}\big\|^r_{L^p}\frac{\mathrm{d}t}{t}
\leq& C\|u\|^r_{F\dot{B}^{-2s}_{p,r}}\sum_{j\in\mathbb{Z}}c^r_{r,j}\int_0^\infty t^{s}2^{2js}e^{-ct2^{2j}} \,\frac{\mathrm{d}t}{t}\\
\leq&C\Gamma(s)\|u\|^r_{F\dot{B}^{-2s}_{p,r}}\quad\text{with}\quad\Gamma(s)=\int_0^\infty t^{s-1}e^{-t}\,\mathrm{d}t.
\end{split}
\end{equation*}
To prove the other inequality, we use the following identity
\begin{equation*}
\widehat{\dot{\Delta}_ju}=\frac{1}{\Gamma(s+1)}\int_0^\infty t^{s}|\xi|^{s+1}e^{-t|\xi|^2}\widehat{\dot{\Delta}_ju}\,\mathrm{d}t
\end{equation*}
Since $e^{-t|\xi|^2}=e^{-\frac{t}{2}|\xi|^2}e^{-\frac{t}{2}|\xi|^2}$, we have
\begin{equation*}
\begin{split}
\big\|\widehat{\dot{\Delta}_ju}\big\|_{L^p}\leq &C\int_0^\infty t^{s}2^{2j(s+1)}e^{-ct2^{2j}}\Big\|\widehat{\dot{\Delta}_je^{-\frac{t}{2}\Delta}u}\Big\|_{L^p}\,\mathrm{d}t\\
\leq &C\int_0^\infty t^{s}2^{2j(s+1)}e^{-ct2^{2j}}\big\|\widehat{e^{-t\Delta}u}\big\|_{L^p}\,\mathrm{d}t.
\end{split}
\end{equation*}
If $r=\infty$, then we have
\begin{equation*}
\begin{split}
\big\|\widehat{\dot{\Delta}_ju}\big\|_{L^p}\leq & C\bigg(\sup_{t>0}t^{s}\big\|\widehat{e^{-t\Delta}u}\big\|_{L^p}\bigg)\int_0^\infty2^{2j(s+1)}e^{-ct2^{2j}}\,\mathrm{d}t\\
\leq&C2^{2js} \bigg(\sup_{t>0}t^{s}\big\|\widehat{e^{-t\Delta}u}\big\|_{L^p}\bigg).
\end{split}
\end{equation*}
If $r<\infty$, we write
\[\sum_{j\in\mathbb{Z}}2^{-2jsr}\big\|\widehat{\dot{\Delta}_ju}\big\|_{L^p}^r\leq C\sum_{j\in\mathbb{Z}}2^{2jr}\bigg(\int_0^\infty t^se^{-ct2^{2j}}\big\|\widehat{e^{-t\Delta}u}\big\|_{L^p}\,\mathrm{d}t\bigg)^r.\]
The H\"older inequality with the weight $e^{-ct2^{2j}}$ implies that
\begin{equation*}
\begin{split}
 \bigg(\int_0^\infty t^se^{-ct2^{2j}}\big\|\widehat{e^{-t\Delta}u}\big\|_{L^p}\,\mathrm{d}t\bigg)^r
\leq&\bigg(\int_0^\infty e^{-ct2^{2j}}\,\mathrm{d}t\bigg)^{r-1}\int_0^\infty t^{rs}e^{-ct2^{2j}}\big\|\widehat{e^{-t\Delta}u}\big\|^r_{L^p}\,\mathrm{d}t \\
\leq&C2^{-2j(r-1)}\int_0^\infty t^{rs}e^{-ct2^{2j}}\big\|\widehat{e^{-t\Delta}u}\big\|^r_{L^p}\,\mathrm{d}t.
\end{split}
\end{equation*}
By resorting to Lemma \ref{lemma.2.35} and the Fubini theorem, we readily get
\begin{equation*}
\begin{split}
\sum_{j\in\mathbb{Z}}2^{-2jsr}\big\|\widehat{\dot{\Delta}_ju}\big\|_{L^p}^r\leq&C\sum_{j\in\mathbb{Z}}2^{-2j}\int_0^\infty t^{rs}e^{-ct2^{2j}}\big\|\widehat{e^{-t\Delta}u}\big\|^r_{L^p}\,\mathrm{d}t\\
\leq&C\int_0^\infty\bigg(\sum_{j\in\mathbb{Z}}t2^{-2j}e^{-ct2^{2j}}\bigg)t^{rs}\big\|\widehat{e^{-t\Delta}u}\big\|^r_{L^p}\,\frac{\mathrm{d}t}{t}\\
\leq&C\int_0^\infty t^{rs}\big\|\widehat{e^{-t\Delta}u}\big\|^r_{L^p}\,\frac{\mathrm{d}t}{t}.
\end{split}
\end{equation*}
The proposition is thus proved.
\end{proof}

\end{document}